\documentclass{article}
\usepackage{amsmath,amssymb,amsthm,graphicx,amsxtra}
\usepackage[hypertex]{hyperref}

\usepackage[all]{xy}

\newtheorem{thm}{Theorem}[section]
\newtheorem{prop}[thm]{Proposition}
\newtheorem{lem}[thm]{Lemma}
\newtheorem{cor}[thm]{Corollary}
\newtheorem{CT}[thm]{Chiarellotto-Tsuzuki conjecture}
\newtheorem{propdfn}[thm]{Proposition-Definition}
\newtheorem{lemdfn}[thm]{Lemma-Definition}

\theoremstyle{definition}
\newtheorem{dfn}[thm]{Definition}
\newtheorem*{convention}{Convention}
\newtheorem{ass}[thm]{Assumption}
\newtheorem{rem}[thm]{Remark}
\newtheorem{ex}[thm]{Example}
\newtheorem*{notation}{Notation}
\newtheorem{construction}[thm]{Construction}

\newtheorem{theorem}{Theorem}

\newcommand{\Fil}{\mathrm{Fil}}
\newcommand{\Sol}{\mathrm{Sol}}
\newcommand{\bd}{\mathrm{bd}}
\newcommand{\Int}{\mathrm{int}}
\newcommand{\NP}{\mathrm{NP}}
\newcommand{\SP}{\mathrm{sp}}
\newcommand{\Ext}{\mathrm{Ext}}
\newcommand{\ev}{\mathrm{ev}}
\newcommand{\Hom}{\mathrm{Hom}}
\newcommand{\id}{\mathrm{id}}
\newcommand{\GL}{\mathrm{GL}}
\newcommand{\rank}{\mathrm{rank}}
\newcommand{\diag}{\mathrm{Diag}}
\newcommand{\M}{\mathrm{M}}
\newcommand{\inc}{\mathrm{inc}}
\newcommand{\pr}{\mathrm{pr}}

\numberwithin{equation}{thm}

\begin{document}
\title{Logarithmic growth filtrations for $(\varphi,\nabla)$-modules over the bounded Robba ring}
\author{Shun Ohkubo
\footnote{
Graduate School of Mathematics, Nagoya University, Furocho, Chikusaku, Nagoya 464-8602, Japan. E-mail address: shun.ohkubo@gmail.com 2010 Mathematics Subject Classification. Primary 12H25; secondary 14G22. Keywords: $p$-adic differential equations, logarithmic growth, Picard-Fuchs equations.}
}
\maketitle

\begin{abstract}
In this paper, we study the logarithmic growth (log-growth) filtration, a mysterious invariant found by B. Dwork, for $(\varphi,\nabla)$-modules over the bounded Robba ring. The main result is a proof of a conjecture proposed by B. Chiarellotto and N. Tsuzuki on a comparison between the log-growth filtration and Frobenius slope filtration. One of the ingredients of the proof is a new criterion for {\it pure of bounded quotient}, which is a notion introduced by Chiarellotto and Tsuzuki to formulate their conjecture. We also give several applications to log-growth Newton polygons, including a conjecture of Dwork on the semicontinuity, and an analogue of a theorem due to V. Drinfeld and K. Kedlaya on Frobenius Newton polygons for indecomposable convergent $F$-isocrystals.
\end{abstract}

\tableofcontents
\section*{Introduction}

Picard-Fuchs modules arise from proper, flat, generically smooth morphisms $f:X\to\mathbb{P}^1_{\mathbb{C}}$ (see \cite[Definition 22.1.1]{pde}). Even the case that $f$ is given by a family of Calabi-Yau varieties is important in a wide area such as mirror symmetry (for example, computations of Yukawa couplings in \cite{Mor}), arithmetic geometry (for example, the proof of Sato-Tate conjecture in \cite{HST}).

When $\mathbb{C}$ is replaced by a field $k$ of characteristic $p>0$, under some assumptions, there exists an analogous construction of Picard-Fuchs modules using the rigid cohomology (see \cite[Definition 22.2.1]{pde}). The resulting Picard-Fuchs modules are of characteristic $0$. One feature that distinguishes the case of $k$ from the case of $\mathbb{C}$ is the existence of Frobenius structures, which are some kind of self-similarity. The aim of this paper is to study the category of $(\varphi,\nabla)$-modules over the {\it bounded Robba ring} $\mathcal{R}^{\bd}$ (sometimes denoted by $\mathcal{E}^{\dagger}$ or $\mathbb{B}^{\dagger}$), which includes Picard-Fuchs modules after suitable localizations, by measuring the {\it logarithmic growth} (log-growth) of solutions: it is a mysterious invariant discovered by B. Dwork about half a century ago (\cite{Dw}), and is brought an attention by B. Chiarellotto and N. Tsuzuki in \cite{CT,CT2}. The main result of this paper is a proof of the fundamental conjecture in Chiarellotto-Tsuzuki theory (Theorem \ref{thm:intro1}).

We introduce notation and terminology to describe our result. Let $K$ be a complete discrete valuation fields of characteristic $(0,p)$, $\mathcal{O}_K$ the integer ring of $K$, $\mathfrak{m}_K$ the maximal ideal of $\mathcal{O}_K$, and $k$ the residue field of $\mathcal{O}_K$. Let $|\cdot|$ denote the norm on $K$ normalized by $|p|=p^{-1}$. For a non-negative real number $\lambda$, a power series $\sum_{i\in\mathbb{Z}}a_it^i\in K[\![t]\!],a_i\in K$ is of {\it log-growth} $\lambda$ if the $a_i$'s satisfy the growth condition $|a_i|=O(i^{\lambda})$ as $i\to+\infty$. The {\it $\lambda$-th log-growth filtration} $K[\![t]\!]_{\lambda}$ is the $K$-vector space consisting of power series over $K$ of log-growth $\lambda$. The log-growth filtration is an increasing filtration, and $K[\![t]\!]_0$ coincides with the ring $\mathcal{O}_K[\![t]\!]\otimes_{\mathcal{O}_K}K$, which can be regarded as the set of bounded analytic functions on the open unit disc $0\le |t|<1$. Let $\varphi:K[\![t]\!]_0\to K[\![t]\!]_0$ be a {\it $q$-power Frobenius lift} with $q$ a positive power of $p$, that is, a ring endomorphism of the form $\varphi(\sum_ia_it^i)=\sum_i\varphi_K(a_i)\varphi(t)^i$, where $\varphi_K:K\to K$ is an isometric ring endomorphism inducing the $q$-power map on $k$, and $\varphi(t)\in t^q+\mathfrak{m}_K\mathcal{O}_K[\![t]\!]$.

Given a ring $R$ equipped with a ring endomorphism $\varphi$, we define a {\it $\varphi$-module} over $R$ as a finite free module equipped with a $R$-linear automorphism $\varphi_M^*:M\otimes_{R,\varphi}R\to M$. Then $\varphi_M^*$ can be viewed as a $\varphi$-semi-linear map $\varphi_M:M\to M$.

Let $d:K[\![t]\!]\to K[\![t]\!]dt;f\mapsto df/dt\cdot dt$ denote the canonical derivation. We define a $(\varphi,\nabla)$-module over $K[\![t]\!]_0$ as a $\varphi$-module over $K[\![t]\!]_0$ plus a connection $\nabla_M:M\to Mdt$, that is, an additive map satisfying the Leibniz rule $\nabla_M(cv)=c\nabla_M(v)+vdc$ for $c\in K[\![t]\!]_0$ and $v\in M$, that makes the following diagram commute:
\[\xymatrix{
M\ar[r]^{\nabla_M}\ar[d]^{\varphi_M}&Mdt\ar[d]^{\varphi_Md\varphi(t)}\\
M\ar[r]^{\nabla_M}&Mdt,
}\]
where the right vertical morphism sends $mdt$ to $\varphi_M(m)d(\varphi(t))$. A {\it solution} of $M$ is a $K[\![t]\!]_0$-linear map $f:M\to K[\![t]\!]$ making the following diagram commute:
\[\xymatrix{
M\ar[d]^{f}\ar[r]^{\nabla_M}&Mdt\ar[d]^{fdt}\\
K[\![t]\!]\ar[r]^{d}&K[\![t]\!]dt.
}\]
Let $\Sol(M)$ denote the $K$-vector space of solutions of $M$. Then $\Sol(M)$ has a natural structure of $\varphi$-module over $K$, in particular, is endowed with Frobenius slope filtration $S_{\bullet}(\Sol(M))$. We say that $f$ is {\it of log-growth} $\lambda$ for $\lambda\ge 0$ if $f(M)\subset K[\![t]\!]_{\lambda}$. Let $\Sol_{\lambda}(M)$ denote the $K$-vector space of solutions of log-growth $\lambda$, and put $\Sol_{\lambda}(M)=0$ for $\lambda<0$. Thus we obtain the growth filtration $\Sol_{\bullet}(M)$.

The idea of Chiarellotto-Tsuzuki conjecture is to compare the two filtrations $\Sol_{\bullet}(M)$ and $S_{\bullet}(\Sol(M))$ up to a specified shifting. To formulate their conjecture precisely, they introduce the notion of {\it pure of bounded quotient} (PBQ for short) as follows. Let $\mathcal{E}$ denote the {\it Amice ring}, which is defined as the fraction fields of the $p$-adic completion of $\mathcal{O}_K[\![t]\!][1/t]$. We put $M_{\mathcal{E}}=M\otimes_{K[\![t]\!]_0}\mathcal{E}$, which is the ``generic fiber'' of $M$. We define the ring homomorphism $\tau:\mathcal{E}\to\mathcal{E}[\![X-t]\!]_0;f\mapsto\sum_{n=0}^{\infty}(d^nf/dt^n)(X-t)^n/n!$ with the new variable $X-t$ (``Taylor expansion'' at Dwork generic point $t$). Then $\tau^*M_{\mathcal{E}}$ is a $(\varphi,\nabla)$-module over $\mathcal{E}[\![X-t]\!]_0$, and by repeating the construction in the previous paragraph over $\mathcal{E}[\![X-t]\!]_0$, we obtain the $\varphi$-module $\Sol(\tau^*M_{\mathcal{E}})$ together with the filtration $\Sol_{\bullet}(\tau^*M_{\mathcal{E}})$. By a theorem of Robba, there exists a unique $(\varphi,\nabla)$-submodule $M_{\mathcal{E}}^{\lambda}$ of $M_{\mathcal{E}}$ satisfying $\Sol(\tau^*(M_{\mathcal{E}}/M_{\mathcal{E}}^{\lambda}))\cong\Sol_{\lambda}(\tau^*M_{\mathcal{E}})$. We say that $M$ is {\it pure of bounded quotient} if the {\it bounded quotient} $M_{\mathcal{E}}/M_{\mathcal{E}}^0$ is pure as a $\varphi$-module over $\mathcal{E}$.

Then Chiarellotto-Tsuzuki conjecture is stated as follows, which will be proved as a generalized form to the bounded Robba ring $\mathcal{R}^{\bd}$ in Theorem \ref{conj:bd} (ii).

\begin{theorem}[{Chiarellotto-Tsuzuki conjecture \ref{conj:CT} (ii)}]\label{thm:intro1}
Let $M$ be a $(\varphi,\nabla)$-module over $K[\![t]\!]_0$, and $\lambda_{\max}$ the maximum Frobenius slope of $M_{\mathcal{E}}$. If $M$ is pure of bounded quotient, then
\[
\Sol_{\lambda}(M)=S_{\lambda-\lambda_{\max}}(\Sol(M))
\]
for an arbitrary real number $\lambda$.
\end{theorem}

Precisely speaking, Chiarellotto-Tsuzuki conjecture has a weaker form (Conjecture \ref{conj:CT} (i)) which has been proved by the author (\cite{Ohk}), and has an analogous form over $\mathcal{E}$ which has been proved by Chiarellotto and Tsuzuki (see Theorem \ref{conj:CTgen}). Alternative proofs of these variants of Chiarellotto-Tsuzuki conjecture using Theorem \ref{thm:intro1} will be given in part II.

We should mention that one may study the log-growth filtration without assuming Frobenius structures as Dwork originally does (\cite{Dw}). On this theory, we have a comprehensive lecture note by G. Christol (\cite{Chr}), and there are some recent developments such as \cite{And}, \cite[\S 18]{pde}, and \cite{Ohkc}. In this paper, we concentrate to study the log-growth filtrations for $(\varphi,\nabla)$-modules (over various base rings).

\subsection*{Structure of part I}

This paper is consisting of two parts. In part I, we give a proof of Theorem \ref{thm:intro1} by proving its generalization over $\mathcal{R}^{\bd}$ (Theorem \ref{conj:bd} (ii)). Part II is devoted to applications. Throughout this paper, it is written in a reasonably self-contained manner so that the topic of this paper is easily accessible to non-experts.

We explain the structure of part I in the following. We will explain the structure of part II later.

In \S \ref{sec:Robba}, we recall the construction of various analytic rings, such as the extended (bounded) Robba ring $\tilde{\mathcal{R}}^{(\bd)}$, due to K. Kedlaya in \cite{rel}. We also construct a log version of $\tilde{\mathcal{R}}$, which does not appear in \cite{rel}.

In \S \ref{sec:lg}, we equip those rings in \S 1 with the log-growth filtrations, and study its basic properties.

In this paper, we consider $(\varphi,\nabla)$-modules over various rings such as $K[\![t]\!]_0,\mathcal{R}^{\bd}$, and $\mathcal{E}$. In \S \ref{sec:phinabla}, we give a unified framework that all cases can be treated at a time.

In \S \ref{sec:slope}, we recall the definition of Frobenius slope filtration on a $\varphi$-module over $K$ due to Tsuzuki (\cite{Tsu}): it is indexed by $\mathbb{R}$ as the log-growth filtration. We restate some results on difference modules in \cite{pde}.

In \S \ref{sec:logfil}, we define the sets of analytic horizontal sections $V(M)$, and the solution space $\Sol(M)$ for $(\varphi,\nabla)$-modules $M$ over $K[\![t]\!]_0,\mathcal{R}^{\bd}$, or $\mathcal{E}$, and endow them with growth filtrations.

In \S \ref{sec:ex}, we give (explicit) examples of $(\varphi,\nabla)$-modules, and calculate their log-growth filtrations and Frobenius slope filtrations. Some of them will be useful as counterexamples.

In \S \ref{sec:statement}, we state Chiarellotto-Tsuzuki conjecture \ref{conj:CT} over $K[\![t]\!]_0$, and an analogue over $\mathcal{E}$ (Theorem \ref{conj:CTgen}). We also state the main theorem in this paper (Theorem \ref{conj:bd} (ii)), which can be seen as a generalization of Chiarellotto-Tsuzuki conjecture over $\mathcal{R}^{\bd}$.

In \S \ref{sec:criterion}, we prove a key ingredient of the proof of Theorem \ref{thm:intro1} called Slope criterion (Proposition \ref{prop:key}), which asserts that if a $(\varphi,\nabla)$-module $M$ over $\mathcal{E}$ is PBQ, then the maximum Frobenius of slope of $M_{\mathcal{E}}$ does not change under quotients.

In \S \ref{sec:rev}, we recall the construction of the reverse filtration over the extended bounded Robba ring $\tilde{\mathcal{R}}^{\bd}$ due to R. Liu (\cite{Liu}). As a consequence, we obtain the existence of a certain eigenvector over $\tilde{\mathcal{R}}^{\bd}$ (Proposition \ref{prop:maximum eigen}), which is another ingredient of the proof of Theorem \ref{thm:intro1}.

In \S \ref{sec:pf}, we prove Theorem \ref{thm:intro1} with the strategy explained in the next subsection.

\subsection*{Strategy of proof}

We outline the proof of Theorem \ref{thm:intro1}. Let notation and assumption be as in Theorem \ref{thm:intro1}. For simplicity, we assume that $k$ is algebraically closed.

Step 1 (preliminary reduction to Proposition \ref{prop:pf}): By Proposition \ref{prop:CT1}, it suffices to prove $\Sol_{\lambda}(M)\subset S_{\lambda-\lambda_{\max}}(\Sol(M))$. Thanks to Dieudonn\'e-Manin theorem (Lemma \ref{lem:DM}), it reduces to prove (the following special case of) Proposition \ref{prop:pf}, which asserts that if a non-zero solution $f\in \Sol_{\lambda}(M)$ satisfies $\varphi^d(f)=q^af$ (that is, $f$ is a {\it Frobenius $d$-eigenvector} of $\Sol_{\lambda}(M)$), then we have $\mu\le\lambda-\lambda_{\max}$, where $\mu=a/d$ denotes the {\it Frobenius slope} of $f$.

Step 2 (key reduction to the case that $f$ is injective): By a typical argument using pushforward and twist for Frobenius structures (note that both operations do no change differential structures), we may assume $\varphi(f)=f$, where our goal is to prove $\lambda_{\max}\le\lambda$. Since $f:M\to K[\![t]\!]$ is $\varphi,\nabla$-equivariant by assumption, $N=\ker{f}$ is a $(\varphi,\nabla)$-submodule of $M$. For simplicity, let $f$ also denote the solution of $M/N$ induced by $f$. Then we still have $f\in \Sol_{\lambda}(M/N)$ and $\varphi(f)=f$. By Slope criterion (Proposition \ref{prop:key}), the maximum Frobenius slope of $(M/N)_{\mathcal{E}}$ is equal to $\lambda_{\max}$. Therefore, after replacing $M$ by $M/N$, we may further assume that $f:M\to K[\![t]\!]$ is injective.

Step 3 (base change argument): We consider {\it the extended (bounded) Robba ring} $\tilde{\mathcal{R}}^{(\bd)}$ together with a $\varphi$-equivariant embedding $K[\![t]\!]_0\to\tilde{\mathcal{R}}^{\bd}$ (Definition \ref{dfn:extended Robba}, Proposition \ref{prop:embedding}). Then we can extend $f$ to an injective, $\varphi$-equivariant map
\[
\tilde{f}:M\otimes_{K[\![t]\!]_0}\tilde{\mathcal{R}}^{\bd}\hookrightarrow\tilde{\mathcal{R}}.
\]
The $\lambda$-th log-growth filtration $K[\![t]\!]_{\lambda}$ naturally extends to a filtration $\Fil_{\lambda}\tilde{\mathcal{R}}$ by measuring the growth of Gauss norms (Definition \ref{dfn:log-growth extended}), and we can show $\tilde{f}(M\otimes_{K[\![t]\!]_0}\tilde{\mathcal{R}}^{\bd})\subset\Fil_{\lambda}\tilde{\mathcal{R}}$ by using the assumption $f\in \Sol_{\lambda}(M)$. The advantage of performing the base change $K[\![t]\!]_0\to\tilde{\mathcal{R}}^{\bd}$ is that one can use the {\it reverse filtration}. As a result, there exists a Frobenius $d$-eigenvector $\tilde{v}\in M\otimes_{K[\![t]\!]_0}\tilde{\mathcal{R}}^{\bd}$ of slope $\lambda_{\max}$ (Proposition \ref{prop:maximum eigen}). Then $\tilde{f}(\tilde{v})\in\tilde{\mathcal{R}}$ is also a Frobenius $d$-eigenvector of slope of $\lambda_{\max}$ since $\tilde{f}$ is $\varphi$-equivariant and injective. We conclude $\lambda_{\max}\le\lambda$ by the fact that any Frobenius $d$-eigenvector $w\in\tilde{\mathcal{R}}$ of slope $\mu$ is {\it exactly of log-growth} $\mu$ (Lemma \ref{lem:calc log-growth}).

We should remark that Step 3 is analogous to the proof of the following theorem due to A. J. de Jong (\cite[9.1]{dJ}): assume that $K$ is absolutely unramified and $\varphi(t)=t^p$. Let $M$ be a $(\varphi,\nabla)$-module over $\mathcal{O}_K[\![t]\!]$ in the sense of \cite{dJ}. Let $f:M\to\mathcal{O}_{\mathcal{E}}$ be a horizontal $\mathcal{O}_K[\![t]\!]$-linear map such that $\varphi(f)=p^{-l}f$ for some $l\in\mathbb{N}$. Then $f(M)\subset\mathcal{O}_K[\![t]\!]$. In fact, the author regards de Jong's theorem as a ``generic'' result, and proves Proposition \ref{prop:pf} as its ``special'' analogue. Note that the differential structure on $M$ is used only in Steps 1, 2. In Step 3, we need only the Frobenius structure on $M$, therefore, we may perform the base change $K[\![t]\!]_0\to\tilde{\mathcal{R}}^{\bd}$.

\subsection*{Structure of part II}

We give a summary of results in part II, restricted to the case of $K[\![t]\!]_0$ instead of $\mathcal{R}^{\bd}$ for simplicity.

Theorem \ref{thm:intro1} enables us to compare the log-growth filtration with Frobenius slope filtration. It is natural to ask that which properties the log-growth filtration and Frobenius slope filtration share. A result of this kind is Theorem \ref{conj:CT} (i), that is, the right continuity and the rationality of the slopes of the log-growth filtration. We will prove {\it Dwork's conjecture} below, which is an analogue of Grothendieck-Katz specialization theorem for Frobenius Newton polygons (\cite[Theorem 15.3.2]{pde}).

\begin{theorem}[{Corollary \ref{cor:LGFDW}}]
The log-growth Newton polygons for any $(\varphi,\nabla)$-module over $K[\![t]\!]_0$ is semicontinuous under specialization.
\end{theorem}

Frobenius slope filtration is compatible with tensor products and duals, however, a na\"ive analogue for the log-growth filtration fails: one can find counterexamples in \S \ref{sec:ex}. Instead, we obtain ``weak'' compatibility results as follows.

\begin{theorem}[{Proposition \ref{prop:weak tensor}}]
Let $M,N$ be $(\varphi,\nabla)$-modules over $K[\![t]\!]_0$. If $\lambda$ and $\mu$ are slopes of the log-growth Newton polygons of $M$ and $N$ respectively, then $\lambda+\mu$ is also a slope of the log-growth Newton polygon of $M\otimes_{K[\![t]\!]_0}N$.
\end{theorem}

\begin{theorem}[{Theorem \ref{thm:inv}}]\label{thm:invintro}
Let $M$ be a $(\varphi,\nabla)$-module over $\mathcal{E}$, and $M\spcheck$ its  dual. Then the maximum slopes of the log-growth Newton polygons of $M$ and $M\spcheck$ coincide.
\end{theorem}

V. Drinfeld and Kedlaya prove that for an indecomposable convergent $F$-isocrystal $M$ of rank $n$ on a smooth irreducible quasi-compact variety $X$ over $k$ (assuming that $k$ is perfect), the Frobenius slopes $a_i^{\eta}(M)$ for $1\le i\le n$ at the generic point $\eta$ of $X$ satisfies $a_i^{\eta}(M)-a_{i+1}^{\eta}(M)\le 1$ for all $i\in\{0,\dots,n-1\}$ (\cite[Theorem 1.1.5]{DK}). We have the following (local) analogue for the log-growth filtration.

\begin{theorem}[{Theorem \ref{thm:Gri}}]\label{thm:Griintro}
Let $M$ be a $(\varphi,\nabla)$-module over $\mathcal{E}$ of rank $n$. Let $\lambda_1(M)\le\dots\le\lambda_n(M)$ denote the slopes of the log-growth Newton polygon of $M$ with multiplicity. Then $\lambda_{i+1}(M)-\lambda_{i}(M)\le 1$ for all $i\in\{0,\dots,n-1\}$. 
\end{theorem}

We also have the following remarks on Chiarellotto-Tsuzuki conjecture.

\begin{enumerate}
\item[---] One can ask the necessity of PBQ hypothesis in Theorem \ref{thm:intro1}. In \S \ref{sec:conv}, we answer affirmatively this question by proving the converse, that is, if the equality in Theorem \ref{thm:intro1} holds for all $\lambda$, then $M$ is PBQ. We also prove similar results for $\mathcal{R}^{\bd}$ and $\mathcal{E}$.
\item[---] In \S \ref{sec:CTgen}, we prove that Theorem \ref{thm:intro1} implies an analogue of Chiarellotto-Tsuzuki conjecture over $\mathcal{E}$ (Theorem \ref{conj:CTgen}).
\end{enumerate}

We explain some technical aspects of part II. We would like to study the log-growth filtration of an arbitrary $(\varphi,\nabla)$-module by exploiting Theorem \ref{thm:intro1} somehow. To achieve this, we have two methods. One is to use the {\it PBQ filtration}, which is constructed over $K[\![t]\!]_0$ and $\mathcal{E}$ in \cite{CT2}, and is generalized over $\mathcal{R}^{\bd}$ in Definition \ref{dfn:PBQ fil}. Another method is to use {\it Generating theorem} (Theorem \ref{thm:generating}), a novelty of this paper, that asserts that any $(\varphi,\nabla)$-module over $\mathcal{E}$ is generated by PBQ $(\varphi,\nabla)$-submodules. We will use Generating theorem in an essential way to prove Theorems \ref{thm:invintro} and \ref{thm:Griintro}.

\subsection*{Notation and terminology}
\begin{enumerate}
\item[(1)] We recall some terminology on difference rings in \cite[\S 14]{pde}. Let $F$ be a field equipped with a ring endomorphism $\phi:F\to F$. We call the pair $(F,\phi)$, or, $F$ for simplicity, a {\it difference field}. We say that $F$ is {\it inversive} if $\phi$ is an automorphism. We say that $F$ is {\it weakly difference-closed} if any equation of the form $\phi(x)=cx$ with $c\in F^{\times}$ has always a solution $x\in F$. We say that $F$ is {\it strongly difference-closed} if $F$ is weakly difference-closed and inversive.
\item[(2)] Let $p$ be a prime number. Let $K$ be a complete discrete valuation field of mixed characteristic $(0,p)$, $\mathcal{O}_K$ the integer ring of $K$, $\mathfrak{m}_K$ the maximal ideal of $K$, $k$ the residue field of $K$. Let $|\cdot|$ denote a valuation of $K$. We put the normalization condition by $|p|=p^{-1}$ unless otherwise is mentioned. Let $\varphi_K$ be an isometric ring endomorphism of $K$. Let $\varphi_k:k\to k$ denote the ring endomorphism on $k$ induced by $\varphi_K$. Let $q$ denote a positive power of $p$. We say that $\varphi_K$ is a {\it $q$-power Frobenius lift} if $\varphi_k$ is the $q$-power map. Unless otherwise is mentioned, precisely speaking, except \S\S \ref{sec:Robba}, \ref{sec:lg}, and a part of \S \ref{sec:max}, we assume that $\varphi_K$ is a $q$-power Frobenius lift. The reader who are interested only in the proof of Theorem \ref{thm:intro1} may assume that $\varphi_K$ is always a $q$-power Frobenius lift. Recall that there exists an extension $L/K$ of discrete valuation fields, on which $\varphi_K$ extends isometrically to $\varphi_L:L\to L$, such that $(l,\varphi_l)$ is inversive, where $l$ denotes the residue field of $L$, and $\varphi_l$ denotes the endomorphism on $l$ induced by $\varphi_L$: for example, the completion of $\varinjlim (K\to K\to\dots)$ with transition maps $\varphi_K$ (the {\it $\varphi_K$-completion of $K$}). We may even assume that $(l,\varphi_l)$ is strongly-difference closed (\cite[Proposition 3.2.4]{rel}). 
\item[(3)] For a ring homomorphism $R\to S$ and an $R$-module $M$, let $M_S$ denote the $S$-module $M\otimes_RS$. For a ring $R$, let $\M_n(R)$ denote the set of $n\times n$-matrices of $R$, $\GL_n(R)$ the set of matrices admitting an inverse in $\M_n(R)$. For $a_1,\dots,a_n\in R$, let $\diag(a_1,\dots,a_n)$ denote the diagonal matrix whose $(i,i)$-component is equal to $a_i$. 
\item[(4)] Let $V$ be an finite dimensional vector space over a field $F$. Let $\{V^{\lambda};\lambda\in\mathbb{R}\}$ be a decreasing, separated, and exhaustive filtration by subspaces of $V$; we give a similar definition for an increasing filtration. We say that $V^{\bullet}$ is {\it right continuous} if $V^{\lambda}=\cap_{\mu>\lambda}V^{\mu}$ for all $\lambda$. We put $m({\lambda})=\dim_F(\cap_{\varepsilon>0}V^{\lambda-\varepsilon})/(\cup_{\varepsilon>0}V^{\lambda+\varepsilon})=\lim_{\varepsilon\to 0+}\dim_FV^{\lambda-\varepsilon}-\lim_{\varepsilon\to 0+}\dim_FV^{\lambda+\varepsilon}$. Then $m({\lambda})=0$ for all but finitely many $\lambda$. If $m({\lambda})\neq 0$, then we call $\lambda$ a {\it slope} of $V^{\bullet}$. Let $\lambda_1<\dots<\lambda_d$ be the slopes of $V^{\bullet}$ (without multiplicity). We define the {\it Newton polygon} $\NP(V^{\bullet})$ of $V^{\bullet}$ as the lower convex hull in the $xy$-plane of the set of points $(0,0)$ and $(m(\lambda_1)+\dots+m(\lambda_i),\lambda_1\cdot m(\lambda_1)+\dots+\lambda_i\cdot m(\lambda_i))$ for $i=1,\dots,d$. Note that
\[
V^{\lambda}=
\begin{cases}
V&\text{if }\lambda\in (-\infty,\lambda_1),\\
V^{\lambda_i}&\text{if }\lambda\in (\lambda_i,\lambda_{i+1})\text{ for }i=1,\dots,d-1,\\
0&\text{if }\lambda\in (\lambda_d,+\infty).
\end{cases}
\]
Let $\lambda_1(V^{\bullet})\le\dots\le\lambda_n(V^{\bullet})$ denote the slope multiset of $\NP(V^{\bullet})$, i.e.,
\[
\{\lambda_1(V^{\bullet}),\dots,\lambda_n(V^{\bullet})\}=\{\underbrace{\lambda_1,\dots,\lambda_1}_{m({\lambda_1})\text{ times}},\dots,\underbrace{\lambda_d,\dots,\lambda_d}_{m({\lambda_d})\text{ times}}\}.
\]
\end{enumerate}

\part{Chiarellotto-Tsuzuki conjecture}
\section{The Robba ring and the extended Robba ring}\label{sec:Robba}

In this section, we recall the definitions of the Robba ring, and the extended Robba ring constructed by Kedlaya (\cite{rel}), which is studied further by Liu (\cite{Liu}). We also define a log version of the extended Robba ring. These rings will be important tools throughout this paper. In this section, we work with notation and assumption as in \cite{rel}, in particular, we allow a relative Frobenius lift as a Frobenius lift on the Robba ring.


\subsection{The Robba ring}\label{subsec:Robba1}

\begin{dfn}\label{dfn:Robba ring}
\begin{enumerate}
\item (\cite[Definitions 1.1.1, 1.2.3]{rel}) For $r>0$, let $\mathcal{R}^r$ be the ring of rigid analytic functions on the $K$-annulus $e^{-r}\le |t|<1$, and let $\mathcal{R}$ be the union of the $\mathcal{R}^r$. The ring $\mathcal{R}$ is called the {\it Robba ring} over $K$. Explicitly, we have
\[
\mathcal{R}^r=\{\sum_{i\in\mathbb{Z}}a_it^i\in K[\![t^{-1},t]\!];a_i\in K,|a_i|e^{-sn}\to 0\ (i\to\pm\infty)\ \forall s\in (0,r]\},
\]
\[
\mathcal{R}=\cup_{r>0}\mathcal{R}^r.
\]
For $r>0$, let $|\cdot|_r$ denote the supremum norm on the circle $|t|=e^{-r}$, as applied to elements of $\mathcal{R}^s$ for $s\ge r$; one easily verifies that
\[
|\sum_{i\in\mathbb{Z}}a_it^i|_r=\sup_{i\in\mathbb{Z}}\{|a_i|e^{-ri}\}.
\]
\item ({\cite[Definition 1.1.3]{rel}}) Let $\mathcal{R}^{\Int}$ be the subring of $\mathcal{R}$ consisting of series with coefficients in $\mathcal{O}_K$; this ring is a discrete valuation ring with residue field $k((t))$, which is not complete but is henselian \cite[Lemma~3.9]{plm}. Let $\mathcal{R}^{\bd}$ be the subring of $\mathcal{R}$ consisting of series with bounded coefficients; it is the fraction field of $\mathcal{R}^{\Int}$. We call $\mathcal{R}^{\bd}$ the {\it bounded Robba ring} over $K$. Let $\mathcal{E}$ denote the fraction field of the $\mathfrak{m}_K$-adic completion of $\mathcal{R}^{\Int}$. Explicitly, we have
\[
\mathcal{R}^{\Int}=\{\sum_{i\in\mathbb{Z}}a_it^i\in\mathcal{R};a_i\in \mathcal{O}_K\},
\]
\[
\mathcal{R}^{\bd}=\{\sum_{i\in\mathbb{Z}}a_it^i\in\mathcal{R};a_i\in K,\ \sup_{i\in\mathbb{Z}}{|a_i|}<\infty\}=\mathcal{R}^{\Int}\otimes_{\mathcal{O}_K}K,
\]
\[
\mathcal{E}=\{\sum_{i\in\mathbb{Z}}a_it^i\in K[\![t^{-1},t]\!];a_i\in K,\ \sup_{i\in\mathbb{Z}}|a_i|<\infty,\ |a_i|\to 0\ (i\to-\infty)\}.
\]
Let $|\cdot|_0$ denote Gauss norm on $\mathcal{E}$, i.e.,
\[
|\sum_{i\in\mathbb{Z}}a_it^i|_0=\sup_{i\in\mathbb{Z}}|a_i|.
\]
\end{enumerate}
\end{dfn}


\begin{dfn}[{\cite[Definition 1.2.1]{rel}}]
Fix an integer $q>1$. A {\it relative ($q$-power) Frobenius lift} on the Robba ring is a homomorphism $\varphi:\mathcal{R}\to\mathcal{R}$ of the form $\sum_ic_it^i\mapsto \sum_i\varphi_K(c_i)s^i$, where $s\in\mathcal{R}^{\bd}$ satisfies $|s-t^q|_0<1$. We define an {\it absolute ($q$-power) Frobenius lift} as a relative Frobenius lift in which $\varphi_K$ is itself a $q$-power Frobenius lift.
\end{dfn}

\subsection{The extended Robba ring}\label{subsec:Robba2}

One regards the Robba ring $\mathcal{R}$ as an analytic ring corresponding to the field of Laurent series $k((t))$. In this subsection, we recall Kedlaya's construction of the extended Robba ring $\tilde{\mathcal{R}}$, which is an analytic ring corresponding to the field of Hahn series $k((t^{\mathbb{Q}}))$. We also gather technical results on $\tilde{\mathcal{R}}$.

\begin{ass}[{\cite[Hypothesis 2.1.1]{rel}}]\label{ass:relative}
Throughout this subsection, assume that $\varphi$ is a relative Frobenius lift on $\mathcal{R}$ such that $\varphi_K$ is an automorphism of $K$. Also assume that any \'etale $\varphi$-module over $K$ is trivial; this is equivalent to asking that the residue field $k$ is strongly difference closed, i.e., any $\varphi$-module over $k$ is trivial (\cite[Definition 14.3.1]{pde}). When $\varphi_K$ is a $q$-power Frobenius lift, the condition is satisfied if $k$ is algebraically closed (\cite[Proposition 14.3.4]{pde}).
\end{ass}

\begin{dfn}[{\cite[Definition 2.2.1, Notation 2.5.1]{rel}}]\label{dfn:Hahn}
Let $k((u^{\mathbb{Q}}))$ denote the field of {\it Hahn series} over $k$ with value group $\mathbb{Q}$, that is, the set of functions $f:\mathbb{Q}\to k$ with well-ordered support, with pointwise addition and multiplication given by convolution. We endow $k((u^{\mathbb{Q}}))$ with the automorphism
\[
\varphi:\sum c_iu^i\mapsto\sum\varphi_k(c_i)u^{qi},
\]
where $\varphi_k$ denotes the automorphism on $k$ induced by $\varphi$. Recall that $k((u^{\mathbb{Q}}))$ equipped with $\varphi$ is strongly difference-closed (\cite[Proposition 2.5.5]{rel}).
\end{dfn}

\begin{dfn}[{\cite[Definition 2.2.4]{rel}, \cite[Definition 1.4.1]{Liu}}]\label{dfn:extended Robba}
For $r>0$, let $\tilde{\mathcal{R}}^r$ be the set of formal sums $\sum_{i\in\mathbb{Q}}a_iu^i$ with $a_i\in K$, satisfying the following conditions.
\begin{enumerate}
\item[(a)] For each $c>0$, the set of $i\in\mathbb{Q}$ such that $|a_i|\ge c$ is well-ordered.
\item[(b)] We have $|a_i|e^{-ri}\to 0$ as $i\to+\infty$.
\item[(c)] We have $\sup_{i\in\mathbb{Q}}|a_i|e^{-ri}<\infty$.
\item[(d)] For all $s>0$, we have $|a_i|e^{-si}\to 0$ as $i\to-\infty$.
\end{enumerate}
Then $\tilde{\mathcal{R}}^r$ can be shown to form a ring. We call the union $\tilde{\mathcal{R}}=\tilde{\mathcal{R}}_K=\cup_r\tilde{\mathcal{R}}^r$ the {\it extended Robba ring} over $K$. Let $\tilde{\mathcal{R}}^{\bd}$ and $\tilde{\mathcal{R}}^{\Int}$ be the subrings of $\tilde{\mathcal{R}}$ consisting series with bounded and integral coefficients, respectively. We equip $\tilde{\mathcal{R}}^r$ with the norm
\[
|\sum_ia_iu^i|_r=\sup_i\{|a_i|e^{-ri}\}
\]
and $\tilde{\mathcal{R}}$ with the automorphism
\[
\varphi(\sum_ia_iu^i)=\sum_i\varphi_K(a_i)u^{qi}.
\]
Note that $|\varphi(f)|_r=|f|_{qr}$ for $f\in\tilde{\mathcal{R}}^{qr}$. Let $\tilde{\mathcal{E}}$ denote the fraction field of the $\mathfrak{m}_K$-adic completion of $\tilde{\mathcal{R}}^{\Int}$. We equip $\tilde{\mathcal{E}}$ with Gauss norm $|\cdot|_0$ defined as
\[
|\sum_ia_iu^i|_0=\sup_i|a_i|.
\]
We list some properties of the above rings (see \cite[2.2.5]{rel} for details).
\begin{enumerate}
\item[$\bullet$] The ring $\tilde{\mathcal{R}}$ is a B\'ezout domain.
\item[$\bullet$] The ring $\tilde{\mathcal{R}}^{\Int}$ is henselian discrete valuation ring and its fraction field is $\tilde{\mathcal{R}}^{\bd}$.
\item[$\bullet$] The units of $\tilde{\mathcal{R}}$ are the nonzero elements of $\tilde{\mathcal{R}}^{\bd}$.
\item[$\bullet$] The field $\tilde{\mathcal{E}}$ is a complete discrete valuation field with residue field $k((u^{\mathbb{Q}}))$.
\end{enumerate}
\end{dfn}

\begin{prop}[{\cite[Proposition 2.2.6]{rel}, \cite[Remark 1.4.10]{Liu}}]\label{prop:embedding}
There exist a $\varphi$-equivariant embedding $\psi:\mathcal{R}\hookrightarrow\tilde{\mathcal{R}}$ and $r_0>0$ such that for any $r\in (0,r_0)$, $\mathcal{R}^r$ maps to $\tilde{\mathcal{R}}^r$ preserving $|\cdot|_r$. Moreover, $\mathcal{R}^{\Int}$ (resp. $\mathcal{R}^{\bd}$) maps $\tilde{\mathcal{R}}^{\Int}$ (resp. $\tilde{\mathcal{R}}^{\bd}$) preserving $|\cdot|_0$.
\end{prop}

In the rest of this paper, for a given relative Frobenius lift $\varphi$ on $\mathcal{R}$, we fix $\psi$ and $r_0$ unless otherwise is mentioned.

\begin{prop}[{\cite[Proposition 2.5.8]{rel}}]\label{prop:matrix}
Let $A$ be an $n\times n$ matrix over $\tilde{\mathcal{R}}^{\Int}$. If $\mathbf{v}\in\tilde{\mathcal{E}}^n$ is a column vector such that $A\mathbf{v}=\varphi(\mathbf{v})$, then $\mathbf{v}\in (\tilde{\mathcal{R}}^{\bd})^n$.
\end{prop}

\begin{lem}\label{lem:inj}
The multiplication map $\tilde{\mathcal{R}}^{\bd}\otimes_{\mathcal{R}^{\bd}}\mathcal{R}\to\tilde{\mathcal{R}};x\otimes y\mapsto x\psi(y)$ is injective.
\end{lem}
\begin{proof}
By \cite[Propositions 3.2.4, 3.5.2]{rel}, there exists an extension $L/K$ of complete discrete valuation fields, on which $\varphi_K$ extends isometrically, such that the multiplication map $\tilde{\mathcal{R}}^{\bd}_L\otimes_{\mathcal{R}^{\bd}}\mathcal{R}\to\tilde{\mathcal{R}}_L;x\otimes y\mapsto x\psi(y)$ is injective, where $\tilde{\mathcal{R}}_L^{(\bd)}$ is the extended (bounded) Robba ring over $L$. In a commutative diagram
\[\xymatrix{
\tilde{\mathcal{R}}^{\bd}\otimes_{\mathcal{R}^{\bd}}\mathcal{R}\ar[d]\ar[r]&\tilde{\mathcal{R}},\ar[d]\\
\tilde{\mathcal{R}}^{\bd}_L\otimes_{\mathcal{R}^{\bd}}\mathcal{R}\ar[r]&\tilde{\mathcal{R}}_L
}\]
the morphisms other than the upper horizontal one are injective, which implies the assertion.
\end{proof}

\begin{lem}\label{lem:lg1}
Let $f=\sum_{i\in\mathbb{Q}}a_iu^i\in\tilde{\mathcal{R}},a_i\in K$. Then we have $\sup_{i\le i_0}|a_i|<\infty$ for any $i_0\in\mathbb{Q}$.
\end{lem}
\begin{proof}
We have only to prove $\sup_{i\le 0}|a_i|<\infty$, and $\sup_{0\le i\le i_0}|a_i|<\infty$ for any $i_0\ge 0$. Assume $f\in\tilde{\mathcal{R}}^r$. Then
\[
\sup_{i\le 0}|a_i|\le\sup_{i\le 0}|a_i|e^{-ri}\le |f|_r,
\]
which implies the first assertion. For $i_0\ge 0$, we have
\[
\sup_{0\le i\le i_0}|a_i|e^{-ri_0}\le\sup_{0\le i\le i_0}|a_i|e^{-ri}\le |f|_r,
\]
which implies the second assertion.
\end{proof}

\begin{lem}\label{lem:max}
\begin{enumerate}
\item For $0<r\le r'$, $\tilde{\mathcal{R}}^{r'}\subset\tilde{\mathcal{R}}^r$.

By (i), $\tilde{\mathcal{R}}^{r'}$ is endowed with a family of norms $\{|\cdot|_r;r\in (0,r']\}$.
\item (maximal modulus principle) Let $r'>0$ and $I=[r_1,r_2]\subset (0,r']$ be a closed interval. Then, for any $f\in\tilde{\mathcal{R}}^{r'}$, we have
\[
\sup_{r\in I}|f|_r=\max\{|f|_{r_1},|f|_{r_2}\},
\]
\[
\inf_{r\in I}|f|_r=\min\{|f|_{r_1},|f|_{r_2}\}.
\]
\end{enumerate}
\end{lem}
\begin{proof}
\begin{enumerate}
\item Let $f=\sum_{i\in\mathbb{Q}}a_iu^i\in\tilde{\mathcal{R}}^{r'}$ with $a_i\in K$. It suffices to prove that $f$ satisfies the conditions (b) and (c) for $r$ in Definition \ref{dfn:extended Robba}. For $i\le 0$,
\[
|a_i|e^{-ri}\le |a_i|e^{-r'i}\to 0\ (i\to+\infty).
\]
It also implies that there exists $i_0\ge 0$ such that $|a_i|e^{-ri}\le 1$ for all $i\ge i_0$. We have $\sup_{i\le i_0}|a_i|<\infty$ by Lemma \ref{lem:lg1}. Hence we have
\begin{align*}
\sup_{i\in\mathbb{Q}}|a_i|e^{-ri}&=\max\{\sup_{i\le 0}|a_i|e^{-ri},\sup_{0\le i\le i_0}|a_i|e^{-ri},\sup_{i\ge i_0}|a_i|e^{-ri}\}\\
&\le\max\{\sup_{i\le 0}|a_i|e^{-r'i},\sup_{0\le i\le i_0}|a_i|,1\}\\
&\le\max\{|f|_{r'},\sup_{0\le i\le i_0}|a_i|,1\}<\infty.
\end{align*}
\item We have only to prove
\[
\sup_{r\in I}|f|_r\le \max\{|f|_{r_1},|f|_{r_2}\},\ \inf_{r\in I}|f|_r\ge \min\{|f|_{r_1},|f|_{r_2}\}.
\]
Hence we may assume that $f$ is a monomial, in which case the assertion is obvious.
\end{enumerate}
\end{proof}

\begin{lem}\label{lem:gen1}
For $f\in\tilde{\mathcal{R}}^r$, we have $f\in\tilde{\mathcal{R}}^{\bd}$ if and only if $\sup_{s\in (0,r]}|f|_s<\infty$. Moreover, when they hold, $\lim_{s\to 0+}|f|_s=|f|_0$.
\end{lem}
\begin{proof}
Let $f\in\tilde{\mathcal{R}}^r$. Write $f=\sum_{i}a_iu^i,a_i\in K$, and put $|f|_0=\sup_{i}|a_i|$ (may take $\infty$). To prove the first assertion, it suffices to prove
\begin{equation}\label{eq:gen1}
\sup_{s\in (0,r]}|f|_s=\max\{|f|_0,|f|_r\}.
\end{equation}
Since, for any $i\in\mathbb{Q}$,
\[
\max\{|a_i|,|a_i|e^{-ir}\}=\sup_{s\in (0,r]}|a_i|e^{-is}\le\sup_{s\in (0,r]}|f|_s,
\]
we have $\max\{|f|_0,|f|_r\}\le\sup_{s\in (0,r]}|f|_s$. Since, for any $s\in (0,r]$,
\[
|f|_s=\sup_{i\in\mathbb{Q}}|a_i|e^{-si}\le\sup_{i\in\mathbb{Q}}\max\{|a_i|,|a_i|e^{-ir}\}=\max\{|f|_0,|f|_r\},
\]
we have $\sup_{s\in (0,r]}|f|_s\le\max\{|f|_0,|f|_r\}$, which implies (\ref{eq:gen1}).

We will prove the second assertion. Assume that $\sup_{s\in (0,r]}|f|_s<\infty$ holds. Put $f_+=\sum_{i\ge 0}a_iu^i,f_-=\sum_{i<0}a_iu^i$. Since $|f|_s=\max\{|f_+|_s,|f_-|_s\}$ for $s\in [0,r]$, we may assume $f=f_{\pm}$. Assume $f=f_+$. Then $|f|_s$ increases as $s$ decreases, and is bounded above by $|f|_0$. Hence $\lim_{s\to 0+}|f|_s$ exists, and is bounded above by $|f|_0$. We also have $|f|_0\le\lim_{s\to 0+}|f|_s$ by (\ref{eq:gen1}), which implies the assertion. Assume $f=f_-$. Then $|f|_s$ decreases as $s$ decreases, and is bounded below by $|f|_0$. Hence $\lim_{s\to 0+}|f|_s$ exists, and is bounded below by $|f|_0$. It suffices to prove $\lim_{s\to 0+}|f|_s\le |f|_0$. By Definition \ref{dfn:extended Robba}, there exists $i_0<0$ such that $|a_iu^i|_r<|f|_0$ for all $i<i_0$. Then we have
\begin{align*}
\max\{|f|_0,|f|_s\}=\max\{|f|_0,\sup_{i\in (-\infty,i_0)}|a_iu^i|_s,\sup_{i\in [i_0,0]}|a_iu^i|_s\}\le&\max\{|f|_0,\sup_{i\in (-\infty,i_0)}|a_iu^i|_r,\sup_{i\in [i_0,0]}|a_iu^i|_s\}\\
\le&\max\{|f|_0,|f|_0e^{-si_0}\}.
\end{align*}
By passing to the limit $s\to 0+$, we obtain the assertion.
\end{proof}


\subsection{Log extensions}\label{subsec:Robba3}

Throughout this subsection, we keep Assumption \ref{ass:relative} when we consider $\tilde{\mathcal{R}}$.

We define $\mathcal{R}_{\log}=\mathcal{R}[\log{t}]$ as the polynomial ring over $\mathcal{R}$. We extend $\varphi$ on $\mathcal{R}$ to $\mathcal{R}_{\log}$ by
\[
\varphi(\log{t})=\log{(\varphi(t)/t^q)}+q\log{t}\in\mathcal{R}\oplus\mathcal{R}\cdot\log{t}.
\]
Here we define
\[
\log{f}=\sum_{n=1}^{\infty}(-1)^{n-1}(f-1)^n/n\in\mathcal{R}^{\bd}
\]
for $f\in 1+\mathfrak{m}_K\mathcal{R}^{\Int}$ (\cite[\S 6.5]{plm}). We also define $\tilde{\mathcal{R}}_{\log}=\tilde{\mathcal{R}}[\log{u}]$ with $\varphi(u)=q\log{u}$. In this subsection, we will extend any $\varphi$-equivariant embedding $\psi:\mathcal{R}\hookrightarrow\tilde{\mathcal{R}}$ given in Proposition \ref{prop:embedding} to a $\varphi$-equivariant embedding $\mathcal{R}_{\log}\hookrightarrow\tilde{\mathcal{R}}_{\log}$.

\begin{lem}\label{lem:vanishing}
For any $b\in\mathfrak{m}_K$,
\[
\varphi-b:\tilde{\mathcal{R}}^{\bd}\to\tilde{\mathcal{R}}^{\bd}
\]
is bijective.
\end{lem}
\begin{proof}
Since $\varphi:\tilde{\mathcal{R}}^{\bd}\to\tilde{\mathcal{R}}^{\bd}$ is bijective, we have only to prove that $1-b\varphi^{-1}:\tilde{\mathcal{R}}^{\bd}\to\tilde{\mathcal{R}}^{\bd}$ is bijective. The injectivity immediately follows from $|b\varphi^{-1}(x)|_0=|b||x|_0<|x|_0$ for $x\neq 0$. Therefore, it suffices to find a solution $y\in\tilde{\mathcal{R}}^{\bd}$ of the equation $x=(1-b\varphi^{-1})y$ for any given $x\in\tilde{\mathcal{R}}^{\bd}$. Write $x=\sum_{i\in\mathbb{Q}}a_iu^i,a_i\in K$. We define the sequence $\{b_i\}_{i\in\mathbb{Q}}$ of $K$ by $b_i=\sum_{n\in\mathbb{N}}b\cdot\varphi_K^{-1}(b)\cdot\dots\cdot\varphi_K^{-n+1}(b)\varphi_K^{-n}(a_{iq^n})$, and a formal sum $y=\sum_{i\in\mathbb{Q}}b_iu^i$. Then $\{b_i\}_{i\in\mathbb{Q}}$ is a bounded sequence since we have
\begin{equation}\label{log:eq1}
|b\cdot\varphi_K^{-1}(b)\cdot\dots\cdot\varphi_K^{-n+1}(b)\varphi_K^{-n}(a_{iq^n})|\le\sup_{n\in\mathbb{N}}|b|^n|a_{iq^n}|\le\sup_{n\in\mathbb{N}}|a_{iq^n}|\le |x|_0<\infty.
\end{equation}
We fix $r>0$ such that $x\in\tilde{\mathcal{R}}^r$. It suffices to verify that $y$ satisfies the conditions (a)-(d) for $r$ in Definition \ref{dfn:extended Robba}: if this is the case, then $y$ is the desired solution.

We start by checking the condition (a). Fix $c>0$. We choose sufficiently large $m\in\mathbb{N}$ so that $|b^{m+1}|\cdot |x|_0<c/2$. Let $I$ denote the set $\{i\in\mathbb{Q}:|a_i|\ge c\}$, which is well-ordered by the condition (a) for $x$. Then $I\cup\dots\cup q^{-m}I\subset\mathbb{Q}$ is also well-ordered. If $|b_i|\ge c$, then we have $\max\{|a_i|,|a_{iq}|,\dots,|a_{iq^m}|\}\ge c$ by the inequalities
\[
|b_i|\le\sup_{n\in\mathbb{N}}|b|^n|a_{iq^n}|\le\sup\{|a_i|,|a_{iq}|,\dots,|a_{iq^m}|,|b^{m+1}|\cdot |x|_0\}.
\]
Hence we have $\{i\in\mathbb{Q};|b_i|\ge c\}\subset I\cup\dots\cup q^{-m}I$. Thus (a) is verified. By (\ref{log:eq1}), we have $|b_i|\le |x|_0$, which implies (b). Similarly, we have
\[
\sup_{i\ge 0}|b_i|e^{-ri}\le |x|_0\sup_{i\ge 0}e^{-ri}=|x|_0<\infty,
\]
\[
\sup_{i<0}|b_i|e^{-ri}\le\sup_{i<0}\sup_{n\in\mathbb{N}}|a_{iq^n}|e^{-ri}\le\sup_{i<0}\sup_{n\in\mathbb{N}}|a_{iq^n}|e^{-riq^n}\le\sup_{j<0}|a_j|e^{-rj}\le |x|_r<\infty.
\]
Thus (c) is verified. Fix $s>0$. For $\varepsilon\in\mathbb{R}_{>0}$, there exists $i_{\varepsilon}<0$ such that
\[
|a_i|e^{-si}\le\varepsilon\ \forall i\le i_{\varepsilon}
\]
by the condition (d) for $x$. For any $i\le i_{\varepsilon}$, we have, by (\ref{log:eq1}),
\[
|b_i|e^{-si}\le\sup_{n\in\mathbb{N}}|a_{iq^n}|e^{-si}\le\sup_{n\in\mathbb{N}}|a_{iq^n}|e^{-siq^n}\le\varepsilon,
\]
which implies (d).
\end{proof}

\begin{lem}\label{lem:log calculation}
\begin{enumerate}
\item There exists (a unique) $c_0\in\tilde{\mathcal{R}}^{\bd}$ such that
\[
\varphi(c_0)=qc_0+\psi(\log{(\varphi(t)/t^q)}).
\]
\item Let $c_1\in (K^{\times})^{\varphi_K=1}$ (for example, $c_1=1$) and $c_0$ as in (i). We extend $\psi:\mathcal{R}\hookrightarrow\tilde{\mathcal{R}}$ given in Proposition \ref{prop:embedding} to $\mathcal{R}_{\log}\to\tilde{\mathcal{R}}_{\log}$ by sending $\log{t}$ to $c_0+c_1\log{u}$. Then the resulting $\mathcal{R}$-algebra homomorphism is injective and $\varphi$-equivariant.
\end{enumerate}
\end{lem}
\begin{proof}
\begin{enumerate}
\item Since $\psi(\log(\varphi(t)/t))\in\tilde{\mathcal{R}}^{\bd}$ by Proposition \ref{prop:embedding}, it follows from Lemma \ref{lem:vanishing}.
\item Let $x=\sum_{i=0}^nx_i(\log{t})^i,x_i\in\mathcal{R}$. If $\psi(x)=0$, then $x_n\cdot c_1^n=0$, and hence, $x_n=0$. Thus, we obtain the injectivity. The Frobenius compatibility follows from (i).
\end{enumerate}
\end{proof}

\begin{dfn}\label{dfn:log embedding}
We fix $c_1\in (K^{\times})^{\varphi_K=1}$, and define $\psi:\mathcal{R}_{\log}\hookrightarrow\tilde{\mathcal{R}}_{\log}$ as in Lemma \ref{lem:log calculation} (ii).
\end{dfn}

\begin{lem}\label{lem:log inj}
The multiplication map
\[
\tilde{\mathcal{R}}^{\bd}\otimes_{\mathcal{R}^{\bd}}\mathcal{R}_{\log}\to\tilde{\mathcal{R}}_{\log};x\otimes y\mapsto x\psi(y)
\]
is injective.
\end{lem}
\begin{proof}
We identify $\tilde{\mathcal{R}}^{\bd}\otimes_{\mathcal{R}^{\bd}}\mathcal{R}_{\log}$ as $\tilde{\mathcal{R}}^{\bd}\otimes_{\mathcal{R}^{\bd}}\mathcal{R}[\log{t}]$, i.e., the polynomial ring over $\tilde{\mathcal{R}}^{\bd}\otimes_{\mathcal{R}^{\bd}}\mathcal{R}$ with variable $\log{t}$. If $z=\sum_{i=0}^nz_i(\log{t})^i,z_i\in \tilde{\mathcal{R}}^{\bd}\otimes_{\mathcal{R}^{\bd}}\mathcal{R}$ is zero in $\tilde{\mathcal{R}}_{\log}$, then $\psi(z_n)\cdot c_1^n=0$ in $\tilde{\mathcal{R}}_{\log}$, and hence, $z_n=0$ by Lemma \ref{lem:inj}. Thus $z=0$.
\end{proof}

\section{Logarithmic growth filtrations on Robba rings}\label{sec:lg}

In this section, we filter the Robba ring $\mathcal{R}$ and its variants $\mathcal{R}_{\log},\tilde{\mathcal{R}}_{(\log)}$ by measuring the growth of Gauss norms $|\cdot|_r$.
Throughout this section, we keep Assumption \ref{ass:relative} when we consider $\tilde{\mathcal{R}}$ or $\tilde{\mathcal{R}}_{\log}$.

\subsection{On $\tilde{\mathcal{R}}$}\label{subsec:lg1}

\begin{dfn}\label{dfn:log-growth extended}
Let $f\in\tilde{\mathcal{R}}$. Assume $f\in\tilde{\mathcal{R}}^{r'}$. We say that $f$ is {\it of log-growth} $\lambda$ for $\lambda\in\mathbb{R}_{\ge 0}$ if there exists a constant $C$ such that
\[
r^{\lambda}|f|_r\le C\ \forall r\in (0,r'].
\]
The definition does not depend on the choice of $r'$ by Lemma \ref{lem:max} (ii). We define the {\it $\lambda$-th log-growth filtration} $\Fil_{\lambda}\tilde{\mathcal{R}}$ of $\tilde{\mathcal{R}}$ as the $K$-vector space consisting of those elements of log-growth $\lambda$. For $\lambda<0$, we set $\Fil_{\lambda}\tilde{\mathcal{R}}=0$. Note that $\Fil_{\lambda}\tilde{\mathcal{R}}$ is not closed under multiplication unless $\lambda\le 0$.
\end{dfn}

One has an equivalent characterization of log-growth using Taylor expansion.

\begin{lem}[Taylor expansion criterion]\label{lem:Taylor extended}
Let $f=\sum_{i\in\mathbb{Q}}a_iu^i\in\tilde{\mathcal{R}},a_i\in K$. For $\lambda\in\mathbb{R}_{\ge 0}$, the following are equivalent.
\begin{enumerate}
\item $f\in \Fil_{\lambda}\tilde{\mathcal{R}}$.
\item $|a_i|=O(i^{\lambda})$ as $i\to+\infty$.
\end{enumerate}
\end{lem}
\begin{proof}
(ii)$\Rightarrow$(i) We choose $r'>0$ such that $f\in\tilde{\mathcal{R}}^{r'}$. By assumption, there exists $i_0>0$ such that $|a_i|\le Ci^{\lambda}$ for all $i\ge i_0$. We put $I_1=(-\infty,0),I_2=[0,i_0),I_3=[i_0,+\infty)\subset\mathbb{R}$, and $g_j=\sum_{i\in I_j}a_iu^i$ for $j=1,2,3$. Then $g_1,g_2,g_3\in\tilde{\mathcal{R}}^{r'}$, and, for all $r\in (0,r']$,
\[
|f|_r=\max\{|g_1|_r,|g_2|_r,|g_3|_r\},
\]
\[
|g_1|_r\le |g_1|_{r'}<\infty,\ |g_2|_r\le|g_2|_0<\infty
\]
by Lemma \ref{lem:lg1}. Hence it suffices to prove $r^{\lambda}|g_3|_r\to 0$ as $r\to 0+$. For $r\in (0,r']$,
\[
r^{\lambda}|g_3|_r=r^{\lambda}\sup_{i\in [i_0,+\infty)}|a_i|e^{-ri}\le C\sup_{i\ge 0}(ri)^{\lambda}e^{-ri}\le C\lambda^{\lambda} e^{-\lambda},
\]
where the last inequality follows from the fact that the function $x^{\lambda}e^{-x}$ for $x\ge 0$ achieves the maximum at $x=\lambda$ (we understand $0^0=1$ here).

\noindent (i)$\Rightarrow$(ii) Fix $r'>0$ sufficiently small so that $f\in\mathcal{R}^{r'}$. Let $C$ be a constant such that $r^{\lambda}|f|_r\le C$ for all $r\in (0,r']$. Then, for $i>0$ and $r\in (0,r']$, we have $r^{\lambda}|a_i|e^{-ri}\le r^{\lambda}|f|_r\le C$, i.e.,
\[
|a_i|\le Ci^{\lambda}\cdot e^{ri}/(ri)^{\lambda}.
\]
If $i\ge 1/r'$, then, by evaluating the above inequality at $r=1/i$, we have
\[
|a_i|\le Ci^{\lambda}\cdot e,
\]
which implies (ii).
\end{proof}

We gather basic properties on the log-growth filtration.

\begin{lem}[{cf. \cite[Lemma 4.7]{Ohk}}]\label{lem:property extended}
We have the following.
\begin{enumerate}
\item The filtration $\Fil_{\bullet}\tilde{\mathcal{R}}$ is an increasing filtration satisfying
\[
\Fil_{\lambda}\tilde{\mathcal{R}}\cdot \Fil_{\mu}\tilde{\mathcal{R}}\subset \Fil_{\lambda+\mu}\tilde{\mathcal{R}}\ \forall\lambda,\mu\in\mathbb{R}.
\]
\item For an arbitrary $\lambda\in\mathbb{R}$, we have $\varphi(\Fil_{\lambda}\tilde{\mathcal{R}})\subset \Fil_{\lambda}\tilde{\mathcal{R}}$.
\item $\Fil_{0}\tilde{\mathcal{R}}=\tilde{\mathcal{R}}^{\bd}$.
\end{enumerate}
\end{lem}
\begin{proof}
\begin{enumerate}
\item The first assertion is obvious, and the second one follows from the formula $r^{\lambda+\mu}|fg|_r=r^{\lambda}|f|_r\cdot r^{\mu}|g|_r$ for sufficiently small $r>0$.
\item We choose $r'>0$ sufficiently small so that $f\in\tilde{\mathcal{R}}^{qr'}\cap\Fil_{\lambda}\tilde{\mathcal{R}}$. Then $r^{\lambda}|\varphi(f)|_r=q^{-\lambda}(qr)^{\lambda}|f|_{qr}$ for $r\in (0,r']$, which implies the assertion.
\item We have $\tilde{\mathcal{R}}^{\bd}\subset \Fil_0\tilde{\mathcal{R}}$ by Lemma \ref{lem:Taylor extended}. The converse follows from Lemma \ref{lem:gen1}.
\end{enumerate}
\end{proof}


\subsection{On $\mathcal{R}$}\label{subsec:lg2}

A very similar construction as in \S \ref{subsec:lg1} works on $\mathcal{R}$.

\begin{dfn}
Let $f\in\mathcal{R}$. Assume $f\in\mathcal{R}^{r'}$. We say that $f$ is {\it of log-growth} $\lambda$ for $\lambda\in\mathbb{R}_{\ge 0}$ if there exists a constant $C$ such that
\[
r^{\lambda}|f|_r\le C\ \forall r\in (0,r'].
\]
The definition does not depend on the choices of $r'$ and $\varphi$. We define the {\it $\lambda$-th log-growth filtration} $\Fil_{\lambda}\mathcal{R}$ of $\mathcal{R}$ as the $K$-vector space consisting of elements of log-growth $\lambda$. For $\lambda<0$, we set $\Fil_{\lambda}\mathcal{R}=0$.
\end{dfn}

Similar results as in \S \ref{subsec:lg1} hold.

\begin{lem}\label{lem:log-growth psi}
Let $\psi:\mathcal{R}\hookrightarrow\tilde{\mathcal{R}}$ be any $\varphi$-equivariant embedding given by Proposition \ref{prop:embedding}. Let $f\in\mathcal{R}$. For any real number $\lambda$, the following are equivalent.
\begin{enumerate}
\item $f\in \Fil_{\lambda}\mathcal{R}$.
\item $\psi(f)\in \Fil_{\lambda}\tilde{\mathcal{R}}$.
\end{enumerate}
\end{lem}
\begin{proof}
It follows from the norm compatibility of $\psi$.
\end{proof}

\begin{lem}[{Taylor expansion criterion, cf. \cite[Proposition 2.3.3]{Chr}}]\label{lem:Taylor}
Let $f=\sum_{i\in\mathbb{Z}}a_it^i\in\mathcal{R},a_i\in K$. For $\lambda\in\mathbb{R}_{\ge 0}$, the following are equivalent.
\begin{enumerate}
\item $f\in \Fil_{\lambda}\mathcal{R}$.
\item $|a_i|=O(i^{\lambda})$ as $i\to+\infty$.
\end{enumerate}
\end{lem}
\begin{proof}
A similar proof as Lemma \ref{lem:Taylor extended} works. Alternatively, we can deduce from Lemmas \ref{lem:Taylor extended} and \ref{lem:log-growth psi} by choosing the absolute $q$-power Frobenius lift $\varphi(t)=t^q$ as $\varphi$, in which case $\psi(t)=u$.
\end{proof}

Recall that $f=\sum_{i}a_it\in K[\![t]\!]$ with $a_i\in K$ is {\it of log-growth} $\lambda\ge 0$ if $|a_i|=O(i^{\lambda})$ as $i\to+\infty$, and we define $K[\![t]\!]_{\lambda}$ as the set of power series of log-growth $\lambda$. For simplicity, we put $K[\![t]\!]_{\lambda}=0$ for $\lambda<0$.

\begin{cor}\label{cor:logfilR1}
Let $f=\sum_{i\in\mathbb{Z}}a_it^i,a_i\in K$. For $\lambda\in\mathbb{R}$, $f\in K[\![t]\!]_{\lambda}$ if and only if $f\in \Fil_{\lambda}\mathcal{R}$.
\end{cor}

\begin{lem}[{\cite[Lemma 4.7]{Ohk}}]\label{lem:property Robba}
We have the following.
\begin{enumerate}
\item The filtration $\Fil_{\bullet}\mathcal{R}$ is increasing and
\[
\Fil_{\lambda}\mathcal{R}\cdot \Fil_{\mu}\mathcal{R}\subset \Fil_{\lambda+\mu}\mathcal{R}\text{ for }\lambda,\mu\in\mathbb{R}.
\]
\item For an arbitrary $\lambda\in\mathbb{R}$, we have $\varphi(\Fil_{\lambda}\mathcal{R})\subset \Fil_{\lambda}\mathcal{R}$.
\item $\Fil_{0}\mathcal{R}=\mathcal{R}^{\bd}$.
\item If $f\in\mathcal{R}$ satisfies $df/dt\in \Fil_{\lambda}\mathcal{R}$ for some $\lambda\in\mathbb{R}$, then $f\in \Fil_{\lambda+1}\mathcal{R}$.
\end{enumerate}

Similar assertions for $K[\![t]\!]_{\bullet}$ also hold by Corollary \ref{cor:logfilR1}.
\end{lem}
\begin{proof}
By Lemma \ref{lem:log-growth psi}, parts (i) and (ii) are reduced to parts (i) and (ii) in Lemma \ref{lem:property extended} respectively. Parts (iii) and (iv) are consequences of Lemma \ref{lem:Taylor}.
\end{proof}


\subsection{On $\mathcal{R}_{\log}$ and $\tilde{\mathcal{R}}_{\log}$}\label{subsec:lg3}

To define the log-growth filtrations on $\mathcal{R}_{\log}$ and $\tilde{\mathcal{R}}_{\log}$, we formally regard $\log{t}$ and $\log{u}$ respectively as elements exactly of log-growth $1$ as follows.

\begin{dfn}[{cf. \cite[Definition 4.11]{Ohk}}]
Let $R$ be either $\mathcal{R}$ or $\tilde{\mathcal{R}}$. For $\lambda\in\mathbb{R}$, we define the {\it $\lambda$-th log-growth filtration} of $R_{\log}$ by
\[
\Fil_{\lambda}R_{\log}=\oplus_{n=0}^{\infty}\Fil_{\lambda-n}R\cdot (\log{\star})^n\subset R_{\log},
\]
where $\star\in\{t,u\}$ respectively (recall that $\Fil_{\mu}R=0$ for $\mu<0$). An element $f\in R_{\log}$ is {\it of log-growth} $\lambda$ if $\lambda\in \Fil_{\lambda}R_{\log}$. Furthermore, $f\in R_{\log}$ is {\it exactly of log-growth} $\lambda$ if $\lambda\in \Fil_{\lambda}R_{\log}$ and $f\notin \Fil_{\mu}R_{\log}$ for any $\mu<\lambda$. Note that we have $\Fil_{\lambda}R_{\log}=0$ for $\lambda<0$ by definition.

We put $K\{t\}=K[\![t]\!]\cap\mathcal{R}=\{\sum_{i\in\mathbb{N}}a_it^i;a_i\in K,\ |a_i|e^{-ri}\to 0\ (i\to +\infty)\ \forall r\in (0,+\infty)\}$, and $K\{t\}_{\log}=K\{t\}[\log{t}]$, which is endowed with the {\it $\lambda$-th log-growth filtration}
\[
\Fil_{\lambda}K\{t\}_{\log}=K\{t\}_{\log}\cap\Fil_{\lambda}\mathcal{R}_{\log}=\oplus_{n=0}^{\infty}K[\![t]\!]_{\lambda-n}\cdot (\log{t})^n.
\]
Note that for $f\in K\{t\}_{\log}$, $f\in \Fil_{\lambda}K\{t\}_{\log}$ if and only if $f\in \Fil_{\lambda}\mathcal{R}_{\log}$ by Lemma \ref{lem:Taylor}.
\end{dfn}

Similar results as in \S\S \ref{subsec:lg1} and \ref{subsec:lg2} hold.

\begin{lem}\label{lem:log-growth psi log}
Let $\psi:\mathcal{R}_{\log}\hookrightarrow\tilde{\mathcal{R}}_{\log}$ be any $\varphi$-equivariant embedding given by Definition \ref{dfn:log embedding}. Let $f\in\mathcal{R}_{\log}$. For any real number $\lambda$, the following are equivalent.
\begin{enumerate}
\item $f\in \Fil_{\lambda}\mathcal{R}_{\log}$.
\item $\psi(f)\in \Fil_{\lambda}\tilde{\mathcal{R}}_{\log}$.
\end{enumerate}
\end{lem}
\begin{proof}
Write $f=\sum_{i=0}^na_i(\log{t})^i,a_i\in\mathcal{R}$. Then $\psi(f)=\sum_{i=0}^nA_i(\log{u})^i$, where
\[
A_i=\sum_{j=i}^n\binom{j}{i}c_0^{j-i}c_1^i\psi(a_j)
\]
with notation as in Definition \ref{dfn:log embedding}.

Assume $f\in \Fil_{\lambda}\mathcal{R}_{\log}$. Then $a_i\in \Fil_{\lambda-i}\mathcal{R}$. By Lemma \ref{lem:log-growth psi}, $\psi(a_i)\in \Fil_{\lambda-i}\tilde{\mathcal{R}}$. Since $c_0,c_1\in \Fil_0\tilde{\mathcal{R}}$ by Lemma \ref{lem:property extended} (iii), we have $A_i\in \Fil_{\lambda-i}\tilde{\mathcal{R}}$ by Lemma \ref{lem:property extended} (i). Hence $\psi(f)\in \Fil_{\lambda}\tilde{\mathcal{R}}_{\log}$.

Assume $\psi(f)\in \Fil_{\lambda}\tilde{\mathcal{R}}_{\log}$. Since $A_i=c_1^i\psi(a_i)+\sum_{j=i+1}^n\binom{j}{i}c_0^{j-i}c_1^i\psi(a_j)\in \Fil_{\lambda-i}\tilde{\mathcal{R}}$, we have $\psi(a_i)\in \Fil_{\lambda-i}\tilde{\mathcal{R}}$ by reverse induction on $i$. By Lemma \ref{lem:log-growth psi}, we have $a_i\in \Fil_{\lambda-i}\mathcal{R}$, and hence, $f\in \Fil_{\lambda}\mathcal{R}_{\log}$.
\end{proof}

\begin{lem}[{\cite[Lemma 4.7]{Ohk}}]\label{lem:property log}
Let $R$ be either $\mathcal{R}$ or $\tilde{\mathcal{R}}$. We have the following.
\begin{enumerate}
\item The filtration $\Fil_{\bullet}R_{\log}$ is increasing and
\[
\Fil_{\lambda}R_{\log}\cdot \Fil_{\mu}R_{\log}\subset \Fil_{\lambda+\mu}R_{\log}\text{ for }\lambda,\mu\in\mathbb{R}.
\]
\item For an arbitrary $\lambda\in\mathbb{R}$, we have $\varphi(\Fil_{\lambda}R_{\log})\subset \Fil_{\lambda}R_{\log}$.
\item $\Fil_{0}R_{\log}=R^{\bd}$.
\end{enumerate}

Similar assertions for $K\{t\}_{\log}$ hold.
\end{lem}
\begin{proof}
For $R=\mathcal{R}$ (resp. $\tilde{\mathcal{R}}$),  each assertion immadiately follows from the corresponding one in Lemma \ref{lem:property Robba} (resp. \ref{lem:property extended}).
\end{proof}

In the proof of Main Theorem (Theorem \ref{conj:bd}), we need to determine the log-growth of particular elements in $\tilde{\mathcal{R}}_{\log}$ as follows.

\begin{dfn}\label{dfn:eigen log extended}
Let $d$ be a positive integer. A {\it (Frobenius) $d$-eigenvector} of $\tilde{\mathcal{R}}_{\log}$ is a non-zero element $f$ of $\tilde{\mathcal{R}}_{\log}$ such that
\[
\varphi^d(f)=cf
\]
for some $c\in (\tilde{\mathcal{R}}^{\bd})^{\times}$. We refer to the quotient $\log{|c|_0}/\log{|q^d|}$ as the {\it (Frobenius) slope} of $f$.

Note that:
\begin{enumerate}
\item[$\bullet$] a $d$-eigenvector $f$ of slope $\lambda$ is a $d'$-eigenvector of slope $\lambda$ for any multiplier $d'>0$ of $d$;
\item[$\bullet$] if $f_i$ for $i=1,2$ is a $d_i$-eigenvector of slope $\lambda_i$, then $f_1f_2$ is a $d$-eigenvector of slope $\lambda_1+\lambda_2$ for any common multiplier $d>0$ of $d_1$ and $d_2$.
\end{enumerate}
\end{dfn}

\begin{lem}\label{lem:trivialization}
Let $a\in\tilde{\mathcal{R}}^{\bd}$ with $|a|_0=1$. Then there exists $b\in (\tilde{\mathcal{R}}^{\bd})^{\times}$ such that $\varphi(b)=ab$.
\end{lem}
\begin{proof}
Since any \'etale $\varphi$-module over $\tilde{\mathcal{E}}$ is trivial by the strongly difference-closedness of the residue field $k((u^{\mathbb{Q}}))$ (instead, by using \cite[Proposition 2.1.6]{rel}), there exists $b\in (\tilde{\mathcal{E}})^{\times}$ such that $\varphi(b)=ab$ (see the proof of \cite[Theorem 14.6.3]{pde}). By Proposition \ref{prop:matrix}, we have $b\in\tilde{\mathcal{R}}^{\bd}$ as desired.
\end{proof}

\begin{lem}[{cf. \cite[Theorem 6.1]{Ohk}}]\label{lem:calc log-growth}
If $f\in\tilde{\mathcal{R}}_{\log}$ is a $d$-eigenvector of slope $\lambda$, then $\lambda\ge 0$, and $f$ is exactly of log-growth $\lambda$.
\end{lem}
\begin{proof}
After replacing $(\varphi,q)$ by $(\varphi^d,q^d)$, we may assume $d=1$. Write $f=\sum_ix_i(\log{u})^i$ with $x_i\in\tilde{\mathcal{R}}$. Then $x_i$ is a $d$-eigenvector of slope $\lambda-i$ unless $x_i=0$. Hence we may reduce to the case $f\in\tilde{\mathcal{R}}$. Let $c\in (\tilde{\mathcal{R}}^{\bd})^{\times}$ such that $\varphi(f)=cf$ with $\lambda=\log{|c|_0}/\log{|q|}$.

$\bullet$ The case $c\in K^{\times}$

Suppose $\lambda<0$, i.e., $|c|>1$. Write $f=\sum_{i\in\mathbb{Q}}a_iu^i$ with $a_i\in K$. By $\varphi(f)=cf$, we have $a_i=\varphi_K(a_{i/q})/c$. Fix $i\in\mathbb{Q}$ and choose $n$ sufficiently large so that $i/q^n\le 1$. Then,
\[
|a_i|=|\varphi^n_K(a_{i/q^n})|/|c\cdot\varphi_K(c)\cdot\dots\cdot\varphi^{n-1}_K(c)|=|a_{i/q^n}|/|c|^n\le (\sup_{j\le 1}|a_j|)/|c|^n\to 0\ (n\to+\infty).
\]
since $\sup_{j\le 1}{|a_j|}<\infty$ by Lemma \ref{lem:lg1}. Hence we have $a_i=0$ for all $i$, which contradicts to $f\neq 0$. Thus $\lambda\ge 0$.

We fix $r'>0$ such that $f\in\tilde{\mathcal{R}}^{r'}$. To prove the second assertion, it suffices to prove the inequalities
\[
(r'/q r)^{\lambda}\inf_{s\in [r'/q,r']}|f|_s\le |f|_r\le (r'/r)^{\lambda}\sup_{s\in [r'/q,r']}|f|_s.
\]
for all $r\in (0,r']$ since $\inf_{s\in [r'/q,r']}|f|_s,\sup_{s\in [r'/q,r']}|f|_s\neq 0$ by Lemma \ref{lem:max} (ii). Let $r\in (0,r']$. We choose $n\in\mathbb{N}$ such that $r\in [r'/q^{n+1},r'/q^n]$. Since $q^nr\in [r'/q,r']$ and $\lambda\ge 0$, we obtain $\inf_{s\in [r'/q,r']}|f|_s\le |f|_{q^nr}\le \sup_{s\in [r'/q,r']}|f|_s$, and $(r'/q r)^{\lambda}\le q^{n\lambda}\le (r'/r)^{\lambda}$. Hence we obtain
\[
(r'/q r)^{\lambda}\inf_{s\in [r'/q,r']}|f|_s\le q^{n\lambda}|f|_{q^nr}\le (r'/r)^{\lambda}\sup_{s\in [r'/q,r']}|f|_s.
\]
Since $|f|_r=q^{n\lambda}|f|_{q^nr}$ by $\varphi(f)=cf$, we obtain the desired inequalities.

$\bullet$ The general case

Since $K$ is discretely valued, there exists $c'\in K^{\times}$ such that $|c'|=|c|_0$. By Lemma \ref{lem:trivialization}, we choose $b\in (\tilde{\mathcal{R}}^{\bd})^{\times}$ such that $\varphi(b)=(c/c')\cdot b$. Hence $\varphi(f/b)=c'\cdot (f/b)$. By the previous case, $f/b$ is exactly of log-growth $\lambda$. By Lemma \ref{lem:property Robba} (i) and (iii), $f$ is also exactly of log-growth $\lambda$.
\end{proof}

\section{$(\varphi,\nabla)$-modules}\label{sec:phinabla}

Usually, $(\varphi,\nabla)$-modules are defined in an ad-hoc manner once a base ring is given as in \cite[4.9]{dJ}, \cite[2.5]{plm}, and \cite[3.2]{Tsu}. In this paper, we consider various base rings, hence, we should use a unified framework. The aim of this section is to construct such a framework. We also give some examples of base rings in \S \ref{subsec:phinabla3}. In \S \ref{subsec:uni}, we recall the category of unipotent $\nabla$-modules over $\mathcal{R}$. This section is devoted to fix notation and define terminology, and we have no new results.

\subsection{Definition}\label{subsec:phinabla1}

A {\it quadruple} $(R,\varphi,\nabla,d\varphi)$, denoted by $R$ for simplicity, consists of the following data:
\begin{enumerate}
\item[$\bullet$] $R$ is a commutative ring;
\item[$\bullet$] $\varphi:R\to R$ is a ring endomorphism;
\item[$\bullet$] $\nabla:R\to\Omega_R$ is a derivation with $\Omega_R$ an $R$-module;
\item[$\bullet$] $d\varphi:\Omega_R\to\Omega_R$ is a $\varphi$-semi-linear map,
\end{enumerate}
that make the following diagram commute
\[\xymatrix{
R\ar[r]^{\nabla}\ar[d]^{\varphi}&\Omega_R\ar[d]^{d\varphi}\\
R\ar[r]^{\nabla}&\Omega_R.
}\]
A {\it $(\varphi,\nabla)$-module over $R$} is a triple $(M,\varphi_M,\nabla_M)$ consisting of the following data
\begin{enumerate}
\item[$\bullet$] $M$ is a finite free $R$-module,
\item[$\bullet$] $\varphi_M:M\to M$ is a $\varphi$-semi-linear endomorphism such that $\varphi_M^*:\varphi^*M\to M;r\otimes m\mapsto r\varphi_M(m)$ is an isomorphism,
\item[$\bullet$] $\nabla_M:M\to M\otimes_R\Omega_R$ is an additive map satisfying
\[
\nabla_M(rm)=m\otimes\nabla(r)+r\cdot\nabla_M(m)\ \forall r\in R,\ \forall m\in M,
\]
\end{enumerate}
that make the following diagram commute
\[\xymatrix{
M\ar[r]^(.4){\nabla_M}\ar[d]^{\varphi_M}&M\otimes_R\Omega_R\ar[d]^{\varphi_M\otimes d\varphi}\\
M\ar[r]^(.4){\nabla_M}&M\otimes_R\Omega_R.
}\]
We ignore the integrability condition in this paper since we treat only the cases where $\Omega_R$ is of rank one or $\Omega_R=0$. We put $M^{\nabla}=\ker{(\nabla_M:M\to M\otimes_R\Omega_R)}$.

Unless otherwise is mentioned, we endow $R$ with the ``trivial'' $(\varphi,\nabla)$-module structure given by $(\varphi_R,\nabla_R)=(\varphi,\nabla)$ under the identification $\Omega_R\cong R\otimes_R\Omega_R$. In the following, we also drop subscripts such as $R$ or $M$ if no confusion arises.

A {\it morphism} $f:(M,\varphi_M,\nabla_M)\to (N,\varphi_N,\nabla_N)$ of $(\varphi,\nabla)$-modules over $R$ is an $R$-linear map $f:M\to N$ that makes the following diagrams commute
\[\xymatrix{
M\ar[r]^{\varphi_M}\ar[d]^{f}&M\ar[d]^{f}&&M\ar[r]^(.4){\nabla_M}\ar[d]^f&M\otimes_R\Omega\ar[d]^{f\otimes \id_{\Omega}}\\
N\ar[r]^{\varphi_N}&N,&&N\ar[r]^(.4){\nabla_N}&N\otimes_R\Omega.
}\]
Direct sums and tensor products are defined in an obvious way (see \cite[Definitions 5.3.2, 14.1.1]{pde}). Furthermore, we may endow the set of $R$-linear maps $\Hom_R(M,N)$ a $(\varphi,\nabla)$-module structure uniquely determined by the following conditions: for all $m\in M$,
\[
\varphi_{\Hom_R(M,N)}(f)(\varphi_M(m))=\varphi_N(f(m)),
\]
\[
\nabla_{\Hom_R(M,N)}(f)(m)=\nabla_N(f(m))-(f\otimes \id_{\Omega})(\nabla_M(m)),
\]
where we identify $\Hom_R(M,N)\otimes_R\Omega$ as $\Hom_R(M,N\otimes_R\Omega)$. Then the set $\Hom(M,N)$ of morphisms of $(\varphi,\nabla)$-modules coincides with the set
\[
\Hom_R^{\varphi,\nabla}(M,N)=\{f\in \Hom_R(M,N);\varphi_{\Hom_R(M,N)}(f)=f,\nabla_{\Hom_R(M,N)}(f)=0\}.
\]
We denote by $M\spcheck$ the $R$-dual $\Hom_R(M,R)$ of $M$ endowed with the $(\varphi,\nabla)$-module structure above. Note that the natural pairing $M\otimes_RM\spcheck\to R$ of $R$-modules can be regarded as a morphism of $(\varphi,\nabla)$-modules.

Let $c\in R^{\times}\cap R^{\nabla}$: for example, $c\in\mathbb{Q}^{\times}$ when $R$ is a $\mathbb{Q}$-algebra. As in \cite[Definition 3.1.5]{Doc}, we define the twist $R(c)$ of $R$ by $c$ as the rank one $(\varphi,\nabla)$-module given by
\[
R(c)=Re_c,
\]
\[
\varphi_{R(c)}(r e_c)=\varphi(r)ce_c\ \forall r\in R,
\]
\[
\nabla_{R(c)}(re_c)=e_c\otimes\nabla(r)\in R(c)\otimes_R\Omega\ \forall r\in R.
\]
For a $(\varphi,\nabla)$-module $M$, we define $M(c)$ as $M\otimes_RR(c)$.

We discuss some functorial properties (cf. \cite[3.3]{Tsu}). A {\it morphism} of quadruples
\[
(R,\varphi,\nabla,d\varphi)\to (S,\phi,\nabla,d\phi)
\]
is a pair $f=(f,df)$ consisting of the following data
\begin{enumerate}
\item[$\bullet$] $f:R\to S$ is a ring homomorphism,
\item[$\bullet$] $df:\Omega_R\to\Omega_S$ is an $f$-semi-linear map,
\end{enumerate}
that make the following diagrams commute
\[\xymatrix{
R\ar[r]^f\ar[d]^{\varphi}& S\ar[d]^{\phi}&&R\ar[d]^{\nabla}\ar[r]^f&S\ar[d]^{\nabla}&&\Omega_R\ar[r]^{df}\ar[d]^{d\varphi}&\Omega_S\ar[d]^{d\phi}\\
R\ar[r]^f&S,&&\Omega_R\ar[r]^{df}&\Omega_S,&&\Omega_R\ar[r]^{df}&\Omega_S.
}\]
We define the {\it pull-back} of a $(\varphi,\nabla)$-module $(M,\varphi_M,\nabla_M)$ over $R$ via $f$, denoted by $f^*M$, as the triple $(f^*M,\phi_{f^*M},\nabla_{f^*M})$, where
\[
\phi_{f^*M}:f^*M\to f^*M;m\otimes s\mapsto\varphi_M(m)\otimes\phi(s),
\]
\[
\nabla_{f^*M}:f^*M\to f^*M\otimes_S\Omega_S\cong M\otimes_R\Omega_S;m\otimes s\mapsto (\id_M\otimes df)(\nabla_M(m))\cdot s+m\otimes\nabla(s).
\]

Let $a\ge 1$ be an integer. Given a quadruple $(R,\varphi,\nabla,d\varphi)$, we consider the quadruple $(R,\varphi^a,\nabla,d\varphi^a)$, where $\varphi^a,d\varphi^a$ are the $a$-fold compositions of $\varphi, d\varphi$ respectively. For simplicity, a $(\varphi^a,\nabla)$-module over the quadruple $(R,\varphi^a,\nabla,d\varphi^a)$ is called a {\it $(\varphi^a,\nabla)$-module} over $R$. Given a $(\varphi,\nabla)$-module $(M,\varphi_M,\nabla_M)$ over $R$, we define the {\it $a$-pushforward} $[a]_*M$ as the $(\varphi^a,\nabla)$-module $(M,\varphi_M^a,\nabla_M)$ over $R$, where $\varphi^a_M$ denotes the $a$-fold composition of $\varphi_M$. Note that the $a$-pushforward commutes with the pull-back above.

\subsection{Matrix presentation}\label{subsec:phinabla2}

We assume that $\Omega$ is of rank one with a distinguished base $\omega$. We denote by $(\cdot)/\omega:\Omega\cong R$ the isomorphism sending $\omega$ to $1$. We identify $\Omega$ as $R$ via this isomorphism. Then the derivation $\nabla:R\to\Omega$ determines a $\mathbb{Z}$-derivation $d:R\to R$ via $\Hom_R(\Omega^1_{R/\mathbb{Z}},\Omega)\cong \Hom_R(\Omega^1_{R/\mathbb{Z}},R)$. Then to give a $(\varphi,\nabla)$-module over $R$ is equivalent to give a triple $(M,\varphi,D)$, where $M$ is a finite free $R$-module, $\varphi$ is a semi-linear endomorphism, and $D:M\to M$ is an additive map satisfying $D(rm)=dr\cdot m+rD(m)$ for $r\in R,m\in M$ satisfying the compatibility $D\circ\varphi=(\varphi(\omega)/\omega)\varphi\circ D$.

Furthermore, we give a matrix version of the above data $(M,\varphi,D)$. We choose an $R$-basis $\{e_1,\dots,e_n\}$ of $M$, We define the {\it matrix presentation} of $M$ (with respect to $\{e_i\}$) as the pair of the matrices $(A,G)$ of the actions of $(\varphi,D)$ on the basis $\{e_i\}$, that is,
\[
\varphi(e_1,\dots,e_n)=(e_1,\dots,e_n)A,
\]
\[
D(e_1,\dots,e_n)=(e_1,\dots,e_n)G.
\]
Then the condition $\varphi^*M\cong M$ implies $A\in \GL_n(R)$, and the compatibility between $\varphi$ and $D$ induces the equality
\[
d(A)+GA=(\varphi(\omega)/\omega) A\varphi(G),
\]
where $d((a_{ij})_{ij})=(d(a_{ij}))_{ij}$ and $\varphi((g_{ij})_{ij})=(\varphi(g_{ij}))_{ij}$. Conversely, if we are given a pair $(A,G)\in \M_n(R)\times\GL_n(R)$ of $n\times n$ matrices satisfying the above compatibility, then it defines a $(\varphi,\nabla)$-module over $R$ of rank $n$ by the above procedure.

\subsection{Examples of quadruples}\label{subsec:phinabla3}
We give some examples of quadruples. For examples of $(\varphi,\nabla)$-modules, see \S \ref{sec:ex}.
\begin{enumerate}
\item[1.] Difference ring $(\Omega_R=0)$

Let $(R,\varphi,\nabla,d\varphi)$ be a quadruple. Then $R_1=(R,\varphi,0,0)$ with $\Omega_{R_1}=0$ is also a quadruple. We refer to a $(\varphi,\nabla)$-module over $R_1$ as a {\it $\varphi$-module} over $R$ for simplicity. Our definition of $\varphi$-modules coincides with that in \cite[\S 14]{pde}. Moreover, there exists a natural forgetful functor from the category of $(\varphi,\nabla)$-modules over $R$ to the category of $\varphi$-modules over $R$. Thus we can use the results on $\varphi$-modules in \cite{pde}.
\item[2.] Differential ring $(\varphi=\id,d\varphi=\id)$

Let $(R,\varphi,\nabla,d\varphi)$ be a quadruple. Then $R_2=(R,\id_R,\nabla,\id_{\Omega_R})$ with $\Omega_{R_2}=\Omega_R$ is also a quadruple. We refer to a $(\varphi,\nabla)$-module over $R_2$ as a {\it $\nabla$-module} over $R$ for simplicity. Furthermore, when $\Omega_R$ is of rank one, then the category of $\nabla$-modules over $R$ is equivalent to the category of differential modules over $(R,\partial)$ in the sense of \cite[\S 6]{pde}, where $\partial$ is an arbitrary basis of $\Hom_R(\Omega_R,R)$. Thus we can use the results on $\nabla$-modules in \cite{pde}.
\item[3.] Rings of (bounded) analytic functions $K[\![t]\!]_0,K\{t\}$ on the unit disc

Let notation be as in \S\S \ref{sec:Robba}, \ref{sec:lg}. Recall that we put $K[\![t]\!]_0=\mathcal{O}_K[\![t]\!]\otimes_{\mathcal{O}_K}K$ and
\[
K\{t\}=\{\sum_{i\in\mathbb{N}}a_it^i\in K[\![t]\!];a_i\in K,|a_i|e^{-ri}\to 0\ (i\to+\infty)\ \forall r\in (0,+\infty)\}.
\]
Then $K\{t\}$ (resp. $K[\![t]\!]_0$) can be regarded as the $K$-algebra of (resp. bounded) analytic functions on the open unit disc $|t|<1$. We choose $s\in K[\![t]\!]_0$ such that $|s-t^q|_0<1$, where $|\cdot|_0$ denote Gauss norm as in Definition \ref{dfn:Robba ring} (ii). We define the quadruple $(K[\![t]\!]_0,\varphi,\nabla,d\varphi)$ as
\begin{enumerate}
\item[$\bullet$] $\varphi:K[\![t]\!]_0\to K[\![t]\!]_0;\sum{a_it^i}\mapsto\sum{\varphi_K(a_i)s^n}$,
\item[$\bullet$] $\nabla:K[\![t]\!]_0\to\Omega_{K[\![t]\!]_0}:=K[\![t]\!]_0dt;f\mapsto df/dt\cdot dt$,
\item[$\bullet$] $d\varphi:\Omega_{K[\![t]\!]_0}\to\Omega_{K[\![t]\!]_0};fdt\mapsto\varphi(f)\cdot\nabla(s)$.
\end{enumerate}
We can endow $K\{t\}$ with a quadruple structure similarly, and we have a natural morphism of quadruples $K[\![t]\!]_0\to K\{t\}$.

Furthermore, when $s=t^qu$ for $u\in K[\![t]\!]_0^{\times}$, we define the log analogue $(K[\![t]\!]_0,\varphi,\nabla_{\log},d\varphi)$ of $(K[\![t]\!]_0,\varphi,\nabla,d\varphi)$, where
\[
\nabla_{\log}:K[\![t]\!]_0\to\Omega_{K[\![t]\!]_0}(\log):=K[\![t]\!]_0dt/t;f\mapsto tdf/dt\cdot dt/t,
\]
\[
d\varphi:\Omega_{K[\![t]\!]_0}(\log)\to\Omega_{K[\![t]\!]_0}(\log);f\cdot dt/t\mapsto (q+tu^{-1}du/dt)\varphi(f)\cdot dt/t,
\]
(see \cite[Remark 17.1.2]{pde}). We also define the log analogue $(K\{t\}_{\log},\varphi,\nabla_{\log},d\varphi)$ of $(K\{t\},\varphi,\nabla,d\varphi)$ similarly (put $\nabla_{\log}(\log{t})=dt/t$). Following the terminology in \cite[Definition 4.28]{loc}, we refer to a $(\varphi,\nabla_{\log})$-module over this quadruple as a {\it log-$(\varphi,\nabla)$-module} over $K[\![t]\!]_0$. Note that we have the canonical morphism of quadruples
\[
(K[\![t]\!]_0,\varphi,\nabla,d\varphi)\to (K[\![t]\!]_0,\varphi,\nabla_{\log},d\varphi)
\]
given by $(\id_{K[\![t]\!]_0},(fdt\mapsto tf\cdot dt/t))$. Moreover, it also extends to a canonical morphism of quadruples $(K\{t\},\varphi,\nabla,d\varphi)\to (K\{t\}_{\log},\varphi,\nabla_{\log},d\varphi)$. Note that the pull-back of a $(\varphi,\nabla)$-module $(M,\varphi,\nabla)$ over $K[\![t]\!]_0$ (resp. $K\{t\}$) via the above morphism is $(M,\varphi,t\nabla)$.

\item[4.] (Bounded) Robba rings $\mathcal{R}^{\bd},\mathcal{R}$, and the log extension $\mathcal{R}_{\log}$

Let notation be as in \S \ref{sec:Robba}. For a given $s\in\mathcal{R}^{\bd}$ with $|s-t^q|_0<1$, we can define quadruple structures on $\mathcal{R}^{\bd},\mathcal{R}$ similarly as $(K[\![t]\!]_0,\varphi,\nabla,d\varphi)$ in Example 3.3.3. We define the quadruple $(\mathcal{R}_{\log},\varphi,\nabla,d\varphi)$ as an extension of $(\mathcal{R},\varphi,\nabla,d\varphi)$ by putting
\[
\nabla(\log{t})=t^{-1}dt,\ \Omega_{\mathcal{R}_{\log}}=\mathcal{R}_{\log}dt.
\]

When $s\in K[\![t]\!]_0$, we have the canonical morphism of quadruples
\[
(K[\![t]\!]_0,\varphi,\nabla,d\varphi)\to (\mathcal{R}^{\bd},\varphi,\nabla,d\varphi)
\]
given by the inclusions. Furthermore, when $s=t^qu$ for $u\in K[\![t]\!]_0^{\times}$, we have the canonical morphism of quadruples
\[
(K[\![t]\!]_0,\varphi,\nabla_{\log},d\varphi)\to (\mathcal{R}^{\bd},\varphi,\nabla,d\varphi)
\]
given by the inclusions $K[\![t]\!]_0\subset\mathcal{R}^{\bd}$ and the map
\[
\Omega_{K[\![t]\!]_0}(\log)\to\Omega_{\mathcal{R}^{\bd}};f\cdot dt/t\mapsto f/t\cdot dt.
\]
Note that we have a commutative diagram of quadruples
\[\xymatrix{
(K[\![t]\!]_0,\varphi,\nabla,d\varphi)\ar[r]\ar@/_1pc/[rr]&(K[\![t]\!]_0,\varphi,\nabla_{\log},d\varphi)\ar[r]& (\mathcal{R}^{\bd},\varphi,\nabla,d\varphi),
}\]
where the upper left arrow is given in Example 3.3.3. Moreover, the diagram extends to a commutative diagram of quadruples
\[\xymatrix{
(K\{t\},\varphi,\nabla,d\varphi)\ar[r]\ar@/_1pc/[rr]&(K\{t\}_{\log},\varphi,\nabla_{\log},d\varphi)\ar[r]& (\mathcal{R}_{\log},\varphi,\nabla,d\varphi).
}\]

\item[5.] The Amice ring $\mathcal{E}$

Let notation be as in \S \ref{sec:Robba}. We consider the ring of bounded analytic functions on the ``generic '' unit disc in the sense of Dwork, that is, the $\mathcal{E}$-algebra $\mathcal{E}[\![X-t]\!]_0$ with the new variable $X-t$. We define the $K$-algebra homomorphism
\[
\tau:\mathcal{E}\to\mathcal{E}[\![X-t]\!]_0;f\mapsto\sum_{n=0}^{\infty}\frac{1}{n!}\frac{d^nf}{dt^n}(X-t)^n.
\]
We can define a quadruple structure on $\mathcal{E}$ similarly as $(K[\![t]\!]_0,\varphi,\nabla,d\varphi)$ in Example 3.3.3. We have the canonical morphism of quadruples given by the inclusions
\[
(\mathcal{R}^{\bd},\varphi,\nabla,d\varphi)\to (\mathcal{E},\varphi,\nabla,d\varphi).
\]
We endow $\mathcal{E}[\![X-t]\!]_0$ with a structure of a quadruple by
\[
\varphi:\mathcal{E}[\![X-t]\!]_0\to\mathcal{E}[\![X-t]\!]_0;\sum_{n=0}^{\infty}a_n(X-t)^n\mapsto\sum_{n=0}^{\infty}\varphi(a_n)(\tau(\varphi(t))-\varphi(t))^n,
\]
\[
\nabla:\mathcal{E}[\![X-t]\!]_0\to\Omega_{\mathcal{E}[\![X-t]\!]_0}=\mathcal{E}[\![X-t]\!]_0d(X-t);f\mapsto\frac{df}{d(X-t)}d(X-t),
\]
\[
d\varphi:\Omega_{\mathcal{E}[\![X-t]\!]_0}\to\Omega_{\mathcal{E}[\![X-t]\!]_0};fd(X-t)\mapsto\varphi(f)\nabla(\varphi(X-t)).
\]
We put
\[
d\tau:\Omega_{\mathcal{E}}\to\Omega_{\mathcal{E}[\![X-t]\!]_0};fdt\mapsto\tau(f)d(X-t).
\]
Then $\tau=(\tau,d\tau)$ is a morphism of the quadruples $\mathcal{E}\to\mathcal{E}[\![X-t]\!]_0$.
\end{enumerate}

\subsection{Unipotent connection}\label{subsec:uni}
In this paper, we would like to work with unipotent connections rather than quasi-unipotent connections (see Remark \ref{rem:quasi-unipotent}). We interpret some classical results on unipotent connections in terms of $\mathcal{R}_{\log}$.

\begin{dfn}
Let $M$ be a $\nabla$-module over $\mathcal{R}$. We define
\[
\mathbf{V}(M)=(M\otimes_{\mathcal{R}}\mathcal{R}_{\log})^{\nabla},
\]
which is a $K$-vector space endowed with the monodromy operator $N$, i.e., the $K$-linear endomorphism induced by the map
\[
\id_M\otimes d/d\log{t}:M\otimes_{\mathcal{R}}\mathcal{R}_{\log}\to M\otimes_{\mathcal{R}}\mathcal{R}_{\log};m\otimes\sum_ia_i(\log{t})^i\mapsto m\otimes\sum_iia_i(\log{t})^{i-1}.
\]
We say that $M$ is {\it solvable} in $\mathcal{R}_{\log}$ if $\dim_K\mathbf{V}(M)=\rank_{\mathcal{R}}M$. We say that $M$ is {\it unipotent} if $M$ is a successive extension of trivial $\nabla$-modules over $\mathcal{R}$ (\cite[Definition 4.27]{plm}). Note that the category of (unipotent) $\nabla$-modules over $\mathcal{R}$ is an abelian $\otimes$-category (\cite[6.2]{Cre}). Moreover, any subquotient of a unipotent $\nabla$-module over $\mathcal{R}$ is unipotent again. Also note that for $0\to M'\to M\to M''\to 0$ an exact sequence of $(\varphi,\nabla)$-modules over $\mathcal{R}$, $M$ is unipotent if and only if $M'$ and $M''$ are unipotent.
\end{dfn}

\begin{lem}\label{lem:V}
For $M$ a $\nabla$-module over $\mathcal{R}$, the canonical map
\[
\mathbf{V}(M)\otimes_K\mathcal{R}_{\log}\to M\otimes_{\mathcal{R}}\mathcal{R}_{\log}
\]
is injective.
\end{lem}
\begin{proof}
By d\'evissage, we may assume that $M$ is irreducible. We may also assume $\mathbf{V}(M)\neq 0$.

Let $x=\sum_{i=0}^nm_i\otimes (\log{t})^i\in \mathbf{V}(M),m_i\in M,m_n\neq 0$. Then $N^nx=n!m_n\otimes 1\in \mathbf{V}(M)\cap M=M^{\nabla}$. By the irreducibility of $M$, $M=\mathcal{R}m_n\cong\mathcal{R}$, in which case the assertion is obvious.
\end{proof}

\begin{cor}\label{cor:V}
For $M$ a $\nabla$-module over $\mathcal{R}$,
\[
\dim_K\mathbf{V}(M)\le \rank_{\mathcal{R}}M,
\]
and the monodromy operator $N$ is nilpotent.
\end{cor}
\begin{proof}
The first inequality follows from Lemma \ref{lem:V}, and the nilpotency of $N$ follows from the local nilpotency of $N$.
\end{proof}

\begin{lem}\label{lem:solv=uni}
Let $M$ be a $\nabla$-module over $\mathcal{R}$. Then $M$ is solvable in $\mathcal{R}_{\log}$ if and only if $M$ is unipotent.
\end{lem}
\begin{proof}
The necessity is well-known: for example, see the proof of \cite[Theorem 6.13]{plm}. To prove the sufficiency, we may assume that $M$ is irreducible by d\'evissage. We have $M^{\nabla}=\mathbf{V}(M)^{N=0}\neq 0$ by the nilpotency of $N$. Hence the map $M^{\nabla}\otimes_K\mathcal{R}\to M$, which is injective by Lemma \ref{lem:V}, is an isomorphism. In particular, $M$ is trivial.
\end{proof}


\section{Frobenius slope filtration}\label{sec:slope}

In this section, we recall the definition of Frobenius slope filtrations on $\phi$-modules over complete (discrete) valuation fields following \cite{Tsu,CT,CT2}: the formulation may be slightly different from a usual one. Precisely speaking, our Frobenius slope filtration is indexed by $\mathbb{R}$ so that it will be invariant under an arbitrary $a$-pushforward functor $[a]_*$ in \S \ref{subsec:phinabla1}. We also recall some basic properties by translating \cite[\S 14]{pde} into our formulation. The results in this section will be used without referring unless otherwise is specified.

\begin{notation}
In this section, let
\begin{enumerate}
\item[$\bullet$] $(F,|\cdot|)$ a complete discrete valuation field of mixed characteristic $(0,p)$: we do not need to normalize $|\cdot|$.
\item[$\bullet$] $\phi:F\to F$ an isometric ring endomorphism.
\end{enumerate}
Recall that $q$ is a positive power of $p$. In this paper, unless otherwise is mentioned, we consider the case that $\phi$ is a $q$-power Frobenius lift. In the following, a base change means tensoring $E$ over $F$, where $E$ is a complete extension of $F$ to which $\phi$ extends isometrically (with the same $q$ unless otherwise is mentioned). As $E$, we typically consider the $\phi$-completion of $F$, i.e., the completion of $\varinjlim{(F\xrightarrow[]{\phi}F\xrightarrow[]{\phi}\dots)}$, on which $\phi$ is bijective.
\end{notation}

\begin{dfn}[{\cite[Definitions 6.1.3, 14.4.6]{pde}}]
Let $M$ be a non-zero $\phi$-module over $F$. We choose the supremum norm of $M$ with respect to a basis of $M$ (\cite[Definition 1.3.2]{pde}). We define the {\it spectral radius} of $\phi:M\to M$ as
\[
|\phi|_{\SP,M}=\lim_{n\to\infty}(\sup_{v\in V,v\neq 0}{|\phi^n(v)|/|v|})^{1/n}.
\]
We say that $M$ is {\it pure of (Frobenius) slope} $\lambda$ if
\[
|\phi|_{\SP,M}=1/|\phi|_{\SP,M\spcheck}=|q|^{\lambda}.
\]
For example, $F(q)$ (see \S \ref{subsec:phinabla1}) is pure of slope $1$.

The spectral radius (and hence, the slope) is independent of the choice of the basis, and hence, depends only on the isomorphism class of $M$ by \cite[Theorem 1.3.6 and Proposition 6.1.5]{pde}. It also commutes with base change (\cite[Lemma 14.4.3 (c)]{pde}). Note that the slope does depend on the choice of $q$. To formulate Chiarellotto-Tsuzuki conjecture \ref{conj:CT}, we need to make our slope filtration to be invariant under the $a$-pushforward (see Lemma \ref{lem:slope5} (II)-(i)). Hence we choose $q^a$ as $q$ when we consider $\phi^a$-modules over $F$.

Note that $M$ is pure of slope $\lambda$ if and only if $M$ is pure of norm $|q|^{\lambda}$ in the terminology in \cite[\S 14]{pde}. In the following, we translate the results on the difference modules in \cite[\S 14]{pde} into our language.
\end{dfn}

\begin{lem}\label{lem:slope1}
If $M$ is an irreducible $\phi$-module over $F$, then $M$ is pure.
\end{lem}
\begin{proof}
This is an immediate corollary of \cite[Theorem 14.4.15]{pde}.
\end{proof}

\begin{dfn}
Let $M$ be a non-zero $\phi$-module over $F$. We define the {\it (Frobenius) slope multiset} of $M$ as
\[
\coprod_i\{\!\underbrace{\lambda_i,\dots,\lambda_i}_{\dim_F{M_i}\text{ times}}\!\},
\]
where $\{M_i\}$ is the Jordan-H\"older constituents of $M$, and $\lambda_i$ is the slope of $M_i$, which makes sense by Lemma \ref{lem:slope1}.

Recall that the slope multiset is consisting of rational numbers (\cite[Corollary 14.4.5]{pde}), invariant under base changes, and if $0\to M'\to M\to M''\to 0$ is an exact sequence of $\phi$-modules, then the slope multiset of $M$ is the disjoint union of those of $M'$ and $M''$.

For simplicity, we define the slope multiset of $M=0$ as the empty set.
\end{dfn}

\begin{lem}[{\cite[Proposition 14.4.8]{pde}}]\label{lem:slope3}
Let $M$ be a non-zero $\phi$-module over $F$. Then $M$ is pure of slope $\lambda$ if and only if the slope multiset is $\{\lambda,\dots,\lambda\}$.
\end{lem}

\begin{lem}\label{lem:slope4}
Let $M,N$ be non-zero $\phi$-modules over $F$.
\begin{enumerate}
\item $M$ is pure of slope $\lambda$ if and only if $M\spcheck$ is pure of slope $-\lambda$.
\item The slope multiset of $M\spcheck$ is the negative of that of $M$.
\item The slope multiset of $M\otimes_FN$ consists of $\lambda+\mu$, where $\lambda,\mu$ runs the slope multisets of $M,N$ respectively.
\end{enumerate}
\end{lem}
\begin{proof}
\begin{enumerate}
\item It follows from the canonical isomorphism $(M\spcheck)\spcheck\cong M$.
\item It follows from Lemma \ref{lem:slope1}.
\item By Lemma \ref{lem:slope1}, it reduces to the case where $M,N$ are pure. Then the assertion follows from \cite[Corollary 14.4.9]{pde} and Lemma \ref{lem:slope3}.
\end{enumerate}
\end{proof}

\begin{thm}\label{thm:slope1}
\begin{enumerate}
\item (\cite[Theorem 14.4.13]{pde}) Assume that $F$ is inversive. Let $M$ be a $\phi$-module over $F$. Then there exists a unique direct sum decomposition
\[
M=\oplus_{\lambda\in\mathbb{R}}M_{\lambda}
\]
of $\phi$-modules, in which each non-zero $M_{\lambda}$ is pure of slope $\lambda$.

Moreover, $\dim_FM_{\lambda}$ is equal to the multiplicity of $\lambda$ in the slope multiset of $M$.
\item (cf. \cite[Definition 2.3]{CT}) Let $M$ be a $\phi$-module over $F$. Then there exists a unique increasing filtration $\{S_{\lambda}(M);\lambda\in\mathbb{R}\}$ of $\phi$-modules satisfying:
\begin{enumerate}
\item[(a)] the exhaustiveness and separatedness;
\item[(b)] (right continuity)
\[
S_{\lambda}(M)=\cap_{\mu>\lambda}S_{\mu}(M);
\]
\item[(c)] (rationality) if $S_{\lambda}(M)/\cup_{\mu<\lambda}S_{\mu}(M)$ is non-zero, then it is pure of slope $\lambda$ and $\lambda\in\mathbb{Q}$.
\end{enumerate}
Moreover, $\dim_FS_{\lambda}(M)/\cup_{\mu<\lambda}S_{\mu}(M)$ is equal to the multiplicity of $\lambda$ in the slope multiset of $M$.

We call $S_{\bullet}(M)$ {\it (Frobenius) slope filtration} of $M$.
\item If $F$ is inversive, then
\[
S_{\lambda}(M)=\oplus_{\mu\le\lambda}M_{\mu},
\]
\[
S_{\lambda}(M)/\cup_{\mu<\lambda}S_{\mu}(M)\cong M_{\lambda}.
\]
\end{enumerate}
\end{thm}

\begin{lem}\label{lem:slopehom}
Let $M,N$ be $\phi$-modules over $F$. If the slope multisets of $M$ and $N$ are disjoint, then
\[
\Hom_F^{\phi}(M,N)=0.
\]
\end{lem}
\begin{proof}
By base change, we may assume that $F$ is inversive. By identifying $\Hom_F^{\phi}(M,N)$ as $(M\spcheck\otimes_FN)^{\phi=1}$ then using Theorem \ref{thm:slope1} (i), it reduces to prove that if $L$ is pure of slope $\lambda\neq 0$, then $L^{\phi=1}=0$. We may assume $\lambda>0$ by replacing $\phi$ by $\phi^{-1}$ if necessary since $\phi-\id_L=-\phi(\phi^{-1}-\id_L)$. Then $\phi$ is contractive since the operator norm of $\phi$ is less than or equal to $|\phi|_{\SP,L}<1$. Since $F$ is complete, $1-\phi$ is invertible with the inverse $1+\phi+\phi^2+\dots$.
\end{proof}

\begin{proof}[Proof of Theorem \ref{thm:slope1}]
\begin{enumerate}
\item[(ii)] $\bullet$ Existence

By \cite[Theorem 14.4.15]{pde}, there exists a filtration
\[
0=\mathcal{F}_0M\subset\mathcal{F}_1M\subset\dots\subset\mathcal{F}_lM=M
\]
of $\phi$-modules such that $\mathcal{F}_iM/\mathcal{F}_{i-1}M$ is pure of slope $\lambda_i$ with $\lambda_1<\dots<\lambda_l$. Moreover, the $\lambda_i$'s are rational by \cite[Corollary 14.4.5]{pde}. We define
\[
S_{\lambda}(M)=\mathcal{F}_iM\text{ if }\lambda\in [\lambda_i,\lambda_{i+1}),
\]
where we put $\lambda_0=-\infty$, $\lambda_{l+1}=+\infty$.

$\bullet$ Uniqueness

Let $S'_{\bullet}(M)$ be another filtration satisfying the condition. Since $S_{\lambda}(M)$ (resp. $M/S'_{\lambda}(M)$) is a successive extension of $\phi$-modules pure of slope $\le \lambda$ (resp. $>\lambda$), $\Hom_F^{\phi}(S_{\lambda}(M),M/S'_{\lambda}(M))=0$ by Lemma \ref{lem:slopehom}. In particular, $S_{\lambda}(M)\subset S'_{\lambda}(M)$. Similarly, $S_{\lambda}(M)\supset S'_{\lambda}(M)$.
\item[(iii)] Since $\{\oplus_{\mu\ge\lambda}M_{\mu}\}_{\lambda}$ satisfies the conditions in (ii), we obtain the first equality by the uniqueness of Frobenius slope filtration. The second assertion follows from the first one.
\end{enumerate}
\end{proof}

\begin{lem}\label{lem:slope5}
\begin{enumerate}
\item[(I)] Assume that $F$ is inversive.
\begin{enumerate}
\item[(i)] (base change) Let $M$ be a $\phi$-module over $F$, and $E$ a complete extension of $F$ to which $\phi$ extends isometrically. Then there exists a caononical isomorphism
\[
E\otimes_FM_{\bullet}\cong(E\otimes_FM)_{\bullet}.
\]
Moreover, for $a\ge 1$, there exists a canonical isomorphism of $\phi^a$-modules over $F$
\[
[a]_*(M_{\bullet})\cong ([a]_*M)_{\bullet}.
\]
\item[(ii)] (strictness) If $0\to M'\to M\to M''\to 0$ is an exact sequence of $\phi$-modules over $F$, then there exists an exact sequence of filtrations
\[
0\to M'_{\bullet}\to M_{\bullet}\to M''_{\bullet}\to 0.
\]
\item[(iii)] (tensor compatibility) For $M,N$ $\phi$-modules over $F$,
\[
(M\otimes_F N)_{\delta}=\oplus_{\lambda+\mu=\delta}M_{\lambda}\otimes_F N_{\mu}.
\]
\item[(iv)] (duality) For $\lambda$ an arbitrary real number,
\[
((M\spcheck)_{-\lambda})^{\perp}=\oplus_{\mu\neq\lambda}M_{\mu},
\]
where $(\cdot)^{\perp}$ denotes the orthogonal part with respect to the canonical pairing $M\otimes_FM\spcheck\to F$.
\item[(v)] (twist) Let $\alpha\in K^{\times}$ and $\mu\in\mathbb{Q}$ such that $|\alpha|=|q|^{\mu}$. Then, for $M$ a $\phi$-module over $F$, there exists a canonical isomorphism
\[
M(\alpha)_{\lambda}\cong M_{\lambda-\mu}\otimes_F(F(\alpha)).
\]
\end{enumerate}
\item[(II)]
\begin{enumerate}
\item[(i)] (base change) Let $M,E$ be as in (I)-(i). There exists a canonical isomorphism
\[
E\otimes_FS_{\bullet}(M)\cong S_{\bullet}(E\otimes_FM).
\]
Moreover, for $a\ge 1$, there exists a canonical isomorphism of $\phi^a$-modules over $F$
\[
[a]_*(S_{\bullet}(M))\cong S_{\bullet}([a]_*M).
\]
\item[(ii)] (strictness) If $0\to M'\to M\to M''\to 0$ is an exact sequence of $\phi$-modules over $F$, then there exists an exact sequence of filtrations
\[
0\to S_{\bullet}(M')\to S_{\bullet}(M)\to S_{\bullet}(M'')\to 0.
\]
\item[(iii)] (tensor compatibility) For $M,N$ $\phi$-modules over $F$,
\[
S_{\delta}(M\otimes_F N)=\sum_{\lambda+\mu=\delta}S_{\lambda}(M)\otimes_F S_{\mu}(N).
\]
\item[(iv)] (duality) For $\lambda$ an arbitrary real number,
\[
(S_{-\lambda}(M\spcheck))^{\perp}=\cup_{\mu<\lambda}S_{\mu}(M),
\]
where $(\cdot)^{\perp}$ denotes the orthogonal part with respect to the canonical pairing $M\otimes_FM\spcheck\to F$.
\item[(v)] (twist) Let $\alpha\in K^{\times}$ and $\mu\in\mathbb{Q}$ such that $|\alpha|=|q|^{\mu}$. Then, for $M$ a $\phi$-module over $F$, there exists a canonical isomorphism
\[
S_{\lambda}(M(\alpha))\cong S_{\lambda-\mu}(M)\otimes_F (F(\alpha)).
\]
\end{enumerate}
\end{enumerate}
\end{lem}
\begin{proof}
\begin{enumerate}
\item[(I)]
\begin{enumerate}
\item[(i)] It follows from the uniqueness in Theorem \ref{thm:slope1} (i).
\item[(ii)] It follows from Lemma \ref{lem:slope5}.
\item[(iii)] The decomposition
\[
M\otimes_F N=\oplus_{\delta}(\oplus_{\lambda+\mu=\delta}M_{\lambda}\otimes_F N_{\mu})
\]
satisfies the condition in Theorem \ref{thm:slope1} (i) by Lemma \ref{lem:slope4} (ii), which implies the assertion.
\item[(iv)] We identify $M$ as $\oplus_{\lambda}M_{\lambda}$ by Theorem \ref{thm:slope1} (i). By taking the dual of the $\lambda$-th projection $\oplus_{\lambda}M_{\lambda}\to M_{\lambda}$, we identify $(M_{\lambda})\spcheck$ as a $\varphi$-submodule of $M\spcheck$, which is pure of slope $-\lambda$ by Lemma \ref{lem:slope4} (i). Hence the decomposition $M\spcheck=\oplus_{\lambda}(M_{\lambda})\spcheck$ satisfies the condition on Theorem \ref{thm:slope1} (i), in particular, $(M\spcheck)_{-\lambda}=(M_{\lambda})\spcheck$. We have $((M_{\lambda})\spcheck)^{\perp}=\oplus_{\mu\neq\lambda}M_{\mu}$ by construction, which implies the assertion.
\item[(v)] It follows from Lemma \ref{lem:slope4} (iii).
\end{enumerate}
\item[(II)] By base change, we may assume that $F$ is inversive. By Theorem \ref{thm:slope1} (iii), each of (i), (ii), and (iii) reduces to the corresponding assertions in (I). We prove (iv). Since Frobenius slope filtration is stable under base change ((II)-(i)), we may assume that $F$ is inversive. By Theorem \ref{thm:slope1} (iii) and I-(iv),
\[
(S_{-\lambda}(M\spcheck))^{\perp}=(\oplus_{\mu\ge\lambda}(M\spcheck)_{-\mu})^{\perp}=\cap_{\mu\ge\lambda}(((M\spcheck)_{-\mu})^{\perp})=\cap_{\mu\ge\lambda}(\oplus_{\delta\neq\mu}M_{\delta})=\oplus_{\delta<\lambda}M_{\delta}=\cup_{\mu<\lambda}S_{\mu}(M).
\]
The assertion (v) follows from (II)-(iii) and
\[
S_{\delta}(F(\alpha))=
\begin{cases}
F(\alpha)&\text{if }\delta\ge\mu,\\
0&\text{if }\delta<\mu.
\end{cases}
\]
\end{enumerate}
\end{proof}

The following will be necessary in the proof of Main Theorem.

\begin{dfn}[{cf. \cite[4.2.1]{Doc}}]\label{dfn:eigen}
Let $d$ be a positive integer. A {\it (Frobenius) $d$-eigenvector} of a $\phi$-module $M$ over $F$ is a non-zero element $v$ of $M$ such that
\[
\phi^d(v)=cv
\]
for some $c\in F^{\times}$. We refer the quotient $\log{|c|}/\log{|q^d|}$ as the {\it (Frobenius) slope} of $v$.
\end{dfn}

\begin{lem}\label{lem:slope6}
Let $M$ be a $\phi$-module over $F$. If $v\in M$ is a $d$-eigenvector of slope $\lambda$, then $v\in S_{\lambda}(M)$ and $v\notin\cup_{\mu<\lambda}S_{\mu}(M)$. In particular, if $M$ is pure, then $\lambda$ is equal to the slope of $M$.
\end{lem}
\begin{proof}
By Lemma \ref{lem:slope5} (II)-(i), we may assume $d=1$ after replacing $\phi$ by $\phi^d$. Then $Fv$ is a $\phi$-submodule of $M$, and $S_{\lambda}(Fv)=Fv,S_{\mu}(Fv)=0$ for $\mu<\lambda$, which implies the assertion by Lemma \ref{lem:slope5} (II)-(ii).
\end{proof}

\begin{lem}[{Dieudonn\'e-Manin theorem, \cite[Theorem 14.6.3]{pde}}]\label{lem:DM}
Assume that the residue field of $F$ is strongly difference-closed. Let $M$ be a non-zero $\phi$-module over $F$ of rank $n$. Then there exist a positive integer $d$ and a basis $e_1,\dots,e_n$ of $M$ consisting of $d$-eigenvectors. Moreover, if we denote by $\mu_i$ the slope of $e_i$, then the multiset $\{\mu_i;1\le i\le n\}$ coincides with the slope multiset of $M$, and
\[
M_{\lambda}=\oplus_{i:\mu_i=\lambda}Fe_i.
\]
\end{lem}

\begin{dfn}\label{dfn:reverse}
Assume that $F$ is inversive. Let $M$ be a $\phi$-module over $F$, and $M=\oplus_{\mu}M_{\mu}$ the decomposition given by Theorem \ref{thm:slope1} (i). We define the {\it reverse filtration} $\{S^{\lambda}(M);\lambda\in\mathbb{R}\}$ of $M$ by
\[
S^{\lambda}(M)=\oplus_{\mu\ge\lambda}M_{\mu}.
\]

Note that the filtration $S^{\bullet}(M)$ is decreasing, exhaustive, separated, and consisting of $\phi$-submodules of $M$. Moreover, the slope multiset of $S^{\bullet}(M)$ coincides with that of $M_{\bullet}$.
\end{dfn}


\section{Logarithmic growth filtration for $(\varphi,\nabla)$-modules}\label{sec:logfil}

In this section, we recall the definition of the log-growth filtrations for $(\varphi,\nabla)$-modules over the rings $K[\![t]\!]_0,\mathcal{R}^{\bd}$, and $\mathcal{E}$ following \cite{CT,Ohk}. We show that the log-growth filtration is invariant under the base changes via the morphisms in Example 3.3.4 (Lemma \ref{lem:invariance log-growth}). We also recall basic properties on the log-growth filtration, that can be found in the literature.

\begin{convention}
In the rest of this paper, we endow the rings $K[\![t]\!]_0,\mathcal{R}^{\bd}$, and $\mathcal{E}$ with the structures of quadruples defined in \S \ref{subsec:phinabla3} unless otherwise is mentioned. In particular, when we consider a $(\varphi,\nabla)$-module over $\mathcal{R}^{\bd}$, we tacitly assume $\varphi(t)\in\mathcal{R}^{\bd}$ with $|\varphi(t)-t^q|_0<1$; when we further consider a $(\varphi,\nabla)$-module over $K[\![t]\!]_0$, we impose the extra assumption $\varphi(t)\in K[\![t]\!]_0$.
\end{convention}

\subsection{Over $\mathcal{R}^{\bd}$}\label{subsec:logfil1}

Let $M$ be a $(\varphi,\nabla)$-module over $\mathcal{R}^{\bd}$. We say that $M$ is {\it solvable} in $\mathcal{R}_{\log}$ if $M\otimes_{\mathcal{R}^{\bd}}\mathcal{R}$ is solvable in $\mathcal{R}_{\log}$ in the sense of \S \ref{subsec:uni}, or, equivalently, is unipotent by Lemma \ref{lem:solv=uni}. Note that the category of $(\varphi,\nabla)$-modules over $\mathcal{R}^{\bd}$ solvable in $\mathcal{R}_{\log}$ is an abelian $\otimes$-category.

Let $M$ be a $(\varphi,\nabla)$-module of rank $n$ over $\mathcal{R}^{\bd}$ solvable in $\mathcal{R}_{\log}$. As in \cite{Ohk} (where $V,\Sol$ are denoted by $\mathfrak{V},\mathfrak{Sol}$), we define the set of {\it analytic horizontal sections} of $M$ (in $\mathcal{R}_{\log}$) as
\[
V(M)=(M\otimes_{\mathcal{R}^{\bd}}\mathcal{R}_{\log})^{\nabla},
\]
which is a $\varphi$-module over $K$ of dimension $n$. A {\it solution} of $M$ (in $\mathcal{R}_{\log}$) is a $\nabla$-equivariant $\mathcal{R}^{\bd}$-linear map
\[
f:M\to \mathcal{R}_{\log}.
\]
We define $\Sol(M)$ as the set of solutions of $M$, that is,
\[
\Sol(M)=\Hom_{\mathcal{R}^{\bd}}^{\nabla}(M,\mathcal{R}_{\log}),
\]
which is a $\varphi$-module over $K$ of dimension $n$ since $\Sol(M)$ is canonically isomorphic to $V(M\spcheck)$. Moreover, the natural perfect pairing $M\otimes_{\mathcal{R}^{\bd}}M\spcheck\to \mathcal{R}^{\bd}$ induces a perfect pairing
\[
(\ ,\ ):V(M)\otimes_K \Sol(M)\to K.
\]
We endow $\Sol(M)$ and $V(M)$ with growth filtrations as follows. For $\lambda\in\mathbb{R}$, a solution $f$ of $M$ is {\it of log-growth} $\lambda$ if $f(M)\subset \Fil_{\lambda}\mathcal{R}_{\log}$. We define $\Sol_{\lambda}(M)$ as the set of solutions of $M$ of log-growth $\lambda$, that is,
\[
\Sol_{\lambda}(M)=\{f\in \Sol(M);f(M)\subset \Fil_{\lambda}\mathcal{R}_{\log}\}.
\]
We also define the {\it $\lambda$-th log-growth filtration} of $M$ as
\[
V(M)^{\lambda}=(\Sol_{\lambda}(M))^{\perp}=\{v\in V(M);(v,f)=0\ \forall f\in \Sol_{\lambda}(M)\},
\]
where $(\cdot)^{\perp}$ denotes the orthogonal part with respect to the above pairing.

We endow $V(M),\Sol(M)$, which are $\varphi$-modules over $K$, with Frobenius slope filtrations by applying the results in \S \ref{sec:slope} with $(F,\phi,q)=(K,\varphi_K,q)$. Thus $V(M)$ and $\Sol(M)$ are endowed with two filtrations respectively, that is, the growth filtrations $V(M)^{\bullet}$ and $\Sol_{\bullet}(M)$, and Frobenius slope filtrations $S_{\bullet}(V(M))$ and $S_{\bullet}(\Sol(M))$. Frobenius slope filtrations enjoy many good properties as we see in Lemma \ref{lem:slope5}. Although the log-growth filtration will be compared to Frobenius slope filtration, some fundamental properties on the log-growth filtration does not follow by definition. For example, the following proposition will follow from Proposition \ref{prop:main}; we record here since it is not used in the interim.

\begin{prop}\label{prop:lg property}
Let $M$ be a $(\varphi,\nabla)$-module over $\mathcal{R}^{\bd}$ solvable in $\mathcal{R}_{\log}$. Then the filtration $V(M)^{\bullet}$ is decreasing, exhaustive, and separated. Similarly, the filtration $\Sol_{\bullet}(M)$ is increasing, exhaustive, and separated.
\end{prop}
\begin{proof}
By duality, we have only to prove the assertion for $\Sol_{\bullet}(M)$. The only non-trivial part is that $\Sol_{\bullet}(M)$ is exhaustive. By Proposition \ref{prop:main}, we have $\Sol_{\lambda}(M)\supset S_{\lambda-\lambda_{\max}}(\Sol(M))$, where $\lambda_{\max}$ denotes the maximum Frobenius slope of $M_{\mathcal{E}}$. We put $\mu$ the maximum Frobenius slope of $\Sol(M)$. Then $\Sol_{\lambda_{\max}+\mu}(M)=\Sol(M)$.
\end{proof}

The growth filtrations $V(M)^{\bullet}$ and $\Sol_{\bullet}(M)$ share the same information via the duality as above. Nevertheless, we will handle mainly $\Sol_{\bullet}(M)$ in this paper. One reason is that an element of $f\in\Sol_{\bullet}(M)$ is a homomorphism $f:M\to\mathcal{R}_{\log}$ by definition, and the injectivity of $f$ has a special meaning in the proof of Theorem \ref{conj:bd} (ii).

\subsection{Over $K[\![t]\!]_0$}\label{subsec:logfil2}

Let $M$ be a $(\varphi,\nabla)$-module over $K[\![t]\!]_0$. Then the construction in \S \ref{subsec:logfil1} works for $M$ after replacing $\mathcal{R}^{\bd},\mathcal{R}_{\log}$, and $\Fil_{\lambda}\mathcal{R}_{\log}$ by $K[\![t]\!]_0,K\{t\}$, and $K[\![t]\!]_{\lambda}$ respectively: Dwork's trick (\cite[Corollary 17.2.2]{pde}) assures that $M$ is solvable in $K\{t\}$ (in an obvious sense). We denote each resulting object by the same symbol as in \S \ref{subsec:logfil1}, that is, we obtain the following $\varphi$-modules over $K$:
\[
\Sol(M)=\Hom_{K[\![t]\!]_0}^{\nabla}(M,K\{t\}),
\]
\[
\Sol_{\lambda}(M)=\{f\in \Sol(M);f(M)\subset K[\![t]\!]_{\lambda}\},
\]
\[
V(M)=(M\otimes_{K[\![t]\!]_0}K\{t\})^{\nabla},
\]
\[
V(M)^{\lambda}=(\Sol_{\lambda}(M))^{\perp}=\{v\in V(M);(v,f)=0\ \forall f\in \Sol_{\lambda}(M)\},
\]
where $(\cdot)^{\perp}$ denotes the orthogonal part with respect to the perfect pairing $V(M)\otimes_K\Sol(M)\to K$. Our notation is compatible with that in \cite[p. 473]{CT}.

The above construction also works for $M$ log-$(\varphi,\nabla)$-modules over $K[\![t]\!]_0$ after replacing $K\{t\}$ and $K[\![t]\!]_{\lambda}$ by $K\{t\}_{\log}$ and $\Fil_{\lambda}K\{t\}_{\log}$ respectively: a nilpotent analogue of Dwork's trick (\cite[Corollary 17.2.4]{pde}) assures that $M$ is solvable in $K\{t\}_{\log}$ (in an obvious sense). We denote each resulting object by the same symbol as above, that is, we obtain the following $\varphi$-modules over $K$:
\[
\Sol(M)=\Hom_{K[\![t]\!]_0}^{\nabla}(M,K\{t\}_{\log}),
\]
\[
\Sol_{\bullet}(M)=\{f\in \Sol(M);f(M)\subset \Fil_{\bullet}K\{t\}_{\log}\},
\]
\[
V(M)=(M\otimes_{K[\![t]\!]_0}K\{t\}_{\log})^{\nabla},
\]
\[
V(M)^{\bullet}=(\Sol_{\bullet}(M))^{\perp},
\]
where $(\cdot)^{\perp}$ denotes the orthogonal part with respect to the perfect pairing $V(M)\otimes_K\Sol(M)\to K$.

The following lemma asserts that the log-growth filtration is invariant under any base change given in Example 3.3.4. Hence we are allowed to use the same symbols $\Sol_{\bullet}(\cdot)$ and $V(\cdot)^{\bullet}$ for abuse of notation. Also, as a consequence, instead of studying the growth filtrations on $V(M)$ and $\Sol(M)$ for $(\varphi,\nabla)$-modules over $K[\![t]\!]_0$ as in \cite{CT}, we may study those for $(\varphi,\nabla)$-modules over $\mathcal{R}^{\bd}$. In particular, an analogue of Proposition \ref{prop:lg property} for a (log-)$(\varphi,\nabla)$-module over $K[\![t]\!]_0$ holds.

\begin{lem}[{cf. \cite[Lemma 4.15]{Ohk}}]\label{lem:invariance log-growth}
Let $R\to S$ denote any of the following three morphisms of quadruples given in Examples 3.3.3 and 3.3.4
\[\xymatrix{
(K[\![t]\!]_0,\varphi,\nabla,d\varphi)\ar[r]\ar@/_1pc/[rr]&(K[\![t]\!]_0,\varphi,\nabla_{\log},d\varphi)\ar[r]& (\mathcal{R}^{\bd},\varphi,\nabla,d\varphi).
}\]
Under the base change via $R\to S$, the $\varphi$-modules $V(\cdot)$ and $\Sol(\cdot)$ for (log-)$(\varphi,\nabla)$-modules over $R$ are invariant up to natural isomorphisms. Moreover, the growth filtration $V(\cdot)^{\bullet}$ and $\Sol_{\bullet}(\cdot)$ are also invariant up to natural isomorphisms.
\end{lem}
\begin{proof}
It suffices to prove the assertions in the cases of $(K[\![t]\!]_0,\varphi,\nabla,d\varphi)\to (K[\![t]\!]_0,\varphi,\nabla_{\log},d\varphi)$ and $(K[\![t]\!]_0,\varphi,\nabla_{\log},d\varphi)\to (\mathcal{R}^{\bd},\varphi,\nabla,d\varphi)$. We give a proof in the first case; a similar argument works in the second case.

Note that $M_{\mathcal{R}^{\bd}}$ is solvable in $\mathcal{R}_{\log}$ (even in $\mathcal{R}$) by Dwork's trick. By Corollary \ref{cor:logfilR1}, the inclusion $K\{t\}\subset\mathcal{R}_{\log}$ induces canonical injections
\begin{equation}\label{eq:logfil2_1}
\Sol(M)\to \Sol(M_{\mathcal{R}^{\bd}}),
\end{equation}
\begin{equation}\label{eq:logfil2_2}
V(M)\to V(M_{\mathcal{R}^{\bd}}),
\end{equation}
which commute with the canonical pairings
\[
\Sol(M)\otimes_K V(M)\to K,\ \Sol(M_{\mathcal{R}^{\bd}})\otimes_K V(M_{\mathcal{R}^{\bd}})\to K.
\]
By Corollary \ref{cor:V}, the morphism (\ref{eq:logfil2_2}) is an isomorphism, hence, the morphism (\ref{eq:logfil2_1}) is also an isomorphism by duality. We identify $\Sol(M_{\mathcal{R}^{\bd}})$ as $\Sol(M)$ via (\ref{eq:logfil2_1}). Then
\begin{align*}
\Sol_{\lambda}(M_{\mathcal{R}^{\bd}})&=\{f\in \Sol(M);f(M)\cdot\mathcal{R}^{\bd}\subset \Fil_{\lambda}\mathcal{R}_{\log}\}=\{f\in \Sol(M);f(M)\subset \Fil_{\lambda}\mathcal{R}_{\log}\}\\
&=\{f\in \Sol(M);f(M)\subset K[\![t]\!]_{\lambda}\}=\Sol_{\lambda}(M)
\end{align*}
by Lemma \ref{lem:property Robba} (i), (iii), and Corollary \ref{cor:logfilR1}. By duality, we obtain an isomorphism $V(M)^{\bullet}\to V(M_{\mathcal{R}^{\bd}})^{\bullet}$.
\end{proof}

\subsection{Over $\mathcal{E}$}\label{subsec:logfil3}

Let $M$ be a $(\varphi,\nabla)$-module over $\mathcal{E}$. Let $\tau:\mathcal{E}\to\mathcal{E}[\![X-t]\!]_0$ be the morphism of quadruples given by Example 3.3.5. The pull-back $\tau^*M$ of $M$ via $\tau$ is a $(\varphi,\nabla)$-module over $\mathcal{E}[\![X-t]\!]_0$. By applying the construction in \S \ref{subsec:logfil2} to $\tau^*M$, we denote each resulting object by the same symbol as in \S \ref{lem:invariance log-growth}, that is, we obtain the following $\varphi$-modules over $\mathcal{E}$:
\[
\Sol(M)=\Sol(\tau^*M)=\Hom_{\mathcal{E}[\![X-t]\!]_0}^{\nabla}(\tau^*M,\mathcal{E}\{X-t\}),
\]
\[
\Sol_{\bullet}(M)=\Sol_{\bullet}(\tau^*M)=\{f\in \Sol(M);f(\tau^*M)\subset\mathcal{E}[\![X-t]\!]_{\bullet}\},
\]
\[
V(M)=V(\tau^*M)=(\tau^*M\otimes_{\mathcal{E}[\![X-t]\!]_0}\mathcal{E}\{X-t\})^{\nabla},
\]
\[
V(M)^{\bullet}=V(\tau^*M)^{\bullet}=(\Sol_{\bullet}(M))^{\perp},
\]
where $(\cdot)^{\perp}$ denotes the orthogonal part with respect to the perfect pairing $V(M)\otimes_{\mathcal{E}}\Sol(M)\to\mathcal{E}$.

The feature in the case of $\mathcal{E}$ is that the log-growth filtration $V(M)^{\bullet}$ descends to $M$.

\begin{dfn}\label{dfn:bounded}
Let $M$ be a $(\varphi,\nabla)$-module over $\mathcal{E}$ of rank $n$. We say that $M$ is {\it solvable} in $\mathcal{E}[\![X-t]\!]_{\lambda}$ if $\dim_{\mathcal{E}}\Sol_{\lambda}(M)=n$. We say that $M$ is {\it bounded} if $M$ is solvable in $\mathcal{E}[\![X-t]\!]_0$.

Note that if $M$ is solvable in $\mathcal{E}[\![X-t]\!]_{\lambda}$, then any subquotient of $M$ as a $(\varphi,\nabla)$-module over $\mathcal{E}$ is also solvable in $\mathcal{E}[\![X-t]\!]_{\lambda}$.
\end{dfn}

\begin{thm}[{\cite[2.6, 3.5]{Rob},\cite[Theorem 3.2, Proposition 3.9]{CT},\cite[Theorem 2.2]{CT2}}]\label{thm:Rob}
Let $M$ be a $(\varphi,\nabla)$-module over $\mathcal{E}$. For any real number $\lambda$, there exists a unique $(\varphi,\nabla)$-submodule $M^{\lambda}$ of $M$ such that an equality
\[
\Sol_{\lambda}(M)=\Sol(M/M^{\lambda})
\]
holds in $\Sol(M)$, where $\Sol(M/M^{\lambda})$ is canonically regarded as a subspace of $\Sol(M)$. Equivalently, $M^{\lambda}$ is characterized as the minimal $(\varphi,\nabla)$-submodule of $M$ such that $M/M^{\lambda}$ is solvable in $\mathcal{E}[\![X-t]\!]_{\lambda}$.
\end{thm}

\begin{dfn}[{\cite[Definition 3.8]{CT}}]
Let $M$ be a $(\varphi,\nabla)$-module over $\mathcal{E}$. We call $\{M^{\lambda};\lambda\in\mathbb{R}\}$ {\it the log-growth filtration} of $M$. The quotient $M/M^0$ is called the {\it bounded quotient} of $M$. Note that we have $M^{\lambda}=M$ for $\lambda<0$ by definition.
\end{dfn}

We gather some basic facts on the log-growth filtration.

\begin{lem}\label{lem:lgE1}
Let $M$ be a $(\varphi,\nabla)$-module over $\mathcal{E}$.
\begin{enumerate}
\item (\cite[Theorem 3.2 (2), Corollary 3.5]{CT}) The filtration $M^{\bullet}$ is decreasing, exhaustive, and separated.
\item (\cite[Theorem 3.2 (4)]{CT}) If $M\neq 0$, then $M^{\lambda}\neq M$ for $\lambda\ge 0$. In particular, the minimum slope of $M^{\bullet}$ is equal to zero.
\end{enumerate}
\end{lem}

\begin{lem}[{\cite[Proposition 3.6]{CT}}]\label{lem:lgE2}
Let $f:M\to N$ be a morphism of $(\varphi,\nabla)$-modules over $\mathcal{E}$.
\begin{enumerate}
\item $f(M^{\bullet})\subset N^{\bullet}$.
\item If $f$ is surjective, then $f(M^{\bullet})=N^{\bullet}$.
\end{enumerate}
\end{lem}

\begin{lem}\label{lem:lgE3}
For $(\varphi,\nabla)$-modules $M_1,\dots,M_n$ over $\mathcal{E}$, there exists a canonical isomorphism $(\oplus_{i=1}^nM_i)^{\bullet}\cong\oplus_{i=1}^nM_i^{\bullet}$.
\end{lem}
\begin{proof}
By applying Lemma \ref{lem:lgE2} (i) to the canonical injection $M_i\hookrightarrow\oplus_{i=1}^nM_i$, we obtain $\oplus_{i=1}^nM_i^{\lambda}\subset(\oplus_{i=1}^nM_i)^{\lambda}$. The converse follows by applying Lemma \ref{lem:lgE2}~(ii) to the canonical surjection $\oplus_{i=1}^nM_i\to M_i$.
\end{proof}

\begin{prop}[{\cite[Proposition 6.2, Corollary 6.5]{CT}}]\label{prop:lgE5}
Let $M$ be a $(\varphi,\nabla)$-module over $\mathcal{E}$. If $M$ is pure as a $\varphi$-module, then $M$ is bounded, i.e., $M^0=M$.
\end{prop}

Finally, we recall some facts on Frobenius slope filtration for a $(\varphi,\nabla)$-module $M$. Let $S_{\bullet}(M)$ denote Frobenius slope filtration on $M$ (Theorem \ref{thm:slope1}). Then each $S_{\lambda}(M)$ is a $(\varphi,\nabla)$-submodule of $M$ (\cite[Proposition 6.2]{CT}). We should remark that each graded piece of $S_{\bullet}(M)$ is bounded by Proposition \ref{prop:lgE5}, in particular, any $(\varphi,\nabla)$-module over $\mathcal{E}$ is a successive extension of bounded $(\varphi,\nabla)$-modules over $\mathcal{E}$. This fact distinguishes $(\varphi,\nabla)$-modules over $\mathcal{E}$ from those over $\mathcal{R}^{\bd}$.


\subsection{Pure of bounded quotient}

Ultimately, we would like to compare log-growth filtrations and Frobenius slope filtrations. However there exists a difference between Frobenius structures and differential structures, which cannot be ignored: the twisting operation $M\mapsto M(c)$ in \S \ref{subsec:phinabla1} does change Frobenius structures while does not change differential structures. Chiarellotto and Tsuzuki introduce the following notion for $(\varphi,\nabla)$-modules, which is necessary to formulate Chiarellotto-Tsuzuki conjecture \ref{conj:CT}.

\begin{dfn}[{\cite[Definition 5.1]{CT2}}]
\begin{enumerate}
\item A $(\varphi,\nabla)$-module $M$ over $\mathcal{E}$ is said to be {\it pure of bounded quotient} (called PBQ for simplicity) if $M/M^0$ is pure as a $\varphi$-module. By Theorem \ref{thm:Rob}, this is equivalent to say that $\Sol_0(M)$ is pure as a $\varphi$-module over $\mathcal{E}$.
\item A (log-)$(\varphi,\nabla)$-module $M$ over $K[\![t]\!]_0$ or $\mathcal{R}^{\bd}$ is said to be {\it pure of bounded quotient} (called PBQ for simplicity) if the generic fiber $M_{\mathcal{E}}$ of $M$ is PBQ as a $(\varphi,\nabla)$-module over $\mathcal{E}$.
\end{enumerate}
\noindent For simplicity, we define that $M=0$ is PBQ.
\end{dfn}

We give some first properties.

\begin{lem}
Let $M$ be a $(\varphi,\nabla)$-module over $\mathcal{E}$. Then the following are equivalent.
\begin{enumerate}
\item $M$ is PBQ.
\item If $Q$ is a quotient of $M$ as a $(\varphi,\nabla)$-module such that $Q$ is bounded (Definition \ref{dfn:bounded}), then $Q$ is pure as a $\varphi$-module over $\mathcal{E}$.
\end{enumerate}
\end{lem}
\begin{proof}
(i)$\Rightarrow$(ii) By Theorem \ref{thm:Rob}, there exists a surjection of $(\varphi,\nabla)$-modules $M/M^0\to Q$. Hence $Q$ is pure.

\noindent (ii)$\Rightarrow$(i) Apply (ii) to $Q=M/M^0$.
\end{proof}

\begin{cor}[{\cite[Proposition 5.3]{CT2}}]\label{cor:qt inv of PBQ}
Let $M$ be a $(\varphi,\nabla)$-module over $K[\![t]\!]_0,\mathcal{R}^{\bd}$, or $\mathcal{E}$. Then any quotient $Q$ of $M$ as a $(\varphi,\nabla)$-module is PBQ.
\end{cor}
\begin{proof}
It suffices to prove the assertion in the case of $\mathcal{E}$, where the assertion follows the previous lemma.
\end{proof}

\section{Examples of $(\varphi,\nabla)$-modules}\label{sec:ex}

In this section, we give some examples of $(\varphi,\nabla)$-modules. In \S \ref{subsec:ex1}, we calculate their log-growth filtrations and Frobenius slope filtrations, and compare the two filtrations. In \S \ref{subsec:ex2}, we also explain whether some examples in \S \ref{subsec:ex1} are PBQ or not by calculating the bounded quotients. Some examples will be useful as counterexamples to some questions. This section is an exposition of examples, hence, we do not hesitate to use some results, which will be proved later in this paper.

\subsection{Log-growth filtrations and Frobenius slope filtrations}\label{subsec:ex1}

\begin{ex}[the rank one case]
Let $M$ be a $(\varphi,\nabla)$-module of rank one over $\mathcal{R}^{\bd}$ solvable in $\mathcal{R}_{\log}$. We claim that $M$ is trivial as a $\nabla$-module over $\mathcal{R}^{\bd}$, in particular, $\Sol_0(M)=\Sol(M)$. Let $\{e\}$ be a basis of $M$. Since $M\otimes_{\mathcal{R}^{\bd}}\mathcal{R}$ is a trivial $\nabla$-module over $\mathcal{R}$ by Lemma \ref{lem:solv=uni}, there exists $c\in (\mathcal{R})^{\times}=(\mathcal{R}^{\bd})^{\times}$ such that $\{ce\}$ is a basis of $V(M)$, which implies the claim.

By the above claim, we may choose a horizontal basis $e\in M$. Write $\varphi(e)=ae$ with $a\in K$. Let $\lambda_{\max}$ denote the maximum Frobenius slope of $M_{\mathcal{E}}$. Since the map $e\mapsto 1$ is a basis of $\Sol(M)$, the Frobenius slope of $\Sol(M)$ is equal to $-\log{|a|}/\log{|q|}$, which is also equal to $-\lambda_{\max}$. Hence we have
\[
\Sol_{\lambda}(M)=S_{\lambda-\lambda_{\max}}(\Sol(M))=
\begin{cases}
0&\text{if }\lambda\in (-\infty,0),\\
\Sol(M)&\text{if }\lambda\in [0,+\infty).
\end{cases}
\]
By a similar argument,
\[
\Sol_{\lambda}(M\spcheck)=S_{\lambda+\lambda_{\max}}(\Sol(M\spcheck))=
\begin{cases}
0&\text{if }\lambda\in (-\infty,0),\\
\Sol(M\spcheck)&\text{if }\lambda\in [0,+\infty).
\end{cases}
\]

Let $M$ be a (log-)$(\varphi,\nabla)$-module over $K[\![t]\!]_0$ of rank one. Then $M$ is trivial as a (log-)$\nabla$-module over $K[\![t]\!]_0$. In fact, since $M\otimes_{K[\![t]\!]_0}\mathcal{R}^{\bd}$ is a trivial $\nabla$-module over $\mathcal{R}^{\bd}$, the assertion follows from \cite[Proposition 17.2.5]{pde}. We can calculate the log-growth filtration and Frobenius slope filtration similarly as above.

Let $M$ be a $(\varphi,\nabla)$-module over $\mathcal{E}$ of rank one. Then $M$ is bounded since $\tau^*M$ is a trivial $\nabla$-module over $\mathcal{E}[\![X-t]\!]_0$. We can calculate the log-growth filtration and Frobenius slope filtration similarly as above.
\end{ex}
\begin{ex}[direct sum]
Let $M=K[\![t]\!]_0\oplus K[\![t]\!]_0(q)$. Let $e_1=(1,0),e_2=(0,1)$. Then $\Sol(M)$ admits the basis $\{f_1,f_2\}$ defined by $f_1:(e_1,e_2)\mapsto (0,1),f_2:(e_1,e_2)\mapsto (1,0)$. We have $\varphi(f_1,f_2)=(f_1,f_2)\ \diag(q^{-1},1)$. Hence we have
\[
\Sol_{\lambda}(M)=
\begin{cases}
0&\text{if }\lambda\in (-\infty,0),\\
\Sol(M)&\text{if }\lambda\in [0,+\infty),
\end{cases}
S_{\lambda-1}(\Sol(M))=
\begin{cases}
0&\text{if }\lambda\in (-\infty,0),\\
Kf_1&\text{if }\lambda\in [0,1),\\
\Sol(M)&\text{if }\lambda\in [1,+\infty).
\end{cases}
\]

\end{ex}
\begin{ex}
Assume $\varphi(t)=t^q$. Let $\mu\in [0,1]\cap\mathbb{Q}$ such that $q^{\mu}\in (K^{\times})^{\varphi_K=1}$. We define the rank two $(\varphi,\nabla)$-module $M_{\mu}$ over $K[\![t]\!]_0$ with the basis $\{e_1,e_2\}$, whose matrix presentation is given by
\[
A_{\mu}=\begin{pmatrix}
1&-q^{\mu}t\\
0&q^{\mu}
\end{pmatrix},
G_{\mu}=
\begin{pmatrix}
0&g_{\mu}\\
0&0
\end{pmatrix},
\]
where $g_{\mu}=\sum_{i=0}^{\infty}q^{(1-\mu)i}t^{q^i-1}\in K[\![t]\!]_0$. We put
\[
x_{\mu}=\int{g_{\mu}dt}=\sum_{i=0}^{\infty}q^{-\mu i}t^{q^i}\in K\{t\}.
\]
Then $x_{\mu}$ is exactly of log-growth $\mu$. The solution space $\Sol(M_{\mu})$ admits the basis $\{f_1,f_2\}$ defined by
\[
f_1:(e_1,e_2)\mapsto (0,1),\ f_2:(e_1,e_2)\mapsto (1,x_{\mu}).
\]
We also have, by noting $\varphi(x_{\mu})=q^{\mu}x_{\mu}-q^{\mu}t$,
\[
\varphi(f_1,f_2)=(f_1,f_2)\ \diag(q^{-\mu},1).
\]
Hence we have
\[
\Sol_{\lambda}(M_{\mu})=S_{\lambda-\mu}(\Sol(M_{\mu}))=
\begin{cases}
0&\text{if }\lambda\in (-\infty,0),\\
Kf_1&\text{if }\lambda\in [0,\mu),\\
\Sol(M_{\mu})&\text{if }\lambda\in [\mu,+\infty).
\end{cases}
\]

The matrix presentation of the dual $M_{\mu}\spcheck$ with respect to the dual basis $\{e_1\spcheck,e_2\spcheck\}$ of $\{e_1,e_2\}$ is given by
\[
{}^t(A_{\mu}^{-1})=\begin{pmatrix}
1&0\\
t&q^{-\mu}
\end{pmatrix},
-{}^tG_{\mu}=
\begin{pmatrix}
0&0\\
-g_{\mu}&0
\end{pmatrix}.
\]
Hence $\Sol(M_{\mu}\spcheck)$ admits the basis $\{h_1,h_2\}$ defined by
\[
h_1:(e_1\spcheck,e_2\spcheck)\mapsto (1,0),\ h_2:(e_1\spcheck,e_2\spcheck)\mapsto (-x_{\mu},1),
\]
which satisfies
\[
\varphi(h_1,h_2)=(h_1,h_2)\ \diag(1,q^{\mu}).
\]
Hence we have
\[
\Sol_{\lambda}(M_{\mu}\spcheck)=S_{\lambda}(\Sol(M_{\mu}\spcheck))=
\begin{cases}
0&\text{if }\lambda\in (-\infty,0),\\
Kh_1&\text{if }\lambda\in [0,\mu),\\
\Sol(M_{\mu}\spcheck)&\text{if }\lambda\in [\mu,+\infty).
\end{cases}
\]
\end{ex}
\begin{ex}[{pushout; a generalization of \cite[Example 5.2 (3)]{CT2}}]
Let notation and assumption be as in Example 6.1.3. Let $\mu,\delta\in [0,1]\cap\mathbb{Q}$ such that $q^{\mu},q^{\delta}\in (K^{\times})^{\varphi_K=1}$, and $\mu<\delta$. We define the rank three $(\varphi,\nabla)$-module $M_{\mu,\delta}$ over $K[\![t]\!]_0$ with the basis $\{e_1,e_2,e_3\}$, whose matrix presentation is given by
\[
A_{\mu,\delta}=\begin{pmatrix}
1&-q^{\mu}t&-q^{\delta}t\\
0&q^{\mu}&0\\
0&0&q^{\delta}
\end{pmatrix},
G_{\mu,\delta}=
\begin{pmatrix}
0&g_{\mu}&g_{\delta}\\
0&0&0\\
0&0&0
\end{pmatrix}.
\]
Actually, $M_{\mu,\delta}$ satisfies a pushout diagram
\[\xymatrix{
K[\![t]\!]_0\ar[r]\ar[d]&M_{\mu}\ar[d]\\
M_{\delta}\ar[r]&M_{\mu,\delta}.
}\]
Then $\Sol(M)$ admits the basis $\{f_1,f_2,f_3\}$ defined by
\[
f_1:(e_1,e_2,e_3)\mapsto (0,0,1),\ f_2:(e_1,e_2,e_3)\mapsto (0,1,0),\ f_3:(e_1,e_2,e_3)\mapsto (1,x_{\mu},x_{\delta}),
\]
which satisfies
\[
\varphi(f_1,f_2,f_3)=(f_1,f_2,f_3)\ \diag(q^{-\delta},q^{-\mu},1).
\]
Therefore
\[
\Sol_{\lambda}(M_{\mu,\delta})=
\begin{cases}
0&\text{if }\lambda\in (-\infty,0),\\
Kf_1\oplus Kf_2&\text{if }\lambda\in [0,\delta),\\
\Sol(M_{\mu,\delta})&\text{if }\lambda\in [\delta,+\infty),
\end{cases}
\]
\[
S_{\lambda-\delta}(\Sol(M_{\mu,\delta}))=
\begin{cases}
0&\text{if }\lambda\in (-\infty,0),\\
Kf_1&\text{if }\lambda\in [0,\delta-\mu),\\
Kf_1\oplus Kf_2&\text{if }\lambda\in [\delta-\mu,\delta),\\
\Sol(M_{\mu,\delta})&\text{if }\lambda\in [\delta,+\infty).
\end{cases}
\]
In particular, $S_{\lambda-\delta}(\Sol(M_{\mu,\delta}))\subset \Sol_{\lambda}(M_{\mu,\delta})$ with equality unless $\lambda\in [0,\delta-\mu)$.

The matrix presentation of the dual $M_{\mu,\delta}\spcheck$ with respect to the dual basis $\{e_1\spcheck,e_2\spcheck,e_3\spcheck\}$ of $\{e_1,e_2,e_3\}$ is given by
\[
{}^t(A_{\mu,\delta}^{-1})=\begin{pmatrix}
1&0&0\\
t&q^{-\mu}&0\\
t&0&q^{-\delta}
\end{pmatrix},
-{}^tG_{\mu,\delta}=
\begin{pmatrix}
0&0&0\\
-g_{\mu}&0&0\\
-g_{\delta}&0&0
\end{pmatrix}.
\]
As in Example 6.1.3, $\Sol(M_{\mu,\delta}\spcheck)$ admits the basis $\{h_1,h_2,h_3\}$ defined by
\[
h_1:(e_1\spcheck,e_2\spcheck,e_3\spcheck)\mapsto (1,0,0),\ h_2:(e_1\spcheck,e_2\spcheck,e_3\spcheck)\mapsto (-x_{\mu},1,0),\ h_3:(e_1\spcheck,e_2\spcheck,e_3\spcheck)\mapsto (-x_{\delta},0,1),
\]
which satisfies
\[
\varphi(h_1,h_2,h_3)=(h_1,h_2,h_3)\ \diag(1,q^{\mu},q^{\delta}).
\]
Therefore
\[
\Sol_{\lambda}(M_{\mu,\delta}\spcheck)=S_{\lambda}(\Sol(M_{\mu,\delta}\spcheck))=
\begin{cases}
0&\text{if }\lambda\in (-\infty,0),\\
Kh_1&\text{if }\lambda\in [0,\mu),\\
Kh_1\oplus Kh_2&\text{if }\lambda\in [\mu,\delta),\\
\Sol(M_{\mu,\delta}\spcheck)&\text{if }\lambda\in [\delta,+\infty).
\end{cases}
\]
\end{ex}
\begin{ex}[tensor product]
Let notation and assumption be as in Example 6.1.3. We consider $M=M_{\mu}\otimes_{K[\![t]\!]_0}M_{\mu}\spcheck$ with the following basis
\[
\mathbf{e}_1=e_1\otimes e_2\spcheck,\ \mathbf{e}_2=e_1\otimes e_1\spcheck-e_2\otimes e_2\spcheck,\ \mathbf{e}_3=e_2\otimes e_1\spcheck,\ \mathbf{e}_4=e_1\otimes e_1\spcheck+e_2\otimes e_2\spcheck.
\]
Then the matrix presentation is given by
\[
\begin{pmatrix}
q^{-\mu}&2t&-q^{\mu}t^2&0\\
0&1&-q^{\mu}t&0\\
0&0&q^{\mu}&0\\
0&0&0&1
\end{pmatrix},
\begin{pmatrix}
0&-2g_{\mu}&0&0\\
0&0&g_{\mu}&0\\
0&0&0&0\\
0&0&0&0
\end{pmatrix}.
\]
Thus $M$ splits into the direct sum of two modules defined by $\{\mathbf{e}_1,\mathbf{e}_2,\mathbf{e}_3\}$ and $\mathbf{e}_4$ respectively: the first one is the kernel of the canonical pairing $M\otimes_{K[\![t]\!]_0}M\spcheck\to K[\![t]\!]_0$, and the decomposition is nothing but a splitting of the pairing. The solution space $\Sol(M)$ admits the basis $\{f_1,f_2,f_3,f_4\}$ defined by
\[
f_1:(\mathbf{e}_1,\mathbf{e}_2,\mathbf{e}_3,\mathbf{e}_4)\mapsto (0,0,0,1),\ f_2:(\mathbf{e}_1,\mathbf{e}_2,\mathbf{e}_3,\mathbf{e}_4)\mapsto (0,0,1,0),
\] 
\[
f_3:(\mathbf{e}_1,\mathbf{e}_2,\mathbf{e}_3,\mathbf{e}_4)\mapsto (0,1,x_{\mu},0),\ f_4:(\mathbf{e}_1,\mathbf{e}_2,\mathbf{e}_3,\mathbf{e}_4)\mapsto (1,-2x_{\mu},-x_{\mu}^2,0),\ 
\]
We also have
\[
\varphi(f_1,f_2,f_3,f_4)=(f_1,f_2,f_3,f_4)\ \diag(1,q^{-\mu},1,q^{\mu}).
\]
Note that $x_{\mu}^2=\sum_{i,j\in\mathbb{N}}q^{-\mu(i+j)}t^{q^i+q^j}=\sum_{i>j\in\mathbb{N}}q^{-\mu(i+j)}t^{q^i+q^j}+\sum_{i\in\mathbb{N}}2q^{-2\mu i}t^{2q^i}$ is exactly of log-growth $2\mu$ since so is $\sum_{i\in\mathbb{N}}2q^{-2\mu i}t^{2q^i}$. Hence we have
\[
\Sol_{\lambda}(M)=
\begin{cases}
0&\text{if }\lambda\in (-\infty,0),\\
Kf_1\oplus Kf_2&\text{if }\lambda\in [0,\mu),\\
Kf_1\oplus Kf_2\oplus Kf_3&\text{if }\lambda\in [\mu,2\mu),\\
\Sol(M)&\text{if }\lambda\in [2\mu,+\infty),
\end{cases}
\]
\[
S_{\lambda-\mu}(\Sol(M))=
\begin{cases}
0&\text{if }\lambda\in (-\infty,0),\\
Kf_2&\text{if }\lambda\in [0,\mu),\\
Kf_1\oplus Kf_2\oplus Kf_3&\text{if }\lambda\in [\mu,2\mu),\\
\Sol(M)&\text{if }\lambda\in [2\mu,+\infty).
\end{cases}
\]
We also note that
\[
\sum_{\varepsilon+\delta=\lambda}\Sol_{\varepsilon}(M_{\mu})\otimes_K \Sol_{\delta}(M_{\mu}\spcheck)=\sum_{\varepsilon+\delta=\lambda}S_{\varepsilon-\mu}(\Sol(M_{\mu}))\otimes_K S_{\delta}(\Sol(M_{\mu}\spcheck))=S_{\lambda-\mu}(\Sol(M))
\]
by Example 6.1.3 and Lemma \ref{lem:slope5} (II)-(iii). In particular, $\sum_{\varepsilon+\delta=\lambda}\Sol_{\varepsilon}(M_{\mu})\otimes_K \Sol_{\delta}(M_{\mu}\spcheck)\subset\Sol_{\lambda}(M)$ with equality unless $\lambda\in [0,\mu)$.\end{ex}
\begin{ex}[symmetric square]
Let $M$ be a $(\varphi,\nabla)$-module over $K[\![t]\!]_0$ of rank $2$ such that (a) $M$ is not a trivial $\nabla$-module over $K[\![t]\!]_0$, and, (b) there exists a rank one $(\varphi,\nabla)$-submodule $M'$ of $M$. We will calculate the growth filtrations and Frobenius slope filtrations of $M^{\otimes 2}$ and $\mathrm{Sym}^2M$ by assuming several results including Theorem \ref{conj:CT}.

We start by calculating $\Sol_{\bullet}(M)$ and $S_{\bullet}(M)$, which can be seen as a generalization of Example 6.1.3. By Christol transfer theorem (\cite[Proposition 4.3]{CT}), $M_{\mathcal{E}}$ is not bounded. Hence $M$ is PBQ by Example 6.2.2. By Example 6.1.1, $M',M/M'$ are trivial $\nabla$-modules over $K[\![t]\!]_0$. Let $\{e_1,e_2\}$ be a basis of $M$ such that $e_1$ is a horizontal element of $M'$, and $e_2$ is a lift of a horizontal element of $M/M'$. Then $\Sol(M)$ admits a basis $\{f_1,f_2\}$ such that $f_1(e_1,e_2)=(0,1)$ and $f_2(e_1,e_2)=(1,x)$ for some $x\in K[\![t]\!]$. Let $\{0,\lambda\}$ be the slope multiset of $M_{\mathcal{E}}^{\bullet}$, or, equivalently, of $\Sol_{\bullet}(M_{\mathcal{E}})$. By the non-boundedness of $M_{\mathcal{E}}$, we have $\lambda>0$. We claim that the slope multiset of $\Sol_{\bullet}(M)$ is equal to $\{0,\lambda\}$. Since $f_1\in\Sol_0(M)$, the slope multiset of $\Sol_{\bullet}(M)$ is of the form $\{0,\lambda'\}$. By the semicontinuity theorem on the log-growth Newton polygons (Corollary \ref{cor:LGFDW}), we have $\lambda'=\lambda$. By the claim, together with Lemma \ref{lem:right continuous}, we have $\Sol_{\mu}(M)=0$ if $\mu\in (-\infty,0)$, $\Sol_{\mu}(M)=Kf_1$ if $\mu\in [0,\lambda)$, and $\Sol_{\mu}(M)=\Sol(M)$ otherwise. By Theorem \ref{conj:CT} (i), $f_2$ is exactly of log-growth $\lambda$, hence, $x$ is exactly of log-growth $\lambda$. Since $M$ is PBQ, we have $\Sol_{\bullet}(M)=S_{\bullet-\lambda_{\max}}(\Sol(M))$ by Theorem \ref{conj:CT} (ii), where $\lambda_{\max}$ denotes the maximum Frobenius slope of $M_{\mathcal{E}}$. Moreover,  $\varphi$ acts on $\{f_1,f_2\}$ as an upper triangular matrix $(\alpha_{ij})_{i\le j}$ with $|\alpha_{11}|=|q^{-\lambda_{\max}}|$ and $|\alpha_{22}|=|q^{\lambda-\lambda_{\max}}|$.

For a given $a,b,c,d\in K$, put $f=a(f_1\otimes f_1)+b(f_1\otimes f_2-f_2\otimes f_1)+c(f_1\otimes f_2)+d(f_2\otimes f_2)$. Then, we have
\[
f(e_1\otimes e_1,e_1\otimes e_2,e_2\otimes e_1,e_2\otimes e_2)=(d,-b+dx,b+c+dx,a+cx+dx^2).
\]
By Proposition \ref{prop:weak tensor}, $2\lambda$ is a slope of $\Sol_{\bullet}(M^{\otimes 2})$, which implies that $x^2$ is exactly of log-growth $2\lambda$. Hence we obtain
\[
\Sol_{\mu}(M^{\otimes 2})=
\begin{cases}
0&\text{if }\mu\in (-\infty,0),\\
K(f_1\otimes f_1)\oplus K(f_1\otimes f_2-f_2\otimes f_1)&\text{if }\mu\in [0,\lambda),\\
K(f_1\otimes f_1)\oplus K(f_1\otimes f_2-f_2\otimes f_1)\oplus K(f_1\otimes f_2)&\text{if }\mu\in [\lambda,2\lambda),\\
\Sol(M^{\otimes 2})&\text{if }\mu\in [2\lambda,+\infty).
\end{cases}
\]
In particular, $f_2\otimes f_2\notin\Sol_{\mu}(M^{\otimes 2})$ for any $\mu<2\lambda$. We also have $S_{\mu-2\lambda_{\max}}(\Sol(M^{\otimes 2}))=K(f_1\otimes f_1)\subsetneq\Sol_{\mu}(M^{\otimes 2})$ if $\mu\in [0,\lambda)$, and $S_{\mu-2\lambda_{\max}}(\Sol(M^{\otimes 2}))=\Sol_{\mu}(M^{\otimes 2})$ otherwise.

Since $f_1\otimes f_1,f_1\otimes f_2+f_2\otimes f_1$, and $f_2\otimes f_2\in\Sol(M^{\otimes 2})$ kill $e_1\otimes e_2-e_2\otimes e_1$, they can be regarded as a basis of $\Sol(\mathrm{Sym}^2M)$. The above calculation of $\Sol_{\bullet}(M^{\otimes 2})$ implies
\[
\Sol_{\mu}(\mathrm{Sym}^2M)=
\begin{cases}
0&\text{if }\mu\in (-\infty,0),\\
K(f_1\otimes f_1)&\text{if }\mu\in [0,\lambda),\\
K(f_1\otimes f_1)\oplus K(f_1\otimes f_2+f_2\otimes f_1)&\text{if }\mu\in [\lambda,2\lambda),\\
\Sol(\mathrm{Sym}^2M)&\text{if }\mu\in [2\lambda,+\infty).
\end{cases}
\]
We also have $S_{\mu-2\lambda_{\max}}(\Sol(\mathrm{Sym}^2(M)))=\Sol_{\mu}(\mathrm{Sym}^2(M))$ for an arbitrary $\mu$ by using the strictness of Frobenius slope filtration.

\end{ex}
\begin{ex}[Bessel overconvergent $F$-isocrystal]\label{ex:Bessel}
In this example, we will study Bessel overconvergent $F$-isocrystal $\mathcal{M}^{\mathrm{Bessel}}$ constructed by Dwork (\cite{DwB}). We adopt the following treatment in \cite[Example 6.2.6]{Tsu}. Assume for simplicity that $k$ is algebraically closed. Let $X$ be a connected smooth curve over $k$. Then one can define the (bounded, integral) Robba ring $\mathcal{R}^{(\bd,\Int)}_s$ at an arbitrary closed point $s\in X$, which is non-canonically isomorphic to $\mathcal{R}^{(\bd,\Int)}$ (\cite[6.1]{Tsu}). Let $U$ be a non-empty open subscheme of $X$, and $Z=X\setminus U$. Let $\mathcal{M}$ be an overconvergent $F$-isocrystal on $U/K$ around $Z$. Then one can define a corresponding $(\varphi,\nabla)$-module $\mathcal{M}_s$ over $\mathcal{R}^{\bd}_s$ for any closed point $s\in X$. In the case that $\mathcal{M}$ is given by the first relative rigid cohomology associated to Legendre family of elliptic curves, the log-growth filtration and the Frobenius slope filtration of $\mathcal{M}_s$ are determined by Dwork (see \cite[\S 7.4]{CT}). 

We briefly recall a few properties of $\mathcal{M}^{\mathrm{Bessel}}$. Let $p\neq 2$. Assume that $K$ contains Dwork's $\pi$, that is, $\pi^{p-1}=-p$, and $\varphi$ is a $p$-power Frobenius lift on $K$ such that $\varphi(\pi)=\pi$ (such a $\varphi$ exists if $K$ is obtained by the fraction field of the ring of Witt vectors over $k$ then adjoining $\pi$). Put $X=\mathbb{P}^1_k,Z=\{0,\infty\}$, and $U=X\setminus Z$. Then $\mathcal{M}^{\mathrm{Bessel}}$ is an overconvergent $F$-isocrystal of rank $2$ on $U/K$ around $Z$. If $s\neq 0,\infty$, then $\mathcal{M}^{\mathrm{Bessel}}_s\otimes_{\mathcal{R}^{\bd}_s}\mathcal{R}_s$ is unipotent since $\mathcal{M}^{\mathrm{Bessel}}_s$ descends to $K[\![t]\!]_0$. We will study $\mathcal{M}^{\mathrm{Bessel}}_s$ for $s=0,\infty$, where some particular phenomena are observed.

$\bullet$ At $s=0$

We identify $\mathcal{R}^{(\bd)}_0$ as $\mathcal{R}^{(\bd)}$, where $\varphi$ on $\mathcal{R}$ is an absolute $p$-power Frobenius lift. Then the $(\varphi,\nabla)$-module $\mathcal{M}^{\mathrm{Bessel}}_0$ over $\mathcal{R}^{\bd}$ descends to a log-$(\varphi,\nabla)$-module $M_0$ over $K[\![t]\!]_0$, whose matrix presentation $(A,G)$ satisfying: $A=(a_{ij})\in\mathrm{M}_2(\mathcal{O}_K[\![t]\!])$;
\begin{equation}\label{eq:B2}
A|_{t=0}=\begin{pmatrix}
1&a_{12}|_{t=0}\\
0&p
\end{pmatrix};
\end{equation}
\begin{equation}
A\equiv
\begin{pmatrix}\label{eq:B3}
1&0\\
0&0
\end{pmatrix}
\mod{\pi\mathcal{O}_K[\![t]\!]};
\end{equation}
\begin{equation}\label{eq:B6}
\det{A}=p;
\end{equation}
\begin{equation*}
G=
\begin{pmatrix}
0&-1\\
-\pi^2t&0
\end{pmatrix}.
\end{equation*}
The slope multiset of $S_{\bullet}(M_0/tM_0)$ is equal to $\{0,1\}$ by (\ref{eq:B2}). Let $\{s_1,s_2\}$ denote the slope multiset of $S_{\bullet}((M_0)_{\mathcal{E}})$. By (\ref{eq:B6}), $s_1+s_2=1$. By (\ref{eq:B3}), $s_1$ or $s_2$ is equal to $0$, which implies $\{s_1,s_2\}=\{0,1\}$. We will compute $\mathrm{Sol}(M_0)$, which is canonically isomorphic to $\mathrm{Sol}(\mathcal{M}^{\mathrm{Bessel}}_0)$ (Lemma \ref{lem:invariance log-growth}). Put
\[
b=\sum_{i=0}^{\infty}(\pi^2t)^i/i!^2,\ c=-2\sum_{i=0}^{\infty}(1+1/2+\dots+1/i)(\pi^2t)^i/i!^2.
\]
Since $|\pi^i|\le |i!|$ for $i\in\mathbb{N}$, we have $b\in K[\![t]\!]_0$, $c\in K[\![t]\!]_1$. Since we have, for $i=p^r$,
\[
|(1+1/2+\dots+1/i)\pi^{2i}|/|i!^2|=ip^{-2i/(p-1)}/p^{-2(i-1)/(p-1)}=ip^{-2/(p-1)},
\]
$c$ is exactly of log-growth $1$. By a direct calculation (see \cite[\S 5]{DwB}), $\mathrm{Sol}(M_0)$ admits the $K$-basis $\{f_1,f_2\}$ defined by
\[
f_1:(e_1,e_2)\mapsto (-tdb/dt,b),\ f_2:(e_1,e_2)\mapsto (-tdc/dt-b-tdb/dt\cdot\log{t},c+b\cdot\log{t}).
\]
Hence we have
\[
\mathrm{Sol}_{\lambda}(M_0)=
\begin{cases}
0&\text{if}\ \lambda\in (-\infty,0),\\
Kf_1&\text{if}\ \lambda\in [0,1),\\
\mathrm{Sol}(M_0)&\text{if}\ \lambda\in [1,+\infty).
\end{cases}
\]
We have $(b,tdb/dt)\cdot{}^tA=(b,tdb/dt)$ by \cite[(5.6)]{DwB} since ${}^tA$ is equal to $\mathfrak{A}(t)$ in \cite{DwB}. By the above equation and (\ref{eq:B6}), we have $\varphi(f_1)=p^{-1}f_1$. Since the slope multiset of $S_{\bullet}(\mathrm{Sol}(M_0))$ is equal to $\{-1,0\}$, we have
\[
S_{\bullet-1}(\mathrm{Sol}(M_0))=\mathrm{Sol}_{\bullet}(M_0).
\]

$\bullet$ At $s=\infty$

To avoid a complication, we will consider a quadratic extension $\mathcal{R}^{\bd}$ of $\mathcal{R}^{\bd}_{\infty}$ as the base ring (see \cite[Example 6.2.6]{Tsu}), whose quadruple structure satisfies:
\begin{enumerate}
\item[$\bullet$] $\varphi$\text{ is a }$p$\text{-power Frobenius lift such that }$\varphi(t)=2^{p-1}t^p$;
\item[$\bullet$] $\nabla:\mathcal{R}^{\bd}\to\Omega_{\mathcal{R}^{\bd}}=\mathcal{R}^{\bd}\cdot dt/t;f\mapsto tdf/dt\cdot dt/t$;
\item[$\bullet$] $\varphi(dt/t)=p(dt/t)$.
\end{enumerate}
Then the $(\varphi,\nabla)$-module $M_{\infty}=\mathcal{R}^{\bd}e_1\oplus\mathcal{R}^{\bd}e_2$ corresponding to $\mathcal{M}^{\mathrm{Bessel}}_{\infty}$ satisfies
\[
D_{\log}(e_1\ e_2)=(e_1\ e_2)
\begin{pmatrix}
0&2\\
\pi^2/2t^2&0
\end{pmatrix},
\]
where $D_{\log}$ denotes the differential on $M_{\infty}$ with respect to $dt/t$. The $(\varphi,\nabla)$-module $M_{\infty}\otimes_{\mathcal{R}^{\bd}}\mathcal{R}$ over $\mathcal{R}$ is quasi-unipotent in the sense of \cite[Definition 4.1.1 (2)]{Tsu}, but not unipotent. Actually, one can prove $\mathrm{Sol}(M_{\infty})=0$ by using the fact $\{f\in\mathcal{R};D_{\log}^2(f)=(\pi/t)^2f\}=0$. Hence $M_{\infty}$ is beyond the scope of this paper, however, we can treat $M_{\infty}$ in the following ad-hoc manner.

We recall some results in \cite[\S 2]{Doc}. Since $\mathcal{R}^{\Int}$ is a henselian discrete valuation ring, there exists an equivalence of Galois categories
\[
\{\text{finite \'etale extensions of }k((t))\}\to\{\text{finite \'etale extensions of }\mathcal{R}^{\Int}\};F\mapsto\mathcal{R}(F)^{\Int}.
\]
We put $\mathcal{R}(F)^{\bd}=\mathcal{R}(F)^{\Int}\otimes_{\mathcal{R}^{\Int}}\mathcal{R}^{\bd}$. Then, for any $f\in\mathcal{R}(F)^{\bd}$, we can define a norm $|f|_r$ for all sufficiently small $r>0$ and $r=0$ so that if $f\in\mathcal{R}^r$, then $|f|_r$ coincides with $|f|_r$ in Definition \ref{dfn:Robba ring} (i). Via a completion, we obtain the ring $\mathcal{R}(F)$, which is non-canonically isomorphic to the Robba ring with coefficients in $K$. We can endow $\mathcal{R}(F)^{(\bd)}$ with a quadruple structure extending that of $\mathcal{R}^{(\bd)}$, in particular, the log-differential $D_{\log}=td/dt$ on $\mathcal{R}^{(\bd)}$ uniquely extends to a differential $D_{\log}:\mathcal{R}(F)^{(\bd)}\to\mathcal{R}(F)^{(\bd)}$. Moreover, one can endow $\mathcal{R}(F)$ with the log-growth filtration $\mathrm{Fil}_{\bullet}\mathcal{R}(F)$ as in Definition \ref{dfn:log-growth extended}, which satisfies similar properties as in Lemma \ref{lem:property Robba}.

We will construct a finite \'etale extension of $\mathcal{R}^{\bd}$, which kills the ``ramification'' of $M_{\infty}$. Let $F_1,F_2$ be the finite Galois extensions of $k((t))$ defined by the equations $x^2-t=0,y^p-y+1/t=0$ respectively. We put $F=F_1F_2$. Then $F/k((t))$ is a cyclic extension of degree $2p$. We will describe $\mathcal{R}(F)^{\Int}$ as follows. Obviously, we may identify $\mathcal{R}(F_1)^{\Int}$ as $\mathcal{R}^{\Int}[X]/(X^2-t)\mathcal{R}^{\Int}[X]$. Note that $\exp{(p\pi/t)}$ converges on $p^{-1}<|t|$, and
\[
\exp{(p\pi/t)}=\sum_{i=0}^{\infty}(p^i\pi^i/i!)t^{-i}=1+p\pi/t+p^2\pi^2/t^2+\dots\in 1+p\pi\mathcal{R}^{\Int}
\]
We consider the polynomial of the variable $Y$
\[
((\pi Y+1)^p-\exp{(p\pi/t)})/\pi^p=Y^p-\sum_{i=1}^{p-1}\binom{p}{i}p^{-1}\pi^{i-1}Y^i+(\exp{(p\pi/t)}-1)/p\pi\in\mathcal{R}^{\Int}[Y],
\]
which is a lift of $y^p-y+1/t\in k((t))[y]$. Hence we may identify $\mathcal{R}(F_2)^{\Int}$ as $\mathcal{R}^{\Int}[Y]/(((\pi Y+1)^p-\exp{(p\pi/t)})/\pi^p)\mathcal{R}^{\Int}[Y]$. We denote the images of $X,1+\pi Y$ in $\mathcal{R}(F)^{\Int}$ by $t^{1/2},\exp{(\pi/t)}$ respectively, and denote $\exp{(-\pi/t)}$ by $(1+\pi Y)^{-1}$. Then $\mathcal{R}(F)^{\bd}/\mathcal{R}^{\bd}$ is a $2p$-Kummer extension such that
\begin{equation}\label{eq:B5}
\mathrm{Gal}(\mathcal{R}(F)^{\bd}/\mathcal{R}^{\bd})\cong \mu_2(K)\times\mu_p(K);\sigma\mapsto (\sigma(t^{1/2})/t^{1/2},\sigma(\exp{(\pi/t)})/\exp{(\pi/t)}),
\end{equation}
where $\mu_i(K)$ denotes the set of $i$-th roots of unity in $K$. By a formal calculation, we have
\[
D_{\log}(t^{1/2})=1/2,\ D_{\log}(\exp{(\pm\pi/t)})=\mp (\pi/t)\exp{(\pm\pi/t)}.
\]

Instead of $\mathrm{Sol}(M_{\infty})$, we consider the set of $\mathcal{R}(F)$-valued solutions of $M_{\infty}$ (cf. \cite[3.3.6.1]{Chr}), that is, the set of $\mathcal{R}(F)^{\bd}$-linear map $f:M_{\infty}\to \mathcal{R}(F)$ satisfying $f\circ D_{\log}=D_{\log}\circ f$. Then $\mathrm{Sol}(M_{\infty},\mathcal{R}(F))$ is a $K$-vector space of dimension $2$ with the basis $\{f_+,f_-\}$ characterized by
\[
f_{\pm}(e_2)=t^{1/2}\exp{(\pm\pi/t)}u_{\pm}\in \mathcal{R}(F)\text{ with }u_{\pm}=\sum_{i=0}^{\infty}(\pm 1)^i\{((2i-1)!!)^2/(8\pi)^ii!\}t^{i}\in \mathcal{R}
\]
where $(2i-1)!!=\prod_{j=1}^i(2j-1)$ for $i\ge 1$ and $(-1)!!=1$ (see \cite[Example 6.2.6]{Tsu}). T. Nakagawa (\cite[Example 3.11 (ii)]{Nak}) proves that $u_+$, and hence $u_-$, is exactly of log-growth $1/2$ by a Newton polygon argument. We give an alternative proof for the reader's convenience. For $x\in\mathbb{R}$, denote $\lfloor x\rfloor$ the maximum integer less than or equal to $x$. Then $\lfloor 2x\rfloor\ge 2\lfloor x\rfloor$ for $x\in\mathbb{R}$. Hence $v_p((2i)!)\ge 2v_p(i!)$, and, $2v_p((2i)!)-3v_p(i!)\ge v_p((2i)!)/2$, where $v_p(\cdot)=-\log_p{|\cdot|}$. Put $u_i=((2i-1)!!)^2/(8\pi)^ii!\in K$. Then we have
\[
v_p(u_i)=v_p(((2i)!)^2/(2^ii!)^2(8\pi)^ii!)=2v_p((2i)!)-3v_p(i!)-i/(p-1)\ge (v_p((2i)!)-2i/(p-1))/2.
\]
If $2i\in [p^{r-1},p^r)$, then
\[
v_p((2i)!)-2i/(p-1)=\sum_{j=1}^{r-1}{(\lfloor 2i/p^j\rfloor-(2i/p^j))}-\sum_{j=r}^{\infty}2i/p^j\ge\sum_{j=1}^{r-1}(-1)-2i/p^{r-1}(p-1)\ge -r.
\]
Hence $v_p(u_i)\ge-(\lfloor\log_p{2i}\rfloor+1)$ for an arbitrary $i\in\mathbb{N}_{\ge 1}$, with equality when $2i=p^r-1$ by a direct calculation. Thus $u_+$ is exactly of log-growth $1/2$. One can endow $\mathrm{Sol}(M_{\infty},\mathcal{R}(F))$ with a natural structure of a $\varphi$-module over $K$. Thus we obtain the Frobenius slope filtration $S_{\bullet}(\mathrm{Sol}(M_{\infty},\mathcal{R}(F)))$. We also define the growth filtration
\[
\mathrm{Sol}_{\lambda}(M_{\infty},\mathcal{R}(F))=\{f\in \mathrm{Sol}(M_{\infty},\mathcal{R}(F));f(M_{\infty})\subset\mathrm{Fil}_{\lambda}\mathcal{R}(F)\},\ \lambda\in\mathbb{R}.
\]

We will prove
\[
\mathrm{Sol}_{\lambda}(M_{\infty},\mathcal{R}(F))=S_{\lambda-1}(\mathrm{Sol}(M_{\infty},\mathcal{R}(F)))=
\begin{cases}
0&\text{if }\lambda\in (-\infty,1/2),\\
\mathrm{Sol}(M_{\infty},\mathcal{R}(F))&\text{if }\lambda\in [1/2,+\infty).
\end{cases}
\]
Since the slope multiset of the slope filtration $S_{\bullet}(M_{\infty}\otimes_{\mathcal{R}^{\Int}}\mathcal{R})$ in the sense of \cite[Definition 5.1.1]{Tsu} is equal to $\{1/2,1/2\}$ due to Tsuzuki, that of $\mathrm{Sol}(M_{\infty},\mathcal{R}(F))$ is equal to $\{-1/2,-1/2\}$, which implies the assertion for the Frobenius slope filtration. By Nakagawa's result, $\mathrm{Sol}_{1/2}(M_{\infty},\mathcal{R}(F))=\mathrm{Sol}(M_{\infty},\mathcal{R}(F))$. It suffices to prove that $\mathrm{Sol}_{\lambda}(M_{\infty},\mathcal{R}(F))=0$ for any $\lambda\in (-\infty,1/2)$. Suppose the contrary, that is, there exists a non-zero $f\in \mathrm{Sol}_{\lambda}(M_{\infty},\mathcal{R}(F))$ for some $\lambda\in (-\infty,1/2)$. Write $f=c_+f_++c_-f_-$ with $c_{\pm}\in K$. We choose $\zeta\in\mu_p(K),\zeta\neq 1$, and let $\sigma\in\mathrm{Gal}(\mathcal{R}(F)^{\bd}/\mathcal{R}^{\bd})$ be the element corresponding to $(1,\zeta)$ under (\ref{eq:B5}). The Galois group $\mathrm{Gal}(\mathcal{R}(F)^{\bd}/\mathcal{R}^{\bd})$ acts on $\Sol_{\lambda}(M_{\infty},\mathcal{R}(F))$ via the conjugation, and we have
\[
(f^{\sigma}-\zeta^{\mp}f)(e_2)=\pm (\zeta-\zeta^{-1})c_{\pm}t^{1/2}\exp{(\pm\pi/t)}u_{\pm}\in \mathrm{Fil}_{\lambda}\mathcal{R}(F).
\]
Hence either $t^{1/2}\exp{(\pi/t)}u_+$ or $t^{1/2}\exp{(-\pi/t)}u_-$ belongs to $\mathrm{Fil}_{\lambda}\mathcal{R}(F)$. Since $t^{1/2},\exp{(\pm\pi/t)}\in\Fil_0\mathcal{R}(F)$ by $t,\exp{(\pm p\pi/t)}\in\mathcal{R}^{\bd}=\Fil_0\mathcal{R}$, either $u_+$ or $u_-$ belongs to $\mathrm{Fil}_{\lambda}\mathcal{R}$, which contradicts to $u_{\pm}\notin \mathrm{Fil}_{\lambda}\mathcal{R}(F)$.
\end{ex}


\subsection{The bounded quotient}\label{subsec:ex2}

\begin{ex}[the rank one case]
An arbitrary rank one $(\varphi,\nabla)$-module over $K[\![t]\!]_0,\mathcal{R}^{\bd}$, or $\mathcal{E}$ is PBQ by definition. In particular, $M_{\mathcal{E}}^0=0$.

\end{ex}
\begin{ex}[the rank two case]
Let $M$ be a $(\varphi,\nabla)$-module over $\mathcal{E}$ of rank two. Then the following are equivalent:
\begin{enumerate}
\item[(a)] $M$ is not PBQ;
\item[(b)] $M$ is bounded and not pure as a $\varphi$-module over $\mathcal{E}$.
\end{enumerate}
In fact, since any rank one $\varphi$-module over $\mathcal{E}$ is pure, (a) is equivalent to say that $M/M^0$ is of rank two, i.e., $M^0=0$, and is not pure, which implies (b). Conversely, if (b) holds, then $M^0=0$, hence, (a) holds. Note that (a) is equivalent to

\ (a)' $M\spcheck$ is not PBQ

by the dual-invariance of the condition (b).

Thus the following are equivalent:
\begin{enumerate}
\item[(i)] $M$ is PBQ;
\item[(ii)] $M\spcheck$ is PBQ;
\item[(iii)] $M$ is either non-bounded or pure as a $\varphi$-module over $\mathcal{E}$.
\end{enumerate}

\end{ex}
\begin{ex}
With notation and assumption as in Example 6.1.3, put $M=M_{\mu}$. Since $M_{\mathcal{E}}$ is not pure as a $\varphi$-module, $M$ and $M\spcheck$ are PBQ by Example 6.2.2. More precisely, we have $M_{\mathcal{E}}^0=\mathcal{E}e_1$ and $(M\spcheck)_{\mathcal{E}}^0=\mathcal{E}e_2\spcheck$. We will verify the first equality. Since $M_{\mathcal{E}}/\mathcal{E}e_1\cong\mathcal{E}(q^{\mu})$ implies $M_{\mathcal{E}}^0\subset\mathcal{E}e_1$ by Theorem \ref{thm:Rob}. Suppose $M_{\mathcal{E}}^0=0$. Then $M$ is trivial as a $\nabla$-module over $K[\![t]\!]_0$ by Christol's transfer theorem (\cite[Proposition 4.3]{CT}), which contradicts to $\dim_K \Sol_0(M)=1<2$. By a similar argument, we obtain the second equality.
\end{ex}
\begin{ex}
With notation and assumption as in Example 6.1.4, we have $(M_{\mu,\delta})_{\mathcal{E}}^0=\mathcal{E}e_1$. Moreover, $M_{\mu,\delta}$ is not PBQ. Put $M=M_{\mu,\delta}$ for simplicity. Recall that there exists an exact sequence of $(\varphi,\nabla)$-modules over $K[\![t]\!]_0$
\begin{equation}\label{ex:eq1}
0\to K[\![t]\!]_0\to M_{\mu}\oplus M_{\delta}\to M\to 0.
\end{equation}
By applying Lemma \ref{lem:exact} to (\ref{ex:eq1}) with $\lambda=0$, we obtain an exact sequence of $(\varphi,\nabla)$-modules over $\mathcal{E}$
\[
\mathcal{E}\to (M_{\mu})_{\mathcal{E}}/(M_{\mu})_{\mathcal{E}}^0\oplus (M_{\delta})_{\mathcal{E}}/(M_{\delta})_{\mathcal{E}}^0\to M_{\mathcal{E}}/M_{\mathcal{E}}^0\to 0,
\]
where the first arrow is the zero map by Example 6.2.3. Hence $\dim_{\mathcal{E}}M_{\mathcal{E}}^0=1$. Since $M_{\mathcal{E}}/\mathcal{E}e_1\cong\mathcal{E}(q^{\mu})\oplus\mathcal{E}(q^{\delta})$ is trivial, in particular, bounded, we have $M_{\mathcal{E}}^0\subset\mathcal{E}e_1$. Thus $M_{\mathcal{E}}^0=\mathcal{E}e_1$. By the above isomorphism and the assumption $\mu<\delta$, $M$ is not PBQ.

We also prove $(M\spcheck)_{\mathcal{E}}^0=\mathcal{E}e_2\spcheck\oplus\mathcal{E}e_3\spcheck$, in particular, $M\spcheck$ is PBQ. Since $(M\spcheck)_{\mathcal{E}}^0/(\mathcal{E}e_2\spcheck\oplus\mathcal{E}e_3\spcheck)\cong\mathcal{E}$ is trivial, and hence, bounded, we have $(M\spcheck)_{\mathcal{E}}^0\subset\mathcal{E}e_2\spcheck\oplus\mathcal{E}e_3\spcheck$. By applying Lemma \ref{lem:exact} to the dual of (\ref{ex:eq1}), there exists an exact sequence of $(\varphi,\nabla)$-modules over $\mathcal{E}$
\[
0\to (M\spcheck)_{\mathcal{E}}/(M\spcheck)_{\mathcal{E}}^0\to (M_{\mu}\spcheck)_{\mathcal{E}}/(M_{\mu}\spcheck)_{\mathcal{E}}^0\oplus (M_{\delta}\spcheck)_{\mathcal{E}}/(M_{\delta}\spcheck)_{\mathcal{E}}^0\to\mathcal{E}\to 0.
\]
Hence $\dim_{\mathcal{E}}(M\spcheck)_{\mathcal{E}}^0=2$, which implies the assertion.

Finally, note that $M_{\mathcal{E}}$ is indecomposable in the category of $(\varphi,\nabla)$-modules over $\mathcal{E}$; if not, there exists a non-trivial decomposition $M_{\mathcal{E}}\spcheck=N_1\oplus N_2$. Hence $M_{\mathcal{E}}\spcheck/(M_{\mathcal{E}}\spcheck)^0\cong N_{1}/N_{1}^0\oplus N_{2}/N_{2}^0$ (Lemma \ref{lem:lgE3}) is of dimension greater than or equal to $2$ by Lemma \ref{lem:lgE1} (ii), which is a contradiction.

\end{ex}
\begin{ex}
With notation and assumption as in Example 6.1.5, put $N=K[\![t]\!]_0\mathbf{e}_1\oplus K[\![t]\!]_0\mathbf{e}_2$. Then $N$ is a $(\varphi,\nabla)$-submodule of $M$, and the quotient $Q=M/N\cong K[\![t]\!]_0(q^{\mu})\oplus K[\![t]\!]_0$ is trivial as a $\nabla$-module. We will prove $M_{\mathcal{E}}^0=N_{\mathcal{E}}$; as a corollary, $M$ is not PBQ. Since $Q_{\mathcal{E}}$ is bounded (even trivial), $M_{\mathcal{E}}^0\subset N_{\mathcal{E}}$. We have $N_{\mathcal{E}}^0=\mathcal{E}\mathbf{e}_1$ by a similar argument as in Example 6.1.3. Hence $\mathbf{e}_1\in N_{\mathcal{E}}^0\subset M_{\mathcal{E}}^0$ by Lemma \ref{lem:lgE2} (i). Suppose $M_{\mathcal{E}}^0\neq N_{\mathcal{E}}$. Then $M_{\mathcal{E}}^0=\mathcal{E}\mathbf{e}_1$, hence, the bounded quotient $P:=M_{\mathcal{E}}/M_{\mathcal{E}}^0$ contains a copy of $(M_{\mu})_{\mathcal{E}}$ (generated by the images of $\mathbf{e}_2,\mathbf{e}_3$). Since $(M_{\mu})_{\mathcal{E}}^0\neq 0$ by Example 6.2.3, we have $P^0\neq 0$ by Lemma \ref{lem:lgE2} (i), which contradicts to the boundedness of $P$.
\end{ex}
\begin{ex}[direct sum]
Put $M=K[\![t]\!]_0\oplus K[\![t]\!]_0(q)$ as in Example 6.1.2. Then $M_{\mathcal{E}}$ is bounded and not PBQ. More generally, we have the following criterion:

\begin{lem}
Let $R$ be either $K[\![t]\!]_0,\mathcal{R}^{\bd}$, or $\mathcal{E}$. Let $M_1,\dots,M_n$ be non-zero $(\varphi,\nabla)$-modules over $R$. The following are equivalent.
\begin{enumerate}
\item[(i)] The direct sum $\oplus_{i=1}^nM_i$ is PBQ.
\item[(ii)] $M_1,\dots,M_n$ are PBQ, and $\lambda_{\max}(M_1\otimes_R\mathcal{E})=\dots=\lambda_{\max}(M_n\otimes_R\mathcal{E})$.
\end{enumerate}
\end{lem}

\begin{proof}
By definition, we may assume $R=\mathcal{E}$.

\noindent (i)$\Rightarrow$(ii) By regarding $M_i$ as a quotient of $\oplus_{i=1}^nM_i$, $M_i$ is PBQ by Corollary \ref{cor:qt inv of PBQ}, and $\lambda_{\max}(M_i)=\lambda_{\max}(\oplus_{i=1}^nM_i)=\max_{i=1,\dots,n}\lambda_{\max}(M_i)$ by Slope criterion (Proposition \ref{prop:key}).

\noindent (ii)$\Rightarrow$(i) Since $M_i/M_i^0$ is pure of slope $\lambda_{\max}(M_i)$ by Corollary \ref{cor:new2}, $(\oplus_{i=1}^nM_i)/(\oplus_{i=1}^nM_i)^0\cong\oplus_{i=1}^n(M_i/M_i^0)$ (Lemma \ref{lem:lgE3} (i)) is also pure by assumption.
\end{proof}

\end{ex}
\begin{ex}[a non-PBQ submodule of a PBQ module]
Let $M$ be a PBQ $(\varphi,\nabla)$-module over $K[\![t]\!]_0$, and $N$ a $(\varphi,\nabla)$-submodule of $M$. If $M$ is of rank one or two, then $N$ is always PBQ (obvious). This does not hold for $M$ of a higher rank in general: for example, let $M=M_{\mu,\delta}\spcheck$ with notation as in Example 6.1.4. Then $N=K[\![t]\!]_0e_2\spcheck\oplus K[\![t]\!]_0e_3\spcheck\cong K[\![t]\!]_0(q^{-\lambda})\oplus K[\![t]\!]_0(q^{-\mu})$ is not PBQ.

\end{ex}
\begin{ex}[symmetric square]
Let $M$ be a $(\varphi,\nabla)$-module over $\mathcal{E}$ of dimension two, which is not bounded. We will prove that $\mathrm{Sym}^2M$ is PBQ, and $M^{\otimes 2}$ is not PBQ. By assumption and Lemma \ref{lem:lgE1} (ii), we have $\dim_{\mathcal{E}}M^0=1$. Hence we may apply Example 6.1.6 to $\tau^*M$. Thus $\dim_{\mathcal{E}}\Sol_0(\mathrm{Sym}^2M)=1$, which implies the first assertion. Similarly, we have $\Sol_0(M^{\otimes 2})\supsetneq S_{-2\lambda_{\max}}(\Sol(M^{\otimes 2}))\neq 0$, where $\lambda_{\max}$ denote the maximum Frobenius slope of $M_{\mathcal{E}}$. Hence $\Sol_0(M^{\otimes 2})$ is not pure as a $\varphi$-module over $\mathcal{E}$, which implies the second assertion.
\end{ex}

\section{Chiarellotto-Tsuzuki conjecture and Main Theorem}\label{sec:statement}

In this section, we recall the statement of Chiarellotto-Tsuzuki conjecture over $K[\![t]\!]_0$ (Conjecture \ref{conj:CT}), and an analogous statement over $\mathcal{E}$ (Theorem \ref{conj:CTgen}). We also state the main theorem in this paper (Theorem \ref{conj:bd}), which is an analogue of Chiarellotto-Tsuzuki conjecture over $\mathcal{R}^{\bd}$. In \S \ref{subsec:bc}, we also prove that the log-growth filtration is compatible with the base change of the coefficient $K$.

\subsection{Chiarellotto-Tsuzuki conjecture over $K[\![t]\!]_0$}\label{subsec:statement1}

Without any assumption on $M$, we have the following relation between the log-growth filtration and Frobenius slope filtration.

\begin{prop}[{\cite[Theorem 2.3 (2)]{CT2}}]\label{prop:CT1}
Let $M$ be a $(\varphi,\nabla)$-module over $K[\![t]\!]_0$. Let $\lambda_{\max}$ be the maximum Frobenius slope of $M_{\mathcal{E}}$. Then
\[
V(M)^{\lambda}\subset (S_{\lambda-\lambda_{\max}}(V(M\spcheck)))^{\perp},
\]
or, equivalently,
\[
\Sol_{\lambda}(M)\supset S_{\lambda-\lambda_{\max}}(\Sol(M))
\]
for an arbitrary real number $\lambda$.
\end{prop}
\begin{proof}
Since the first relation is verified in the reference, it suffices to prove the equivalence between the two relations. The two canonical perfect pairings
\[
V(M)\otimes_K\Sol(M)\to K,\ V(M)\otimes_KV(M\spcheck)\to K
\]
induce a canonical isomorphism $V(M\spcheck)\cong \Sol(M)$. Hence the two relations are equivalent by taking the orthogonal parts.
\end{proof}

The following conjecture due to Chiarellotto and Tsuzuki, first stated in \cite[Conjecture 6.9]{CT} as a slightly different (though equivalent) form, is denoted by $\mathbf{LGF}_{K[\![t]\!]_0}$ as in \cite{CT2}.

\begin{CT}[{\cite[Conjecture 2.5]{CT2}}]\label{conj:CT}
Let $M$ be a $(\varphi,\nabla)$-module over $K[\![t]\!]_0$.
\begin{enumerate}
\item (\cite[Theorem 3.7 (i)]{Ohk}) All slopes of the log-growth filtration $V(M)^{\bullet}$ are rational, and $V(M)^{\lambda}=\cup_{\mu>\lambda}V(M)^{\mu}$ for any $\lambda$.
\item Let $\lambda_{\max}$ be the maximum Frobenius slope of $M_{\mathcal{E}}$. If $M$ is PBQ, then
\[
V(M)^{\lambda}=(S_{\lambda-\lambda_{\max}}(V(M)\spcheck))^{\perp},
\]
or, equivalently,
\[
\Sol_{\lambda}(M)=S_{\lambda-\lambda_{\max}}(\Sol(M)).
\]
\end{enumerate}
\end{CT}

One may formulate an analogue of Chiarellotto-Tsuzuki conjecture \ref{conj:CT} (even Proposition \ref{prop:CT1}) for a log-$(\varphi,\nabla)$-module over $K[\![t]\!]_0$.

We recall some known results on Chiarellotto-Tsuzuki conjecture \ref{conj:CT} in order of time; in the rank one case, the conjecture is trivial since $M$ is trivial as $\nabla$-module over $K[\![t]\!]_0$ (a very special case of \cite[Theorem 1]{Dw}, or, Example 6.1.1). In \cite{book}, the case of Picard-Fuchs module associated to Legendre family of elliptic curves is discussed (see \cite[\S 7.4]{CT} for details). Both parts (i) and (ii) are verified when $M$ is of rank two (\cite[Theorem 7.1 (2)]{CT}), or, when the bounded quotient of $M_{\mathcal{E}}$ is trivial as a $\nabla$-module over $\mathcal{E}$ (\cite[Theorem 6.5]{CT2}). As a consequence of the second case, Theorem \ref{conj:CTgen} below is proved. It is also proved that part (ii) implies part (i) (\cite[Proposition 7.3]{CT2}, or, Proposition \ref{prop:(ii) implies (i)} below). Recently, part (i) is proved unconditionally by the author in \cite{Ohk}. Part (ii) is also proved when the number of Frobenius slopes of $M_{\mathcal{E}}$ is equal to $2$ (\cite[Theorem 3.7 (ii)]{Ohk}).
\subsection{Chiarellotto-Tsuzuki conjecture over $\mathcal{E}$}\label{subsec:statement2}

Chiarellotto and Tsuzuki establish an analogous theory over $\mathcal{E}$, which is technically easier thanks to Theorem \ref{thm:Rob}.

\begin{prop}[{\cite[Theorem 2.3 (1)]{CT2}}]\label{prop:CT2}
Let $M$ be a $(\varphi,\nabla)$-module over $\mathcal{E}$. Let $\lambda_{\max}$ be the maximum Frobenius slope of $M$. Then
\[
M^{\lambda}\subset (S_{\lambda-\lambda_{\max}}(M\spcheck))^{\perp},
\]
or, equivalently,
\[
\Sol_{\lambda}(M)\supset S_{\lambda-\lambda_{\max}}(\Sol(M))
\]
for an arbitrary real number $\lambda$.
\end{prop}

The following theorem is previously referred to as the conjecture $\mathbf{LGF}_{\mathcal{E}}$ in \cite{CT2}.
\begin{thm}[{\cite[Theorem 7.1]{CT2}}]\label{conj:CTgen}
Let $M$ be a $(\varphi,\nabla)$-module over $\mathcal{E}$.
\begin{enumerate}
\item All slopes of the log-growth filtration $M^{\bullet}$ are rational, and $M^{\lambda}=\cup_{\mu>\lambda}M^{\mu}$ for any $\lambda$.
\item Let $\lambda_{\max}$ be the maximum Frobenius slope of $M$. If $M$ is PBQ, then
\[
M^{\lambda}=(S_{\lambda-\lambda_{\max}}(M\spcheck))^{\perp},
\]
or, equivalently,
\[
\Sol_{\lambda}(M)= S_{\lambda-\lambda_{\max}}(\Sol(M))
\]
for an arbitrary real number $\lambda$.
\end{enumerate}
\end{thm}

\subsection{Main Theorem --- a generalization of Chiarellotto-Tsuzuki conjecture over $\mathcal{R}^{\bd}$}

In \cite{Ohk}, the author gives an attempt to generalize Chiarellotto-Tsuzuki theory to $\mathcal{R}^{\bd}$. One of the advantage to work with $\mathcal{R}^{\bd}$ is that one can find a nice cyclic vector (``generic cyclic vector'') for a Frobenius structure possibly after a finite \'etale extension (\cite[Theorem 5.2]{Ohk}).

\begin{prop}[{\cite[Prospoition 4.18 (i)]{Ohk}}]\label{prop:main}
Let $M$ be a $(\varphi,\nabla)$-module over $\mathcal{R}^{\bd}$ solvable in $\mathcal{R}_{\log}$. Let $\lambda_{\max}$ be the maximum Frobenius slope of $M_{\mathcal{E}}$. Then
\[
V(M)^{\lambda}\subset (S_{\lambda-\lambda_{\max}}(V(M\spcheck)))^{\perp},
\]
or, equivalently,
\[
\Sol_{\lambda}(M)\supset S_{\lambda-\lambda_{\max}}(\Sol(M))
\]
for an arbitrary real number $\lambda$.
\end{prop}

The following theorem is the main result in this paper; we state in a form compatible with Chiarellotto-Tsuzuki conjecture \ref{conj:CT} while part (i) is already proved in \cite{Ohk}.

\begin{thm}\label{conj:bd}
Let $M$ be a $(\varphi,\nabla)$-module over $\mathcal{R}^{\bd}$ solvable in $\mathcal{R}_{\log}$.
\begin{enumerate}
\item (\cite[Theorem 4.19 (i)]{Ohk}) All slopes of the log-growth filtration $V(M)^{\bullet}$ are rational, and $V(M)^{\lambda}=\cup_{\mu>\lambda} V(M)^{\mu}$ for any $\lambda$.
\item (Main Theorem) Let $\lambda_{\max}$ be the maximum Frobenius slope of $M_{\mathcal{E}}$. If $M$ is PBQ, then
\[
V(M)^{\lambda}=(S_{\lambda-\lambda_{\max}}(V(M\spcheck)))^{\perp},
\]
or, equivalently,
\[
\Sol_{\lambda}(M)=S_{\lambda-\lambda_{\max}}(\Sol(M)).
\]
\end{enumerate}
\end{thm}

\begin{prop}
If Theorem \ref{conj:bd} (ii) holds, then Chiarellotto-Tsuzuki conjecture \ref{conj:CT} (ii) holds. Moreover, a log-analogue of Chiarellotto-Tsuzuki conjecture \ref{conj:CT} (even Proposition \ref{prop:CT1}) also holds.
\end{prop}
\begin{proof}
It follows from Lemma \ref{lem:invariance log-growth}.
\end{proof}

In the rest of part I, we will prove Theorem \ref{conj:bd} (ii): in the proof, we do not use Theorem \ref{conj:bd} (i) though we use Proposition \ref{prop:main}. We will deduce part (i) from part (ii) in Theorem \ref{conj:bd} (Proposition \ref{prop:(ii) implies (i)}).

\begin{rem}\label{rem:quasi-unipotent}
As Example \ref{ex:Bessel} suggests, it is natural to generalize Theorem \ref{conj:bd} to an arbitrary $(\varphi,\nabla)$-module $M$ over $\mathcal{R}^{\bd}$. By the $p$-adic local monodromy theorem (\cite[Theorem 20.1.4]{pde}), $M$ is quasi-unipotent, i.e., unipotent after tensoring a suitable extension $\mathcal{R}'$ over $\mathcal{R}$. To trivialize the differential structure, we need a suitable log version of $\mathcal{R}'$. Then, to apply our method of the proof of Theorem \ref{conj:bd} (ii), we also need analogues of the results in \S\S \ref{sec:Robba}, \ref{sec:lg}, for example, Lemma \ref{lem:inj}. Since the author does not have concise answers to these problems, we work under the solvability assumption on $M$ in this paper.
\end{rem}

\subsection{Base change of coefficient}\label{subsec:bc}

The log-growth filtration is compatible with a na\"ive base change of the coefficient field $K$ as follows. This allows us to assume that $k$ is algebraically closed (by replacing $K$ by the completion of the maximal unramified extension), hence, we can apply Dieudonn\'e-Manin theorem to $\Sol_{\bullet}(M)$. The reader may skip here by tacitly assuming that $k$ is algebraically closed in the proof of Theorem \ref{conj:bd} (ii).

\begin{lem}[{base change of coefficient for $K[\![t]\!]_0$, \cite[Proposition 1.10]{CT}}]\label{lem:coefficient power}
Let $L/K$ be an extension of complete discrete valuation fields, on which $\varphi_K$ extends isometrically. Let $M$ be a $(\varphi,\nabla)$-module over $K[\![t]\!]_0$.
\begin{enumerate}
\item There exist canonical isomorphisms of $\varphi$-modules over $L$
\[
L\otimes_K V(M)\to V(M\otimes_{K[\![t]\!]_0}L[\![t]\!]_0),
\]
\[
L\otimes_K\Sol(M)\to \Sol(M\otimes_{K[\![t]\!]_0}L[\![t]\!]_0).
\]
Moreover, the isomorphisms are compatible with the canonical pairings.
\item The above isomorphisms respect the growth filtrations, i.e.,
\[
L\otimes_K V(M)^{\bullet}\cong V(M\otimes_{K[\![t]\!]_0}L[\![t]\!]_0)^{\bullet},
\]
\[
L\otimes_K \Sol_{\bullet}(M)\cong \Sol_{\bullet}(M\otimes_{K[\![t]\!]_0}L[\![t]\!]_0).
\]
\end{enumerate}
\end{lem}
\begin{proof}
The assertion for $\Sol$ is proved in the reference, and the assertion for $V$ follows by duality.
\end{proof}

\begin{lem}[{base change of coefficient for $\mathcal{E}$}]\label{lem:coefficient Amice}
Let $L/K$ be an extension of complete discrete valuation fields, on which $\varphi_K$ extends isometrically. Let $M$ be a $(\varphi,\nabla)$-module over $\mathcal{E}$. Let $\mathcal{E}_L$ denote the Amice ring over $L$. Then there exists a canonical isomorphism of filtrations of $(\varphi,\nabla)$-modules over $\mathcal{E}_L$
\[
M^{\bullet}\otimes_{\mathcal{E}}\mathcal{E}_L\cong (M\otimes_{\mathcal{E}}\mathcal{E}_L)^{\bullet}.
\]
\end{lem}
\begin{proof}
Let $\iota:\mathcal{E}\to\mathcal{E}_L$ denote the canonical map. Let $\tau_L:\mathcal{E}_L\to\mathcal{E}_L[\![X-t]\!]_0$ denote a transfer morphism defined similarly to $\tau$. Then we have a commutative digram
\[\xymatrix{
\mathcal{E}_L\ar[r]^(.35){\tau_L}&\mathcal{E}_L[\![X-t]\!]_0\\
\mathcal{E}\ar[r]^(.35){\tau}\ar[u]_{\iota}&\mathcal{E}[\![X-t]\!]_0,\ar[u]_{\iota}
}\]
where $\iota(\sum_ia_i(X-t)^i)=\sum_i\iota(a_i)(X-t)^i$ for an arbitrary bounded sequence $\{a_i\}$ of $\mathcal{E}$. For a $(\varphi,\nabla)$-module $N$ over $\mathcal{E}$, we have a canonical isomorphism
\[
\Sol(\iota^*N)\cong \iota^*\Sol(N),
\]
which also induces an isomorphism of the growth filtrations $\Sol_{\bullet}(\iota^*N)\cong \iota^*\Sol_{\bullet}(N)$, by applying Lemma \ref{lem:coefficient power} to $\mathcal{E}_L/\mathcal{E}$ and using the above diagram. By using these isomorphisms, we have
\[
\Sol(\iota^*M/\iota^*(M^{\bullet}))\cong \Sol(\iota^*(M/M^{\bullet}))\cong\iota^*\Sol(M/M^{\bullet})\cong\iota^*\Sol_{\bullet}(M)\cong \Sol_{\bullet}(\iota^*M).
\]
By the uniqueness of the log-growth filtration of $\iota^*M$ (Theorem \ref{thm:Rob}), we have $\iota^*(M^{\bullet})=(\iota^*M)^{\bullet}$.
\end{proof}

\begin{lem}[{cf. \cite[Prospoition 2.1]{CT2}}]\label{lem:coefficient log-growth filtration}
Let $L/K$ be an extension of complete discrete valuation fields, on which $\varphi_K$ extends isometrically. Let $\mathcal{R}_{L}^{(\bd)}$ be the (bounded) Robba ring over $L$. Also, let $\mathcal{R}_{L,\log}$ denote $\mathcal{R}_{L}[\log{t}]$. We can define the log-growth filtrations $\Fil_{\bullet}\mathcal{R}_{L}$ and $\Fil_{\bullet}\mathcal{R}_{L,\log}$ on $\mathcal{R}_{L}$ and $\mathcal{R}_{L,\log}$ respectively as in \S \ref{sec:lg}. Let $\star$ denote either (empty) or $\log$.
\begin{enumerate}
\item There exists a canonical injective ring homomorphism
\[
L\otimes_K\mathcal{R}_{\star}\to\mathcal{R}_{L,\star}.
\]
\item Under the above homomorphism,
\[
L\otimes_K\Fil_{\bullet}\mathcal{R}_{\star}\subset \Fil_{\bullet}\mathcal{R}_{L,\star}.
\]
\item Let $c_1,\dots,c_n\in L$ be $K$-linearly independent elements and $r_1,\dots,r_n\in\mathcal{R}_{\star}$. If
\[
\sum_{i=1}^nc_ir_i\in \Fil_{\lambda}\mathcal{R}_{L,\star}\text{ for }\lambda\in\mathbb{R},
\]
then we have $r_i\in \Fil_{\lambda}\mathcal{R}_{\star}$.
\end{enumerate}
\end{lem}
\begin{proof}
We have only to prove the case $\star=(empty)$.
\begin{enumerate}
\item The existence of a morpshim is obvious. It suffices to prove the injectivity. Assume $\sum_{j=1}^nc_j\otimes r_j$, $c_j\in L$, $r_j\in\mathcal{R}$ is in the kernel of the map, that is, $\sum_{j=1}^nc_jr_j=0$ in $\mathcal{R}_{L}$. We may assume that $\{c_j\}$ is $K$-linearly independent. Write $r_j=\sum_{i\in\mathbb{Z}}a^j_it^i$. By $\sum_{j=1}^nc_jr_j=\sum_{i\in\mathbb{Z}}(\sum_{j=1}^nc_ja^j_i)t^i=0$, we have $\sum_{j=1}^nc_ja^j_i=0$, which implies $a^j_i=0$ for all $i,j$.
\item It follows by definition.
\item With notation as in (i),
\[
\sum_{j=1}^nc_j\otimes r_j=\sum_{i\in\mathbb{Z}}\Big(\sum_{j=1}^nc_ja^j_i\Big)t^i\in \Fil_{\lambda}\mathcal{R}_{L}.
\]
Therefore,
\[
\sup_{i\ge 1}i^{-\lambda}|\sum_{j=1}^nc_ja^j_i|<\infty
\]
by Lemma \ref{lem:Taylor}. By \cite[Theorem 1.3.6]{pde}, there exists a constant $C$ such that
\[
\sup\{|\alpha_1|,\dots,|\alpha_n|\}\le C|\sum_{j=1}^nc_j\alpha_j|\ \forall \alpha_1,\dots,\alpha_n\in K.
\]
Therefore,
\[
\sup_{j=1,\dots,n}\sup_{i\ge 1}i^{-\lambda}|a^j_i|=\sup_{i\ge 1}\sup_{j=1,\dots,n}i^{-\lambda}|a^j_i|\le\sup_{i\ge 1}i^{-\lambda}C|\sum_{j=1}^nc_ja^j_i|<\infty,
\]
i.e., $r_j\in \Fil_{\lambda}\mathcal{R}$ for all $j$.
\end{enumerate}
\end{proof}

\begin{lem}[{base change of coefficient for $\mathcal{R}^{\bd}$}]\label{lem:coefficient Robba}
Let notation be as in Lemma \ref{lem:coefficient log-growth filtration}. Let $M$ be a $(\varphi,\nabla)$-module over $\mathcal{R}^{\bd}$ solvable in $\mathcal{R}_{\log}$. Let $M'=M\otimes_{\mathcal{R}^{\bd}}\mathcal{R}^{\bd}_{L}$ be the $(\varphi,\nabla)$-module over $\mathcal{R}^{\bd}_{L}$.
\begin{enumerate}
\item $M'$ is solvable in $\mathcal{R}_{L,\log}$.
\item There exist canonical isomorphisms
\[
L\otimes_K V(M)\to V(M'),
\]
\[
L\otimes_K\Sol(M)\to \Sol(M').
\]
Moreover, the isomorphisms are compatible with the canonical pairings.
\item The above isomorphisms respect the growth filtrations, i.e.,
\[
L\otimes_K V(M)^{\bullet}\cong V(M')^{\bullet},
\]
\[
L\otimes_K \Sol_{\bullet}(M)\cong \Sol_{\bullet}(M').
\]
\end{enumerate}
\end{lem}
\begin{proof}
\begin{enumerate}
\item It follows from Lemma \ref{lem:solv=uni}.
\item The injection $L\otimes_K\mathcal{R}_{\log}\to\mathcal{R}_{L,\log}$ (Lemma \ref{lem:coefficient log-growth filtration} (i)) induces an injection $L\otimes_KV(M)\to V(M')$, which is an isomorphism by comparing dimensions. We define a $K$-linear map $\Sol(M)\to \Sol(M')$ by sending $f:M\to\mathcal{R}_{\log}$ to the composition
\[
M\otimes_{\mathcal{R}^{\bd}}\mathcal{R}_{L}^{\bd}\xrightarrow[]{f\otimes \id}\mathcal{R}_{\log}\otimes_{\mathcal{R}^{\bd}}\mathcal{R}_{L}^{\bd}\to\mathcal{R}_{L,\log},
\]
where the last map is the multiplication map. Since the canonical map $\Sol(M)\otimes_KV(M)\to \Sol(M')\otimes_{L}V(M')$ commutes with the canonical pairings, we obtain an isomorphism $L\otimes_K\Sol(M)\cong \Sol(M')$.
\item By duality, we have only to prove the second isomorphism. By Lemma \ref{lem:property Robba} (i), we have $\mathcal{R}_{L}^{\bd}\cdot \Fil_{\lambda}\mathcal{R}_{\log}\subset\mathcal{R}_{L}^{\bd}\cdot \Fil_{\lambda}\mathcal{R}_{L,\log}\subset \Fil_{\lambda}\mathcal{R}_{L,\log}$. Hence $L\otimes_K\Sol_{\lambda}(M)\subset \Sol_{\lambda}(M')$ under the second isomorphism in (ii). Let $c_1,\dots,c_n$ be a $K$-basis of $L$. Let $f\in \Sol_{\lambda}(M')$. Write $f=\sum_ic_if_i$ with $f_i\in\Sol(M)$. Since $f(m)\in \Fil_{\lambda}\mathcal{R}_{L,\log}$ for all $m\in M$, we have $f_i(m)\in \Fil_{\lambda}\mathcal{R}_{\log}$ for all $i$ and $m\in M$ by Lemma \ref{lem:coefficient log-growth filtration} (iii). Hence $f_i\in \Sol_{\lambda}(M)$ for all $i$, which implies $f\in L\otimes_K\Sol_{\lambda}(M)$.
\end{enumerate}
\end{proof}

\section{Slope criterion}\label{sec:criterion}

This section is the most important part in this paper though it is quite elementary. We prove Slope criterion (Proposition \ref{prop:key}), which is a key ingredient in the proof of Theorem \ref{conj:bd} (ii), that converts the PBQ hypothesis into a suitable form to prove the conjecture: note that we need only the ``easier'' part (i)$\Rightarrow$(ii).

\begin{notation}
For a $\varphi$-module $M$ over $\mathcal{E}$, denote the maximum (resp. minimum) Frobenius slope by $\lambda_{\max}(M)$ (resp. $\lambda_{\min}(M)$).
\end{notation}

\begin{prop}[{Slope criterion}]\label{prop:key}
For $M$ a $(\varphi,\nabla)$-module over $\mathcal{E}$, we consider the following conditions.
\begin{enumerate}
\item $M$ is PBQ.
\item For any non-zero quotient $Q$ of $M$ as a $(\varphi,\nabla)$-module over $\mathcal{E}$, we have $\lambda_{\max}(M)=\lambda_{\max}(Q)$.
\end{enumerate}
Then (i) implies (ii). Moreover, if $k$ is perfect, then (ii) implies (i).
\end{prop}

\begin{lem}[{\cite[Proposition 4.3]{CT2} when $k$ is perfect}]\label{lem:split}
Let
\[
0\to M'\to M\to M''\to 0
\]
be an exact sequence of $(\varphi,\nabla)$-modules over $\mathcal{E}$. If $\lambda_{\min}(M')>\lambda_{\max}(M'')$, then the above exact sequence splits as $(\varphi,\nabla)$-modules over $\mathcal{E}$.
\end{lem}
\begin{proof}
By using a canonical isomorphism $\Ext^1(M'',M')\cong \Ext^1(\mathcal{E},M'\otimes (M'')\spcheck)$ of Yoneda extension groups in the category of $(\varphi,\nabla)$-modules (see \cite[Lemma 5.3.3]{pde}: though the construction is done in the category of $\nabla$-modules, it also works in the category of $(\varphi,\nabla)$-modules), we may assume that $M''=\mathcal{E}$. Since all Frobenius slopes of $M'$ are strictly greater than $0$ by assumption, the operator norm of $\varphi$ on $M''$ is strictly greater than $1$. In particular, the map $1-\varphi:M'\to M'$ is bijective with inverse $1+\varphi+\varphi^2+\dots$. Hence $H^1_{\varphi}(M')=\mathrm{coker}\{1-\varphi:M'\to M'\}=0$, which implies the splitting of the given exact sequence as $\varphi$-modules. We also have $H^0_{\varphi}(M')=\mathrm{ker}\{1-\varphi:M'\to M'\}=0$.

Let $e_1,\dots,e_n$ be a basis of $M'$, and $e_{n+1}\in M$ a lift of $1\in\mathcal{E}$. Let
\[
\begin{pmatrix}
A_{11}&0\\
0&1
\end{pmatrix},
\begin{pmatrix}
G_{11}&G_{12}\\
0&0
\end{pmatrix}
\]
denote the $(n,1)$-block matrix presentation of $M$ with respect to $\{e_1,\dots,e_{n+1}\}$. The $(1,2)$-component of the compatibility condition in \S \ref{subsec:phinabla2} implies $\omega G_{12}=A_{11}\varphi(\omega G_{12})$. Therefore $v=(e_1,\dots,e_n)\omega G_{12}\in M'$ satisfies $v\in H^0_{\varphi}(M')=0$. Hence $G_{12}=0$, which implies the assertion.
\end{proof}

The following theorem, which plays an important role in \cite{CT2}, is only used to prove (ii)$\Rightarrow$(i) in Proposition \ref{prop:key} when $k$ is perfect.

\begin{thm}[{Splitting theorem, \cite[Theorem 4.1]{CT2}}]\label{thm:split}
A $(\varphi,\nabla)$-module $M$ over $\mathcal{E}$ is bounded if and only if $M$ is a direct sum of pure $(\varphi,\nabla)$-modules, that is, there exists an isomorphism
\[
M\cong\oplus_{i=1,\dots,d}(S_{\mu_i}(M)/\cup_{\mu<\mu_i}S_{\mu}(M))
\]
as $(\varphi,\nabla)$-modules, where $\mu_1<\dots<\mu_d$ denote the Frobenius slopes of $M$ without multiplicity.
\end{thm}

\begin{proof}[Proof of Proposition \ref{prop:key}]
Denote by (i)' and (ii)'' the negations of (i) and (ii) respectively.

\noindent (ii)'$\Rightarrow$(i)' It suffices to find a non-zero quotient of $M$, which is bounded and not pure since it is a quotient of $M/M^0$ by Theorem \ref{thm:Rob}. Assume that there exists a non-zero quotient $Q$ of $M$ such that $\lambda_{\max}(Q)\neq\lambda_{\max}(M)$. After replacing $Q$ by $Q/\cup_{\mu<\lambda_{\max}(Q)}S_{\mu}(Q)$, we may assume that $Q$ is pure as a $\varphi$-module over $\mathcal{E}$. Since $\lambda_{\max}(Q)\le\lambda_{\max}(M)$, we have $\lambda_{\max}(Q)<\lambda_{\max}(M)$. Write $Q$ as $M/M'$. Then we have $\lambda_{\max}(M')=\lambda_{\max}(M)$, hence, $\lambda_{\max}(M')>\lambda_{\max}(Q)$. Therefore the exact sequence
\[
0\to M'/\cup_{\mu<\lambda_{\max}(M')}S_{\mu}(M')\to M/\cup_{\mu<\lambda_{\max}(M')}S_{\mu}(M')\to Q\to 0
\]
splits as $(\varphi,\nabla)$-modules over $\mathcal{E}$ by Lemma \ref{lem:split}. Since both $M'/\cup_{\mu<\lambda_{\max}(M')}S_{\mu}(M')$ and $Q$ are pure as $\varphi$-modules over $\mathcal{E}$, they are bounded by Proposition \ref{prop:lgE5}. Thus, the quotient $M/\cup_{\mu<\lambda_{\max}(M')}S_{\mu}(M')$ of $M$ is bounded and not pure.

\noindent (i)'$\Rightarrow$(ii)' (when $k$ is perfect) By Splitting theorem (Theorem \ref{thm:split}), there exists an isomorphism of $(\varphi,\nabla)$-modules over $\mathcal{E}$
\[
M/M^0\cong\oplus_{i=1,\dots,d}(S_{\mu_i}(M/M^0)/\cup_{\mu<\mu_i}S_{\mu}(M/M^0)),
\]
where $\mu_1<\dots<\mu_d$ are the Frobenius slopes of $M/M^0$. By assumption, we have $d\ge 2$, hence, $\mu_1<\mu_d\le \lambda_{\max}(M/M^0)\le\lambda_{\max}(M)$. Hence, $Q=S_{\mu_1}(M/M^0)/\cup_{\mu<\mu_1}S_{\mu}(M/M^0)(=S_{\mu_1}(M/M^0))$ is a non-zero quotient of $M$ with $\lambda_{\max}(Q)=\mu_1<\lambda_{\max}(M)$.
\end{proof}

\begin{cor}[{\cite[Lemma 7.3]{Ohk}}]\label{cor:new2}
Let $M$ be a non-zero PBQ $(\varphi,\nabla)$-module over $\mathcal{E}$. Then,
\[
\lambda_{\max}(M)=\lambda_{\max}(M/M^0).
\]
\end{cor}
\begin{proof}
By $M\neq M^0$ (Lemma \ref{lem:lgE1} (ii)), we apply Proposition \ref{prop:key} (i)$\Rightarrow$(ii) to $M/M^0$.
\end{proof}


\section{Reverse filtration}\label{sec:rev}

The reverse filtration is first introduced by de Jong in \cite{dJ} to prove a positive characteristic analogue of Tate's fully faithful theorem on $p$-divisible groups. Then it is generalized by Kedlaya in a proof of the slope filtration theorem (\cite{Doc}). In this section, we recall the construction of the reverse filtration in the framework of \cite{rel}, precisely speaking, over $\tilde{\mathcal{R}}^{\bd}$ due to Liu (\cite{Liu}). As a corollary of the construction, we obtain Proposition \ref{prop:maximum eigen}, which is another ingredient in the proof of Theorem \ref{conj:bd} (ii).


\begin{ass}
In this section, we work with notation as in \S \ref{sec:Robba}, in particular, we allow a relative Frobenius as a Frobenius lift on $\mathcal{R}$. We also keep Assumption \ref{ass:relative}.
\end{ass}

\begin{propdfn}[{\cite[Proposition 1.6.9]{Liu}}]\label{propdfn:descent}
Let $M$ be a $\varphi$-module over $\tilde{\mathcal{R}}^{\bd}$. Then the reverse filtration on $M\otimes_{\tilde{\mathcal{R}}^{\bd}}\tilde{\mathcal{E}}$ (Definition \ref{dfn:reverse}) descends to $\tilde{\mathcal{R}}^{\bd}$. We call the resulting filtration the reverse filtration of $M$, and denote it by $S^{\bullet}(M)$ (our notation is different from Liu's).
\end{propdfn}

We quickly recall the proof as it is needed in the following. Let $\lambda$ denote the maximum Frobenius slope of $M\otimes_{\tilde{\mathcal{R}}^{\bd}}\tilde{\mathcal{E}}$. It suffices to show that $S^{\lambda}(M\otimes_{\tilde{\mathcal{R}}^{\bd}}\tilde{\mathcal{E}})$ descends to $M$. By Dieudonn\'e-Manin theorem over $\tilde{\mathcal{E}}$ (Lemma \ref{lem:DM}), there exists a basis $\{v_i\}$ of $S^{\lambda}(M\otimes_{\tilde{\mathcal{R}}^{\bd}}\tilde{\mathcal{E}})$ such that $\varphi^d(v_i)=c_iv_i$ with $c_i\in K$ such that $\log{|c_i|}/\log{|q^d|}=\lambda$. By the descending lemma of ($d$-)eigenvectors of slope $\lambda$ (\cite[Lemma 1.6.8]{Liu}), $v_i\in M$.

\begin{dfn}[{cf. Definition \ref{dfn:eigen}}]
Let $d$ be a positive integer. A {\it (Frobenius) $d$-eigenvector} of a $\varphi$-module $M$ over (the non-complete field) $\tilde{\mathcal{R}}^{\bd}$ is a non-zero element $v$ of $M$ such that
\[
\varphi^d(v)=cv
\]
for some $c\in \tilde{\mathcal{R}}^{\bd}$. We refer to the quotient $\log{|c|_0}/\log{|q^d|}$ as the {\it (Frobenius) slope} of $v$.
\end{dfn}

\begin{prop}\label{prop:maximum eigen}
Let $M$ be a $\varphi$-module over $\tilde{\mathcal{R}}^{\bd}$. Then there exists a Frobenius $d$-eigenvector $v\in M$ of slope $\lambda_{\max}(M\otimes_{\tilde{\mathcal{R}}^{\bd}}\tilde{\mathcal{E}})$ for some positive integer $d$.
\end{prop}
\begin{proof}
With notation as in the proof of Proposition-Definition \ref{propdfn:descent}, put $v=v_i$.
\end{proof}


\section{Proof of Theorem \ref{conj:bd} (ii)}\label{sec:pf}

In this section, we will prove:

\begin{prop}\label{prop:pf}
Assume that $k$ is algebraically closed. Let $M$ be a PBQ $(\varphi,\nabla)$-module over $\mathcal{R}^{\bd}$ solvable in $\mathcal{R}_{\log}$. Let $f\in \Sol_{\lambda}(M)$ be a Frobenius $d$-eigenvector of slope $\mu$. Then $\mu\le\lambda-\lambda_{\max}(M_{\mathcal{E}})$.
\end{prop}

\begin{proof}[{Proof of Theorem \ref{conj:bd} (ii) assuming Proposition \ref{prop:pf}}]
We may assume that $k$ is algebraically closed after extending the coefficient field $K$ (Lemmas \ref{lem:slope5} (II)-(i) and \ref{lem:coefficient Robba}). By Proposition \ref{prop:main}, we have only to prove $\Sol_{\lambda}(M)\subset S_{\lambda-\lambda_{\max}}(\Sol(M))$ for an arbitrary $\lambda$. By Dieudonn\'e-Manin theorem (Lemma \ref{lem:DM}), there exists a $K$-basis $\{f_i\}$ of $\Sol_{\lambda}(M)$ consisting of Frobenius $d$-eigenvectors. Let $\mu_i$ denote the slope of $f_i$. By Proposition \ref{prop:pf}, $\mu_i\le\lambda-\lambda_{\max}$. By Lemma \ref{lem:slope6}, $f_i\in S_{\mu_i}(\Sol(M))\subset S_{\lambda-\lambda_{\max}}(\Sol(M))$ for all $i$, hence, $\Sol_{\lambda}(M)\subset S_{\lambda-\lambda_{\max}}(\Sol(M))$.
\end{proof}

\begin{proof}[{Proof of Proposition \ref{prop:pf}}]
\begin{enumerate}
\item[$\bullet$] The case $\varphi(f)=f$

In this case, $\mu=0$, and we have only to prove $\lambda_{\max}(M_{\mathcal{E}})\le\lambda$. Since $f:M\to\mathcal{R}_{\log}$ is $\varphi,\nabla$-equivariant by assumption, $\ker{f}$ is a $(\varphi,\nabla)$-submodule of $M$. Denote by $M''$ the quotient $M/\ker{f}$, and by $f'':M''\hookrightarrow\mathcal{R}_{\log}$ the induced map by $f$. Then $f''\in \Sol_{\lambda}(M'')$ and $\varphi(f'')=f''$ by definition, and $\lambda_{\max}(M_{\mathcal{E}})=\lambda_{\max}(M''_{\mathcal{E}})$ by Proposition \ref{prop:key} (i)$\Rightarrow$(ii). Therefore we may also assume that $f$ is injective after replacing $(M,f)$ by $(M'',f'')$.

We fix a $\varphi$-equivariant embedding $\psi:\mathcal{R}_{\log}\hookrightarrow\tilde{\mathcal{R}}_{\log}$ in Definition \ref{dfn:log embedding}. We consider the composition $\tilde{f}$ of
\[\xymatrix{
M\otimes_{\mathcal{R}^{\bd}}\tilde{\mathcal{R}}^{\bd}\ar@{^{(}->}[r]^(.47){f\otimes \id}&\mathcal{R}_{\log}\otimes_{\mathcal{R}^{\bd}}\tilde{\mathcal{R}}^{\bd}\ar@{^{(}->}[r]&\tilde{\mathcal{R}}_{\log},
}\]
where the second injection is the multiplication map given in Lemma \ref{lem:log inj}. Then we have
\begin{equation}\label{eq:pf1}
\tilde{f}(M\otimes_{\mathcal{R}^{\bd}}\tilde{\mathcal{R}}^{\bd})\subset \psi(f(M))\cdot\tilde{\mathcal{R}}^{\bd}\subset \psi(\Fil_{\lambda}\mathcal{R}_{\log})\cdot\tilde{\mathcal{R}}^{\bd}\subset \Fil_{\lambda}\tilde{\mathcal{R}}_{\log}\cdot \Fil_{0}\tilde{\mathcal{R}}_{\log}\subset \Fil_{\lambda}\tilde{\mathcal{R}}_{\log},
\end{equation}
where each inclusion follows by definition, by $f\in \Sol_{\lambda}(M)$, by Lemmas \ref{lem:log-growth psi log}, \ref{lem:property log} (iii), and by Lemma \ref{lem:property log} (i) respectively. By Proposition \ref{prop:maximum eigen}, there exists a Frobenius $d$-eigenvector $\tilde{v}\in M\otimes_{\mathcal{R}^{\bd}}\tilde{\mathcal{R}}^{\bd}$ of slope $\lambda_{\max}(M_{\mathcal{E}})$. Then $\tilde{f}(\tilde{v})\in\tilde{\mathcal{R}}_{\log}$ is also a $d$-eigenvector of slope $\lambda_{\max}(M_{\mathcal{E}})$ in the sense of Definition \ref{dfn:eigen log extended} by the $\varphi$-equivariantness of $\tilde{f}$. By Lemma \ref{lem:calc log-growth}, $\tilde{f}(\tilde{v})$ is exactly of log-growth $\lambda_{\max}(M_{\mathcal{E}})$. By (\ref{eq:pf1}), the case $\lambda_{\max}(M_{\mathcal{E}})>\lambda$ never occurs, hence, $\lambda_{\max}(M_{\mathcal{E}})\le\lambda$.

\item[$\bullet$] The general case

We will reduce to the previous case. Write $\varphi^d(f)=cf$ for some $c\in K^{\times}$ with $\mu=\log{|c|}/\log{|q^d|}$. With notation as in \S \ref{sec:phinabla}, put $N=([d]_*M)(c)$, which is a $(\varphi^d,\nabla)$-module over $\mathcal{R}^{\bd}$. Since the map $T:M\to N;m\mapsto m\otimes e_c$ is an isomorphism of $\nabla$-modules over $\mathcal{R}^{\bd}$, $N$ is also solvable in $\mathcal{R}_{\log}$. We define $f':N\to\mathcal{R}_{\log}$ as the composition $f\circ T^{-1}$. Then we have $\varphi^d(f')=f'$ and $f'\in\Sol(N)$ by definition, and $f\in \Sol_{\lambda}(M)$ by $f'(N)=f(M)$. By applying the previous case to $f'$, we obtain $\lambda_{\max}(N_{\mathcal{E}})=\lambda_{\max}(M_{\mathcal{E}})+\mu\le\lambda$.
\end{enumerate}
\end{proof}


\part{Applications}
In part II, we give applications of Theorems \ref{conj:CT} (not a conjecture anymore), \ref{conj:CTgen}, and \ref{conj:bd} to the log-growth Newton polygons of $(\varphi,\nabla)$-modules over $K[\![t]\!]_0,\mathcal{E}$, and $\mathcal{R}^{\bd}$. In particular, we prove Dwork's conjecture on the semicontinuity (Corollary \ref{cor:LGFDW}), a weak tensor compatibility (Proposition \ref{prop:weak tensor}), the dual invariance of the maximum generic slope (Theorem \ref{thm:inv}), and an analogue of Drinfeld-Kedlaya theorem (Theorem \ref{thm:Gri}). We also gather some remarks, probably known for experts, such as some interactions between Theorems \ref{conj:CT} and \ref{conj:CTgen} (Propositions \ref{prop:CTgeneric}, \ref{prop:(ii) implies (i)}), for the reader's convenience.

\section{Chiarellotto-Tsuzuki conjecture in the generic case}\label{sec:CTgen}

In this section, we give an alternative proof of Theorem \ref{conj:CTgen} using Theorem \ref{conj:CT}.

\begin{prop}\label{prop:CTgeneric}
The conjecture $\mathbf{LGF}_{\mathcal{E}[\![X-t]\!]_0}$ implies the conjecture $\mathbf{LGF}_{\mathcal{E}}$.
\end{prop}

Actually this is known for experts. For the reader's convenience, we record here.

\begin{lem}\label{lem:evaluation}
Let $M$ be a $(\varphi,\nabla)$-module over $\mathcal{E}$. Then there exists a  functorial isomorphism of $\varphi$-modules over $\mathcal{E}$
\[
M\cong V(M).
\]
\end{lem}
\begin{proof}
We consider the evaluation map $\ev:\mathcal{E}\{X-t\}\to\mathcal{E};f(X-t)\mapsto f(0)$. Then $\ev$ is $\varphi$-equivariant and satisfies $\ev\circ\tau=\id_{\mathcal{E}}$. By Dwork's trick, there exists a functorial isomorphism of $(\varphi,\nabla)$-modules over $\mathcal{E}\{X-t\}$
\[
V(M)\otimes_{\mathcal{E},\tau}\mathcal{E}\{X-t\}\cong\tau^*M\otimes_{\mathcal{E}[\![X-t]\!]_0}\mathcal{E}\{X-t\}.
\]
By pulling back via $\ev$, we obtain the desired isomorphism.
\end{proof}

In the generic case, Frobenius slope filtration and the log-growth filtration commute with the functor $V$.

\begin{lem}\label{lem:CTgeneric2}
Let $M$ be a $(\varphi,\nabla)$-module over $\mathcal{E}$.
\begin{enumerate}
\item There exists a canonical isomorphism
\[
V(S_{\bullet}(M))\cong S_{\bullet}(V(M)).
\]
In particular, the slope multiset of $S_{\bullet}(M)$ coincides with that of $S_{\bullet}(V(M))$.
\item There exists a canonical isomorphism
\[
V(M^{\bullet})\cong V(M)^{\bullet}.
\]
In particular, the slope multiset of $M^{\bullet}$ coincides with that of $V(M)^{\bullet}$.
\end{enumerate}
\end{lem}
\begin{proof}
\begin{enumerate}
\item Both $V(S_{\bullet}(M))$ and $S_{\bullet}(V(M))$ are canonically isomorphic to Frobenius slope filtration $S_{\bullet}(M)$ of $M$ by Lemma \ref{lem:evaluation}.
\item We consider the canonical perfect pairing $V(M)\otimes_{\mathcal{E}}\Sol(M)\to\mathcal{E}$. Then we have
\[
V(M^{\lambda})=(\Sol(M/M^{\lambda}))^{\perp}=(\Sol_{\lambda}(M))^{\perp}=V(M)^{\lambda},
\]
where the second equality follows from Theorem \ref{thm:Rob}. By taking the orthogonal parts, we obtain the assertion.
\end{enumerate}
\end{proof}

\begin{lem}\label{lem:tau inv of PBQ}
Let $M$ be a $(\varphi,\nabla)$-module over $\mathcal{E}$. Then $M$ is PBQ if and only if $\tau^*M$ is PBQ.
\end{lem}
\begin{proof}
Let $\mathfrak{E}$ denote the completion of the fraction field of $\mathcal{E}[\![X-t]\!]_0$ with respect to Gauss norm defined by $|\sum_ia_i(X-t)^i|_0=\sup_i{|a_i|}$ for any bounded sequence $\{a_i\}$ of $\mathcal{E}$. Let $i:\mathcal{E}[\![X-t]\!]_0\to \mathfrak{E}$ be the inclusion, and put $\iota=i\circ\tau$. Let $\tau_{\mathfrak{E}}:\mathfrak{E}\to \mathfrak{E}[\![Y-(X-t)]\!]_0$ denote the map $\tau$ for $(\mathfrak{E},Y-(X-t))$ instead of $(\mathcal{E},X-t)$ (Example 3.3.5). Then we have the commutative diagram
\[\xymatrix{
\mathcal{E}\ar[r]^(.35){\tau}\ar[d]^{\iota}&\mathcal{E}[\![X-t]\!]_0\ar[d]^{\jmath}\\
\mathfrak{E}\ar[r]^(.27){\tau_\mathfrak{E}}&\mathfrak{E}[\![Y-(X-t)]\!]_0,
}\]
where $\jmath(\sum_ia_i(X-t)^i)=\sum_i\iota(a_i)(Y-(X-t))^i$ for any bounded sequence $\{a_i\}$ of $\mathcal{E}$.

By definition,
\begin{center}
$M$ is PBQ $\Leftrightarrow$ $\Sol_0(\tau^*M)$ is pure,
\end{center}
and,
\begin{center}
$\tau^*M$ is PBQ $\Leftrightarrow$ $i^*\tau^*M$ is PBQ $\Leftrightarrow$ $\iota^*M$ is PBQ $\Leftrightarrow$ $\Sol_0(\iota^*M)$ is pure.
 \end{center}
There exist isomorphisms of $\varphi$-modules over $\mathfrak{E}$
\[
\Sol_0(\iota^*M)=\Sol_0(\tau^*_{\mathfrak{E}}\iota^*M)\cong \Sol_0(\jmath^*\tau^*M)\cong \Sol_0(\tau^*M)\otimes_{\mathcal{E},\iota}\mathfrak{E},
\]
where the last isomorphism follows from Lemma \ref{lem:coefficient power}. In particular, $\Sol_0(\iota^*M)$ is pure if and only if so is $\Sol_0(\tau^*M)$, which implies the assertion.
\end{proof}

\begin{proof}[{Proof of Proposition \ref{prop:CTgeneric}}]
Let $M$ be a $(\varphi,\nabla)$-module over $\mathcal{E}$. Note that $\mathbf{LGF}_{\mathcal{E}}$ (i) for $M$ is true if and only if the function $\lambda\mapsto\dim_{\mathcal{E}}M^{\lambda}$ is right continuous, and locally constant at an arbitrary $\lambda\notin\mathbb{Q}$. We also note that the equality in $\mathbf{LGF}_{\mathcal{E}}$ (ii) holds for $M$ if and only if the equation
\begin{equation}\label{eq:CTgeneric4}
\dim_{\mathcal{E}}M-\dim_{\mathcal{E}}M^{\lambda}=\dim_{\mathcal{E}}S_{\lambda_{\max}(M)-\lambda}(M\spcheck)
\end{equation}
holds by Proposition \ref{prop:CT1}. Similar equivalences hold for $\mathbf{LGF}_{\mathcal{E}[\![X-t]\!]_0}$ (i) and (ii). 

$\bullet$ $\mathbf{LGF}_{\mathcal{E}[\![X-t]\!]_0}$ (i) $\Rightarrow$ $\mathbf{LGF}_{\mathcal{E}}$ (i)

By Lemmas \ref{lem:evaluation} and \ref{lem:CTgeneric2} (ii), $\dim_{\mathcal{E}}M^{\lambda}=\dim_{\mathcal{E}}V(\tau^*M)^{\lambda}$. Therefore, if $\mathbf{LGF}_{\mathcal{E}[\![X-t]\!]_0}$ (i) for $\tau^*M$ is true, then so is $\mathbf{LGF}_{\mathcal{E}}$ (i) for $M$.

$\bullet$ $\mathbf{LGF}_{\mathcal{E}[\![X-t]\!]_0}$ (ii) $\Rightarrow$ $\mathbf{LGF}_{\mathcal{E}}$ (ii)

We assume that $M$ is PBQ. By Lemma \ref{lem:tau inv of PBQ}, $\tau^*M$ is also PBQ. We also note that
\[
\lambda_{\max}(M)=\lambda_{\max}(M\otimes_{\mathcal{E},\iota}\mathfrak{E})=\lambda_{\max}((\tau^*M)\otimes_{\mathcal{E}[\![X-t]\!]_0}\mathfrak{E}),
\]
where $\mathfrak{E},\iota$ are as in the proof of Lemma \ref{lem:tau inv of PBQ}. Hence, if $\mathbf{LGF}_{\mathcal{E}[\![X-t]\!]_0}$ (ii) for $\tau^*M$ is true, then we obtain the equality
\begin{equation}\label{eq:CTgeneric5}
\dim_{\mathcal{E}}V(\tau^*M)-\dim_{\mathcal{E}}V(\tau^*M)^{\lambda}=\dim_{\mathcal{E}}S_{\lambda_{\max}(M)-\lambda}(V((\tau^*M)\spcheck)).
\end{equation}
By Lemmas \ref{lem:evaluation} and \ref{lem:CTgeneric2}, (\ref{eq:CTgeneric5}) is nothing but (\ref{eq:CTgeneric4}).
\end{proof}

\begin{rem}\label{rem:principle}
By ``generalizing'' the proof of Proposition \ref{prop:CTgeneric}, we obtain the following general principle. Let $\mathcal{P}(\mathcal{R}^{\bd})$ (resp. $\mathcal{P}(K[\![t]\!]_0),\mathcal{P}(\mathcal{E})$) be some property on the log-growth filtration for an arbitrary $(\varphi,\nabla)$-module $M$ over $\mathcal{R}^{\bd}$ (resp. over $K[\![t]\!]_0,\mathcal{E}$) solvable in $\mathcal{R}_{\log}$: $\mathcal{P}$ is the right continuity, or, the rationality of the slopes in the case of Theorem \ref{conj:bd} (i). Assume that $\mathcal{P}(\mathcal{R}^{\bd})$ is true. Then, by Lemma \ref{lem:invariance log-growth}, the truth of $P(\mathcal{R}^{\bd})$ implies the truth of $\mathcal{P}(K[\![t]\!]_0)$. Similarly, the truth of $P(\mathcal{R}^{\bd}_{\mathcal{E},X-t})$ implies the truth of $\mathcal{P}(\mathcal{E}[\![X-t]\!]_0)$, where $\mathcal{R}^{\bd}_{\mathcal{E},X-t}$ denotes the bounded Robba ring over $\mathcal{E}$ with the variable $X-t$. By Lemma \ref{lem:CTgeneric2}, the truth of $P(\mathcal{E}[\![X-t]\!]_0)$ implies the truth of $\mathcal{P}(\mathcal{E})$. We may also use this principle to study log-$(\varphi,\nabla)$-modules over $K[\![t]\!]_0$. In the following, for simplicity of exposition, we give a proof of $\mathcal{P}(\mathcal{R}^{\bd})$, and omit proofs of $\mathcal{P}(K[\![t]\!]_0)$ and $\mathcal{P}(\mathcal{E})$ as long as this principle works.
\end{rem}

The following d\'evissage lemma will be useful.

\begin{lem}[{\cite[Proposition 2.6]{CT2} in the case of $\mathcal{E}$}]\label{lem:exact}
Let $0\to M'\to M\to M''\to 0$ be an exact sequence of $(\varphi,\nabla)$-modules over $\mathcal{R}^{\bd}$ such that $M$ is solvable in $\mathcal{R}_{\log}$.
\begin{enumerate}
\item[(i)] There exists canonical exact sequences
\[
0\to \Sol_{\lambda}(M'')\to \Sol_{\lambda}(M)\to \Sol_{\lambda}(M'),
\]
\[
V(M')/V(M')^{\lambda}\to V(M)/V(M)^{\lambda}\to V(M'')/V(M'')^{\lambda}\to 0
\]
for an arbitrary real number $\lambda$.
\item[(ii)] If $\Sol_{\lambda}(M')=S_{\lambda-\lambda_{\max}(M_{\mathcal{E}})}(\Sol(M'))$, or, equivalently, $V(M')^{\lambda}=(S_{\lambda-\lambda_{\max}(M_{\mathcal{E}})}(V(M')\spcheck))^{\perp}$ for $\lambda$, then the exact sequences in (i) for $\lambda$ extend to short exact sequences.
\end{enumerate}

A similar assertion for (log-)$(\varphi,\nabla)$-modules over $K[\![t]\!]_0$ or $\mathcal{E}$.
\end{lem}

\begin{rem}\label{rem:exact}
The assumption in (ii) holds if $M'$ is PBQ and $\lambda_{\max}(M'_{\mathcal{E}})=\lambda_{\max}(M_{\mathcal{E}})$ by Theorem \ref{conj:bd} (ii).
\end{rem}

\begin{proof}
By Remark \ref{rem:principle}, we have only to prove in the case of $\mathcal{R}^{\bd}$. By duality, we have only to prove the assertion for $\Sol_{\bullet}(\cdot)$. By Lemma \ref{lem:V} and duality, there exists a short exact sequence
\[
0\to \Sol(M'')\to \Sol(M)\to \Sol(M')\to 0,
\]
which induces the desired left exact sequence in (i). If $\Sol_{\lambda}(M')=S_{\lambda-\lambda_{\max}(M_{\mathcal{E}})}(\Sol(M'))$, then there exists a commutative diagram
\[\xymatrix{
\Sol_{\lambda}(M)\ar[r]&\Sol_{\lambda}(M')\ar@{=}[d]\\
S_{\lambda-\lambda_{\max}(M_{\mathcal{E}})}(\Sol(M))\ar[r]\ar[u]&S_{\lambda-\lambda_{\max}(M_{\mathcal{E}})}(\Sol(M')),
}\]
where the left vertical arrow is given by Proposition \ref{prop:main}. Since the bottom horizontal arrow is surjective by the strictness of the Frobenius slope filtration, $\Sol_{\lambda}(M)\to \Sol_{\lambda}(M')$ is surjective, which implies (ii).
\end{proof}


\section{The maximally PBQ submodule and PBQ filtration}\label{sec:max}

To study the log-growth filtration for an arbitrary $(\varphi,\nabla)$-module $M$, we need to describe $M$ using PBQ $(\varphi,\nabla)$-modules. In this section, we recall a filtering technique established in \cite{CT2}, then give a generalization to $\mathcal{R}^{\bd}$. As a first application, we prove that in Theorem \ref{conj:bd}, part (ii) implies part (i) (Proposition \ref{prop:(ii) implies (i)})

We first give some technical results on the extended Robba ring $\tilde{\mathcal{R}}$. Therefore we work with notation as in \S \ref{sec:Robba} for a while. The reader may start from Proposition \ref{prop:existence max PBQ} admitting these results.

\begin{notation}
For the time being, let $\varphi_K:K\to K$ be an arbitrary isometric ring endomorphism. As in the proof of Lemma \ref{lem:inj}, let $L/K$ be an extension of complete discrete valuation fields, on which $\varphi_K$ extends isometrically, such that any $\varphi$-module over the residue field of $L$ is trivial. Let $\tilde{\mathcal{R}}_L^{(\bd)}$ denote the extended (bounded) Robba rings over $L$.
\end{notation}

\begin{lem}[{cf. \cite[Proposition 8.1]{dJ}}]\label{lem:inj gen}
The multiplication map
\[
\tilde{\mathcal{R}}_L^{\bd}\otimes_{\mathcal{R}^{\bd}}\mathcal{E}\to\tilde{\mathcal{E}}_L;y\otimes z\mapsto y\psi(z)
\]
is injective, where $\psi:\mathcal{E}\to\tilde{\mathcal{E}}_L$ is the composition of $\mathcal{E}\to\tilde{\mathcal{E}}$ induced by $\psi:\mathcal{R}^{\bd}\to\tilde{\mathcal{R}}^{\bd}$ given in Proposition \ref{prop:embedding}, and the inclusion $\tilde{\mathcal{E}}\to\tilde{\mathcal{E}}_L$.
\end{lem}
\begin{proof}
(cf. the proof of \cite[Proposition 3.5.2]{rel}) Recall that there exists a continuous map $f:\tilde{\mathcal{R}}_L\to\mathcal{R}$ satisfying:
\begin{enumerate}
\item[(a)] $f(\psi(a)w)=af(w),a\in\mathcal{R},w\in\tilde{\mathcal{R}}_L$;
\item[(b)] $|w|_r=\sup_{\alpha\in [0,1),a\in L^{\times}}\{|a|^{-1}e^{-\alpha}|f(au^{-\alpha}w)|_r\}$ for $w\in\tilde{\mathcal{R}}^r_L$, $r\in (0,r_0)$,
\end{enumerate}
where $r_0$ is the constant given in Proposition \ref{prop:embedding} (\cite[Definition 3.5.1]{rel}). Fix $w\in\tilde{\mathcal{R}}^{\bd}_L$. By putting $\alpha=0$ and $a=1$ on the RHS of the equality in (b), we obtain
\[
|f(w)|_r\le |w|_r\ \forall r\in (0,r_0).
\]
Since $\sup_{r\in (0,r_0)}|f(w)|_r\le \sup_{r\in (0,r_0)}|w|_r=|w|_0<\infty$, we obtain $f(w)\in\mathcal{R}^{\bd}$ by Lemma \ref{lem:gen1}. Moreover, we obtain the inequality $|f(w)|_0=\lim_{r\to 0+}|f(w)|_r\le\lim_{r\to 0+}|w|_r=|w|_0$, which implies that the restriction of $f$ to $\tilde{\mathcal{R}}_L^{\bd}$ extends to $f_{\eta}:\tilde{\mathcal{E}}_L\to\mathcal{E}$ satisfying a similar property to (a).

Suppose the contrary: choose $x\neq 0$ in the kernel of the multiplication map, and choose a presentation $x=\sum_{i=1}^ny_i\otimes z_i$ with $n$ minimal. Then $z_1,\dots,z_n$ are linearly independent over $\mathcal{R}^{\bd}$ by \cite[Corollary 3.4.3]{rel}. We will construct a non-trivial dependence relation between the $z_i$'s. As a corollary of (b), we may choose $\alpha\in (0,1]$ and $a\in L^{\times}$ such that $f(au^{-\alpha}y_1)\neq 0$. By applying $f_{\eta}$ to $0=\sum_iau^{-\alpha}y_i\cdot \psi(z_i)$, we obtain the non-trivial dependence relation $0=\sum_if(au^{-\alpha}y_i)\cdot z_i$, which is a contradiction.
\end{proof}

\begin{lem}[{cf. \cite[Lemma 5.7]{CT2}}]\label{lem:dJ}
We consider an $\mathcal{R}^{\bd}$-linear map
\[
f:M\to Q,
\]
where $M$ is a $\varphi$-module over $\mathcal{R}^{\bd}$, and $Q$ is a pure $\varphi$-module over $\mathcal{E}$ of slope $\lambda$. If $f$ is $\varphi$-equivariant and injective, then $\lambda=\lambda_{\max}(M_{\mathcal{E}})$.
\end{lem}
\begin{proof}
We extend $f$ to the $\varphi$-equivariant $\tilde{\mathcal{R}}_L^{\bd}$-linear injective map
\[
\tilde{f}:\tilde{\mathcal{R}}_L^{\bd}\otimes_{\mathcal{R}^{\bd}}M\xrightarrow[]{\id\otimes f}\tilde{\mathcal{R}}_L^{\bd}\otimes_{\mathcal{R}^{\bd}}Q\cong\tilde{\mathcal{R}}_L^{\bd}\otimes_{\mathcal{R}^{\bd}}\mathcal{E}\otimes_{\mathcal{E}}Q\hookrightarrow\tilde{\mathcal{E}}_L\otimes_{\mathcal{E}}Q,
\]
where the last injection is induced by the multiplication map in Lemma \ref{lem:inj gen}. The maximum Frobenius slope of $\tilde{\mathcal{E}}_L\otimes_{\mathcal{R}^{\bd}}M$ as a $\varphi$-module over $\tilde{\mathcal{E}}_L$ is equal to $\lambda_{\max}(M_{\mathcal{E}})$. By Proposition \ref{prop:maximum eigen}, there exists a Frobenius $d$-eigenvector $\tilde{v}\in\tilde{\mathcal{R}}_L^{\bd}\otimes_{\mathcal{R}^{\bd}}M$ of slope $\lambda_{\max}(M_{\mathcal{E}})$. Then $\tilde{f}(\tilde{v})\in\tilde{\mathcal{E}}_L\otimes_{\mathcal{E}}Q$ is also a Frobenius $d$-eigenvector of slope $\lambda_{\max}(M_{\mathcal{E}})$. Since $\tilde{\mathcal{E}}_L\otimes_{\mathcal{E}}Q$ is a pure $\varphi$-module over $\tilde{\mathcal{E}}_L$ of slope $\lambda$, we have $\lambda=\lambda_{\max}(M_{\mathcal{E}})$ by Lemma \ref{lem:slope6}.
\end{proof}

We will recall the construction of the PBQ filtration due to Chiarellotto and Tsuzuki: it is first done over $\mathcal{E}$, then over $K[\![t]\!]_0$ by descent. We also give an analogous construction over $\mathcal{R}^{\bd}$ by generalizing the descending argument in \cite{CT2}.

\begin{ass}
In the rest of this section except Proposition \ref{prop:(ii) implies (i)}, assume that $k$ is perfect and $\varphi_K$ is a $q$-power Frobenius lift.
\end{ass}

\begin{prop}[{\cite[Proposition 5.4]{CT2}}]\label{prop:existence max PBQ}
Let $M$ be a $(\varphi,\nabla)$-module over $\mathcal{E}$. Then there exists a unique $(\varphi,\nabla)$-submodule $N$ of $M$ such that the composition
\[
N/N^0\to M/M^0\to M/\cup_{\mu<\lambda_{\max}(M)}S_{\mu}(M)
\]
is an isomorphism, where the first and second maps are induced by Lemma \ref{lem:lgE2} (i) and Proposition \ref{prop:CT2} (plus Lemma \ref{lem:slope5} (II)-(ii)) respectively. Note that if this is the case, then $N/N^0$ is pure of slope $\lambda_{\max}(M)$, in particular, $N$ is PBQ.
\end{prop}

\begin{dfn}[{\cite[Corollary 5.5]{CT2}}]\label{dfn:PBQ fil gen}
With notation as above, we call $N$ the {\it maximally PBQ submodule} of $M$, and we denote $N$ by $P_1(M)$. Note that $P_1(M)\neq 0$ unless $M=0$, and $P_1(M)=M$ if and only if $M$ is PBQ. We put $P_0(M)=0$, and $P_i(M)$ as the inverse image of $P_1(M/P_{i-1}(M))$ under the projection $M\to M/P_{i-1}(M)$. Thus we obtain an increasing filtration
\[
0=P_0(M)\subset P_1(M)\subset\dots\subset P_r(M)\subset\dots
\]
of $(\varphi,\nabla)$-modules over $\mathcal{E}$, which is called the {\it PBQ filtration} of $M$.

Note that:
\begin{enumerate}
\item[$\bullet$] the PBQ filtration is separated and exhaustive. Each graded piece $P_{i+1}(M)/P_i(M)$ is PBQ;
\item[$\bullet$] let $r$ be the minimum natural number such that $P_r(M)=M$. Then $P_i(M)\neq P_{i+1}(M)$ for $i=0,\dots,r-1$;
\item[$\bullet$] if $M$ is not PBQ, then we have
\[
\lambda_{\max}(M/P_1(M))<\lambda_{\max}(M)=\lambda_{\max}(P_1(M)).
\]
In fact, the inclusion $\cup_{\mu<\lambda_{\max}(M)}S_{\mu}(M/P_1(M))\subset M/P_1(M)$ is an equality by the surjectivity of the canonical map $P_1(M)\to M/\cup_{\mu<\lambda_{\max}(M)}S_{\mu}(M)$. This implies $\lambda_{\max}(M/P_1(M))<\lambda_{\max}(M)$, and $\lambda_{\max}(M)=\max\{\lambda_{\max}(M/P_1(M)),\lambda_{\max}(P_1(M))\}=\lambda_{\max}(P_1(M))$.
\end{enumerate}
\end{dfn}

\begin{rem}
Let $M$ be a $(\varphi,\nabla)$-module over $\mathcal{E}$ and $X(M)$ the set of PBQ $(\varphi,\nabla)$-submodules of $M$. In general, maximal objects in $X(M)$ are not uniquely determined. In fact, when $M=\mathcal{E}\oplus\mathcal{E}(q)$, $X(M)=\{\mathcal{E},\mathcal{E}(q)\}$. Note that $P_1(M)$ is characterized as the unique maximal element in $X(M)$, whose maximum Frobenius slope is equal to that of $M$.
\end{rem}

\begin{thm}[{cf. \cite[Theorem 5.6]{CT2}}]\label{thm:descent PBQ}
Let $R$ denote $K[\![t]\!]_0$ or $\mathcal{R}^{\bd}$. Let $M$ be a $(\varphi,\nabla)$-module over $R$. Then $P_1(M_{\mathcal{E}})$ descends to $M$. In particular, $P_{\bullet}(M_{\mathcal{E}})$ descends to a filtration on $M$.
\end{thm}
\begin{proof}
Recall that for a given $(\varphi,\nabla)$-module over $K[\![t]\!]_0$, there exists a naturally bijective correspondence between $(\varphi,\nabla)$-submodules of $M$ and those of $M_{\mathcal{R}^{\bd}}$ by \cite[Proposition 6.4]{dJ}. Hence we may assume $R=\mathcal{R}^{\bd}$. Note that when $M$ is PBQ, we have $P_1(M_{\mathcal{E}})=M_{\mathcal{E}}$, hence, the assertion is trivial. We proceed by induction on the dimension of $M$. We may assume that $M$ is not PBQ. By the induction hypothesis, it suffices to find a $(\varphi,\nabla)$-submodule $N$ of $M$ such that $N\neq M$ and $P_1(N_{\mathcal{E}})=P_1(M_{\mathcal{E}})$. By Splitting theorem (Theorem \ref{thm:split}), there exists a quotient $Q$ of $M_{\mathcal{E}}/M_{\mathcal{E}}^0$ as a $(\varphi,\nabla)$-module over $\mathcal{E}$, which is pure of slope $\lambda$ with $\lambda<\lambda_{\max}(M_{\mathcal{E}})$. We consider the composition
\[
f:M\xrightarrow[]{\inc.}M_{\mathcal{E}}\xrightarrow[]{\pr.}Q.
\]
Let $N$ denote the kernel of $f$. Note that $N$ is a $(\varphi,\nabla)$-submodule of $M$, and $N\neq M$ since $f$ is $\varphi,\nabla$-equivariant, and $Q$ is generated by the image of $f$. By applying Lemma \ref{lem:dJ} to the map $M/N\to Q$ induced by $f$, we have $\lambda_{\max}((M/N)_{\mathcal{E}})=\lambda<\lambda_{\max}(M_{\mathcal{E}})$, which implies $\lambda_{\max}(N_{\mathcal{E}})=\lambda_{\max}(M_{\mathcal{E}})$ and $(M/N)_{\mathcal{E}}=\cup_{\mu<\lambda_{\max}(M_{\mathcal{E}})}S_{\mu}((M/N)_{\mathcal{E}})$. Therefore the canonical map
\begin{equation}\label{eq:max1}
N_{\mathcal{E}}/\cup_{\mu<\lambda_{\max}(M_{\mathcal{E}})}S_{\mu}(N_{\mathcal{E}})\to M_{\mathcal{E}}/\cup_{\mu<\lambda_{\max}(M_{\mathcal{E}})}S_{\mu}(M_{\mathcal{E}})
\end{equation}
is an isomorphism by the strictness of Frobenius slope filtration. By the uniqueness of $P_1(M_{\mathcal{E}})$, the isomorphism (\ref{eq:max1}) implies $P_1(N_{\mathcal{E}})=P_1(M_{\mathcal{E}})$.
\end{proof}

\begin{dfn}[{cf. \cite[Corollary 5.10]{CT2}}]\label{dfn:PBQ fil}
With notation as above, the resulting filtration on $M$ is denoted by $P_{\bullet}(M)$, and called the {\it PBQ filtration} of $M$. We also call $P_1(M)$ the {\it maximally PBQ submodule} of $M$.

By definition, similar properties as in Definition \ref{dfn:PBQ fil gen} hold. It is worth noting that if $M$ is irreducible as a $(\varphi,\nabla)$-module over $R$, then $M$ is PBQ. In fact, we have $M=P_1(M)$ since $P_1(M)\neq 0$.
\end{dfn}

The PBQ filtration is compatible with the log-growth filtration as follows.

\begin{lem}\label{lem:exact PBQ}
Let $M$ be a $(\varphi,\nabla)$-module over $\mathcal{R}^{\bd}$ solvable in $\mathcal{R}_{\log}$. Then, for an arbitrary real number $\lambda$, there exist canonical exact sequences
\[
0\to\Sol_{\lambda}(M/P_1(M))\to\Sol_{\lambda}(M)\to\Sol_{\lambda}(P_1(M))\to 0,
\]
\[
0\to V(P_1(M))/V(P_1(M))^{\lambda}\to V(M)/V(M)^{\lambda}\to V(M/P_1(M))/V(M/P_1(M))^{\lambda}\to 0.
\]
In particular, $\dim_KV(M)^{\lambda}=\dim_KV(P_1(M))+\dim_KV(M/P_1(M))^{\lambda}$, and the slope multiset of $V(M)^{\bullet}$ is equal to the disjoint union of those of $V(P_1(M))^{\bullet}$ and $V(M/P_1(M))^{\bullet}$.

A similar assertion holds for a $(\varphi,\nabla)$-module over $K[\![t]\!]_0$ or $\mathcal{E}$.
\end{lem}
\begin{proof}
It follows from Lemma \ref{lem:exact} and Remark \ref{rem:exact}.
\end{proof}

In the rest of this section, we gather several properties on the PBQ filtrations. The reader may consult here when it is needed.

\begin{prop}[{a refinement of Slope criterion}]
Let $R$ denote $K[\![t]\!]_0$ or $\mathcal{R}^{\bd}$. For a $(\varphi,\nabla)$-module $M$ over $R$, the following are equivalent.
\begin{enumerate}
\item $M$ is PBQ.
\item For any non-zero quotient $Q$ of $M$ as a $(\varphi,\nabla)$-module over $R$, $\lambda_{\max}(M_\mathcal{E})=\lambda_{\max}(Q_\mathcal{E})$.
\end{enumerate}
\end{prop}
\begin{proof}
By Slope criterion, (i) implies (ii). The negation of (i) implies that of (ii) by considering $Q=M/P_1(M)$.
\end{proof}

\begin{lem}\label{lem:inj sub}
Let $M$ be a $(\varphi,\nabla)$-module over $\mathcal{E}$, and $M'$ a $(\varphi,\nabla)$-submodule of $M$. Assume that $M'$ is PBQ and $\lambda_{\max}(M')=\lambda_{\max}(M)$. Then the composition
\[
M'/(M')^{0}\to M/M^0\to M/\cup_{\mu<\lambda_{\max}(M)}S_{\mu}(M)
\]
is injective.
\end{lem}
\begin{proof}
The above map coincides with the composition of canonical maps
\[
M'/(M')^{0}\to M'/\cup_{\mu<\lambda_{\max}(M')}S_{\mu}(M')\to M/\cup_{\mu<\lambda_{\max}(M)}S_{\mu}(M).
\]
The first map is isomorphism since $M'$ is PBQ, and the second map is injective by the strictness of Frobenius slope filtration and $\lambda_{\max}(M')=\lambda_{\max}(M)$, which implies the assertion.
\end{proof}

The PBQ filtration has the following quotient stability (cf. Corollary \ref{cor:qt inv of PBQ}).

\begin{prop}\label{prop:qt inv PBQ fil}
Let $M$ be a $(\varphi,\nabla)$-module over $K[\![t]\!]_0,\mathcal{R}^{\bd}$, or $\mathcal{E}$.
\begin{enumerate}
\item We have the equality as $(\varphi,\nabla)$-submodules of $M/P_j(M)$
\[
P_i(M/P_j(M))=P_{i+j}(M)/P_j(M)
\]
for all $i,j\in\mathbb{N}$.
\item Let $M'$ be a $(\varphi,\nabla)$-submodule of $M$ with $M'\neq M$. Let $i\ge 1$ be the minimum natural number such that $P_{i-1}(M)+M'\neq P_i(M)+M'$. Then we have the equality as $(\varphi,\nabla)$-submodules of $M/M'$
\[
P_1(M/M')=(P_{i}(M)+M')/M'.
\]
\item Let $M'$ be a $(\varphi,\nabla)$-submodule of $M$ with $M'\neq M$. Let $i(1)<\dots<i(j)$ be the sequence of natural numbers $\ge 1$ defined by
\[
\{i;P_{i-1}(M)+M'\neq P_i(M)+M'\}=\{i(1),\dots,i(j)\}.
\]
Set $i(0)=0$. Then we have the equality as $(\varphi,\nabla)$-submodules of $M/M'$
\[
P_k(M/M')=(P_{i(k)}(M)+M')/M'
\]
for all $0\le k\le j$.
\end{enumerate}
\end{prop}
\begin{proof}
Since the PBQ filtration over $K[\![t]\!]_0$ and $\mathcal{R}^{\bd}$ is defined by descent from $\mathcal{E}$, we have only to prove the assertion for $\mathcal{E}$.
\begin{enumerate}
\item We proceed by induction on $i$. When $i=0$, there is nothing to prove. By the induction hypothesis,
\begin{equation}\label{eq:maxE1}
P_{i-1}(M/P_j(M))=P_{i+j-1}(M)/P_j(M).
\end{equation}
By (\ref{eq:maxE1}), it suffices to prove
\[
P_i(M/P_j(M))/P_{i-1}(M/P_j(M))=(P_{i+j}(M)/P_j(M))/(P_{i+j-1}(M)/P_j(M))
\]
By definition,
\begin{equation}\label{eq:maxE2}
P_1(M/P_{i+j-1}(M))=P_{i+j}(M)/P_{i+j-1}(M).
\end{equation}
By identifying $M/P_{i+j-1}(M)$ as $(M/P_j(M))/(P_{i+j-1}(M)/P_j(M))$, (\ref{eq:maxE2}) implies
\[
P_1\Bigl((M/P_j(M))/(P_{i+j-1}(M)/P_j(M))\Bigr)=(P_{i+j}(M)/P_j(M))/(P_{i+j-1}(M)/P_j(M)).
\]
Then the assertion follows from
\begin{align*}
P_1\Bigl((M/P_j(M))/(P_{i+j-1}(M)/P_j(M))\Bigr)&=P_1\Bigl((M/P_j(M))/P_{i-1}(M/P_j(M))\Bigr)\\
&=P_i(M/P_j(M))/P_{i-1}(M/P_j(M)),
\end{align*}
where the first and second equalities follow from (\ref{eq:maxE1}) and by definition respectively.
\item $\bullet$ The case $i=1$

Put $M''=M/M'$, and $\mathcal{P}=(P_1(M)+M')/M'$, which is a non-zero $(\varphi,\nabla)$-submodule of $M''$ by assumption. By the uniqueness of the maximally PBQ submodule (Proposition \ref{prop:existence max PBQ}), it suffices to prove that the composition
\[
\alpha:\mathcal{P}/\mathcal{P}^0\to M''/(M'')^0\to M''/\cup_{\mu<\lambda_{\max}(M'')}S_{\mu}(M'')
\]
is an isomorphism. Since $\mathcal{P}$ is a quotient of $P_1(M)$, $\mathcal{P}$ is PBQ by Corollary \ref{cor:qt inv of PBQ}. Hence $\lambda_{\max}(\mathcal{P})=\lambda_{\max}(P_1(M))=\lambda_{\max}(M)$ by Slope criterion and Definition \ref{dfn:PBQ fil gen}. Since $\mathcal{P}\subset M''$, we also have $\lambda_{\max}(M'')=\lambda_{\max}(\mathcal{P})$. By applying Lemma \ref{lem:inj sub} to $(M',M)=(\mathcal{P},M'')$, $\alpha$ is injective. The surjectivity of $\alpha$ follows from a commutative diagram
\[\xymatrix{
P_1(M)/P_1(M)^0\ar[r]^(.43){\beta}\ar[d]&M/\cup_{\mu<\lambda_{\max}(M)}S_{\mu}(M)\ar[d]^{\gamma}\\
\mathcal{P}/\mathcal{P}^0\ar[r]^(.3){\alpha}&M''/\cup_{\mu<\lambda_{\max}(M)}S_{\mu}(M''),
}\]
where $\beta$ is an isomorphism by the definition of $P_1(M)$, and $\gamma$ is the canonical surjection.

$\bullet$ The general case

By (i) and assumption, $P_1(M/P_{i-1}(M))=P_{i}(M)/P_{i-1}(M)\neq (M'+P_{i-1}(M))/P_{i-1}(M)$. By applying the case $i=1$ to $(M'+P_{i-1}(M))/P_{i-1}(M)\subset M/P_{i-1}(M)$, we obtain the assertion.
\item We proceed by induction on $k$. When $k=0$, the assertion is trivial. By the induction hypothesis,
\begin{equation}\label{eq:maxE4}
P_{k-1}(M/M')=(P_{i(k-1)}(M)+M')/M'.
\end{equation}
By (\ref{eq:maxE4}), it suffices to prove
\[
P_k(M/M')/P_{k-1}(M/M')=((P_{i(k)}(M)+M')/M')/((P_{i(k-1)}(M)+M')/M').
\]
By (ii),
\begin{equation}\label{eq:maxE5}
P_1\Bigl(M/(P_{i(k-1)}(M)+M')\Bigr)=(P_{i(k)}(M)+M')/(P_{i(k-1)}(M)+M').
\end{equation}
By identifying $(M/M')/((P_{i(k-1)}(M)+M')/M')$ as $M/(P_{i(k-1)}(M)+M')$, (\ref{eq:maxE5}) implies
\[
P_1\Bigl((M/M')/((P_{i(k-1)}(M)+M')/M')\Bigr)=((P_{i(k)}(M)+M')/M')/((P_{i(k-1)}(M)+M')/M').
\]
Then the assertion follows from
\[
P_1\Bigl((M/M')/((P_{i(k-1)}(M)+M')/M')\Bigr)=P_1\Bigl((M/M')/P_{k-1}(M/M')\Bigr)=P_k(M/M')/P_{k-1}(M/M'),
\]
where the first and second equalities follow from (\ref{eq:maxE4}) and by definition respectively.
\end{enumerate}
\end{proof}

The generic PBQ filtration is compatible with the base change via $\tau$.

\begin{lem}\label{lem:tau compat PBQ fil}
Let $M$ be a $(\varphi,\nabla)$-module over $\mathcal{E}$, and $\tau:\mathcal{E}\to\mathcal{E}[\![X-t]\!]_0$ as in \S \ref{subsec:phinabla3}. Then there exists a canonical isomorphism of $(\varphi,\nabla)$-modules over $\mathcal{E}[\![X-t]\!]_0$
\[
\tau^*(P_{\bullet}(M))\cong P_{\bullet}(\tau^*M).
\]
\end{lem}
\begin{proof}
We have only to prove $\tau^*(P_1(M))\cong P_1(\tau^*M)$. By the definition of $P_1(\tau^*M)$, it reduces to prove that $i^*\tau^*(P_1(M))\cong\iota^*(P_1(M))$ is the maximally PBQ submodule of $i^*\tau^*M\cong\iota^*M$ with notation as in Lemma \ref{lem:tau inv of PBQ}. By definition, there exists a canonical isomorphism $P_1(M)/P_1(M)^0\cong M/\cup_{\mu<\lambda_{\max}(M)}S_{\mu}(M)$. Since Frobenius slope filtration and the log-growth filtration for a $(\varphi,\nabla)$-module over $\mathcal{E}$ are compatible with the base change via $\iota$ (see the proof of Lemma \ref{lem:tau inv of PBQ}), we obtain an isomorphism
\[
\iota^*(P_1(M))/(\iota^*(P_1(M)))^0\cong \iota^*M/\cup_{\mu<\lambda_{\max}(\iota^*M)}S_{\mu}(\iota^*M),
\]
which implies the assertion.
\end{proof}

By using the PBQ filtration, we can prove:

\begin{prop}[{cf. \cite[Proposition 7.3]{CT2}}]\label{prop:(ii) implies (i)}
In Theorem \ref{conj:bd}, part (ii) implies part (i). A similar assertion holds in Theorems \ref{conj:CT} and \ref{conj:CTgen}.
\end{prop}
\begin{proof}
We prove only the first assertion: a similar proof works in the other cases.

It suffices to prove that the function $\lambda\mapsto\dim_KV(M)^{\lambda}$ is right continuous, and locally constant at any $\lambda\notin\mathbb{Q}$. By Lemma \ref{lem:coefficient Robba}, we may assume that $k$ is algebraically closed. By d\'evissage using Lemma \ref{lem:exact PBQ}, we may assume that $M$ is PBQ. Since we have $\dim_KV(M)^{\lambda}=\dim_KV(M\spcheck)-\dim_KS_{\lambda-\lambda_{\max}(M_{\mathcal{E}})}(V(M\spcheck))$ by Theorem \ref{conj:bd} (ii), we obtain the assertion by Theorem \ref{thm:slope1} (ii).
\end{proof}


\section{Dwork's conjecture on semicontinuity of log-growth Newton polygons}\label{subsec:Dwo}

In this section, we define the log-growth Newton polygon of $(\varphi,\nabla)$-modules over $\mathcal{R}^{\bd}$, and prove a semicontinuity property under specialization (Theorem \ref{thm:Dwo}). As a consequence, we obtain Dwork's conjecture $\mathbf{LGF}_{\mathrm{Dw}}$ formulated in \cite{CT2} (Corollary \ref{cor:LGFDW}).

\begin{dfn}
Let $M$ be a $(\varphi,\nabla)$-module of rank $n$ over $\mathcal{R}^{\bd}$ solvable in $\mathcal{R}_{\log}$. We define the {\it log-growth Newton polygon} $\NP(M)$ of $M$ as the Newton polygon $\NP(V(M)^{\bullet})$ of $V(M)^{\bullet}$. Let $\lambda_i(M)$ denote $\lambda_i(V(M)^{\bullet})$, i.e., $\lambda_1(M)\le\dots\le\lambda_n(M)$ is the slope multiset of $\NP(M)$. We put $b^{\nabla}(M)=\lambda_n(M)$, which is the maximum slope of $\NP(M)$.

We give a similar definition for a (log-)$(\varphi,\nabla)$-module over $K[\![t]\!]_0$ or $\mathcal{E}$.

In the literature, for a $(\varphi,\nabla)$-module $M$ over $K[\![t]\!]_0$, $\NP(M)$ and $\NP(M_{\mathcal{E}})$ are called the {\it special} and {\it generic log-growth Newton polygons} of $M$ respectively.
\end{dfn}

We gather some basic properties on the log-growth filtrations.

\begin{lem}\label{lem:right continuous}
Let $M$ be a $(\varphi,\nabla)$-module of rank $n$ over $\mathcal{R}^{\bd}$ solvable in $\mathcal{R}_{\log}$. Let $\lambda_1<\dots<\lambda_n$ be the slopes (without multiplicity) of $V(M)^{\bullet}$, or, equivalently, of $\Sol_{\bullet}(M)$. Then
\[
V(M)^{\lambda}=
\begin{cases}
V(M)&\text{if }\lambda\in (-\infty,\lambda_1),\\
V(M)^{\lambda_i}&\text{if }\lambda\in [\lambda_i,\lambda_{i+1})\text{ for }i=1,\dots,n-1,\\
0&\text{if }\lambda\in [\lambda_n,+\infty),
\end{cases}
\]
\[
\Sol_{\lambda}(M)=
\begin{cases}
0&\text{if }\lambda\in (-\infty,\lambda_1),\\
\Sol_{\lambda_i}(M)&\text{if }\lambda\in [\lambda_i,\lambda_{i+1})\text{ for }i=1,\dots,n-1,\\
\Sol(M)&\text{if }\lambda\in [\lambda_n,+\infty),
\end{cases}
\]
\[
m(\lambda_i)=
\begin{cases}
\dim_KV(M)-\dim_KV(M)^{\lambda_i}&\text{if }i=1,\\
\dim_KV(M)^{\lambda_{i-1}}-\dim_KV(M)^{\lambda_i}&\text{if }i=2,\dots,n-1,\\
\dim_KV(M)^{\lambda_{n-1}}&\text{if }i=n.
\end{cases}
\]
A similar assertion holds for $M$ a (log-)$(\varphi,\nabla)$-module over $K[\![t]\!]_0$ or $\mathcal{E}$.
\end{lem}
\begin{proof}
It is an immediate consequence of the right continuity of the log-growth filtration.
\end{proof}

\begin{cor}\label{cor:right continuous}
Let $M$ be a $(\varphi,\nabla)$-module over $\mathcal{R}^{\bd}$ solvable in $\mathcal{R}_{\log}$. Then
\[
b^{\nabla}(M)=\min\{\lambda;V(M)^{\lambda}=0\}=\min\{\lambda;\Sol_{\lambda}(M)=\Sol(M)\}.
\]

A similar assertion holds for $M$ a (log-)$(\varphi,\nabla)$-module over $K[\![t]\!]_0$ or $\mathcal{E}$.
\end{cor}

By using the PBQ filtration, we can calculate the slope multisets of log-growth filtrations.

\begin{prop}\label{thm:Dwo2}
Assume that $k$ is perfect. Let $M$ be a $(\varphi,\nabla)$-module over $\mathcal{R}^{\bd}$ solvable in $\mathcal{R}_{\log}$. Let $P_{\bullet}(M)$ denote the PBQ filtration on $M$, and $r$ the minimum natural number such that $P_r(M)=M$. Then the slope multiset of $V(M)^{\bullet}$ coincides with the multiset
\[
\coprod_{i=0}^{r-1}\{\mu+\lambda_{\max}((P_{i+1}(M)/P_i(M))_{\mathcal{E}});\mu\in\Lambda_i\},
\]
where $\Lambda_i$ denotes the slope multiset of $S_{\bullet}(V(P_{i+1}(M)/P_i(M)))$.

A similar assertion holds for a (log-)$(\varphi,\nabla)$-module over $K[\![t]\!]_0$ or $\mathcal{E}$.
\end{prop}
\begin{proof}
$\bullet$ The case of $\mathcal{R}^{\bd}$

We proceed by induction on $r$. When $r=1$, i.e., $M$ is PBQ, the assertion follows from Theorem \ref{conj:bd} (ii). In the general case, the slope multiset of $V(M)^{\bullet}$ is the disjoint union of those of $V(P_1(M))^{\bullet}$ and $V(M/P_1(M))^{\bullet}$ by Lemma \ref{lem:exact PBQ}. Since $P_i(M/P_1(M))=P_{i+1}(M)/P_1(M)$ for $i\ge 1$ by Proposition \ref{prop:qt inv PBQ fil} (i), the slope multisets of $V(P_1(M))^{\bullet}$ and $V(M/P_1(M))^{\bullet}$ are, by the induction hypothesis,
\[
\{\mu+\lambda_{\max}(P_1(M)_{\mathcal{E}});\mu\in\Lambda_0\},\ \coprod_{i=0}^{r-2}\{\mu+\lambda_{\max}((P_{i+2}(M)/P_{i+1}(M))_{\mathcal{E}});\mu\in\Lambda_{i+1}\},
\]
respectively, which implies the assertion.

$\bullet$ The cases of $K[\![t]\!]_0$ or $\mathcal{E}$

By using the principle explained in Remark \ref{rem:principle}, the assertion over $K[\![t]\!]_0$ reduces to that over $\mathcal{R}^{\bd}$. By the principle, together with Lemmas \ref{lem:evaluation} and \ref{lem:tau compat PBQ fil}, the assertion over $\mathcal{E}$ reduces to that over $\mathcal{E}[\![X-t]\!]_0$. Alternatively, a similar proof in the case of $\mathcal{R}^{\bd}$ works.
\end{proof}

\begin{lem}\label{lem:Dwo3}
Let $M$ be a $(\varphi,\nabla)$-module over $\mathcal{R}^{\bd}$ solvable in $\mathcal{R}_{\log}$. Let $S_{\bullet}(M_{\mathcal{R}})$ denote the slope filtration for Frobenius structures of $M_{\mathcal{R}}$ in the sense of \cite[Definition 5.1.1]{Tsu}: the definition of the slopes (\cite[Definition 3.1.5]{Tsu}) is compatible that in \S \ref{sec:slope}. Then there exists a canonical isomorphism
\[
\mathbf{V}(S_{\bullet}(M_{\mathcal{R}}))\cong S_{\bullet}(V(M)).
\]
In particular, the slope multiset of $S_{\bullet}(V(M))$ coincides with that of $S_{\bullet}(M_{\mathcal{R}})$.
\end{lem}
\begin{proof}
It follows from the uniqueness of Frobenius slope filtration on $\mathbf{V}(M_{\mathcal{R}})\cong V(M)$.
\end{proof}

The following is an analogue of Grothendieck-Katz specialization theorem for Frobenius Newton polygons.

\begin{thm}\label{thm:Dwo}
Let $M$ be a $(\varphi,\nabla)$-module over $\mathcal{R}^{\bd}$ solvable in $\mathcal{R}_{\log}$. Then the log-growth Newton polygon of $M$ lies on or above the log-growth Newton polygon of $M_{\mathcal{E}}$ with the same endpoints.

A similar assertion holds for a (log-)$(\varphi,\nabla)$-module over $K[\![t]\!]_0$ or $\mathcal{E}$.
\end{thm}
\begin{proof}
By Remark \ref{rem:principle}, we have only to prove the assertion in the case of $\mathcal{R}^{\bd}$. By Lemmas \ref{lem:coefficient Amice} and \ref{lem:coefficient Robba}, we may assume that $k$ is algebraically closed. Let notation be as in Theorem \ref{thm:Dwo2}. We define $N=\oplus_{i=0}^{r-1}G_i(q^{a_i})$, where we put $G_i=P_{i+1}(M)/P_i(M),a_i=\lambda_{\max}((P_{i+1}(M)/P_i(M))_{\mathcal{E}})$. Then $N$ is a $(\varphi,\nabla)$-module over $\mathcal{R}^{\bd}$ solvable in $\mathcal{R}_{\log}$. The slope multiset of $V(M)^{\bullet}$ coincides with that of $S_{\bullet}(V(N))$ by Proposition \ref{thm:Dwo2}, which also coincides with that of $S_{\bullet}(N_{\mathcal{R}})$ by Lemma \ref{lem:Dwo3}. The slope multiset of $M_{\mathcal{E}}^{\bullet}$ coincides with that of $S_{\bullet}(N_{\mathcal{E}})$ by Proposition \ref{thm:Dwo2}. Therefore the assertion follows from the semicontinuity theorem for Frobenius Newton polygons of $N$ (\cite[Theorem 16.4.6]{pde}).
\end{proof}

By comparing the maximum slopes of Newton polygons, we obtain an analogue of Christol's transfer theorem (\cite[Proposition 4.3]{CT}).

\begin{cor}
Let $M$ be a $(\varphi,\nabla)$-module over $\mathcal{R}^{\bd}$ solvable in $\mathcal{R}_{\log}$. Then
\[
b^{\nabla}(M)\le b^{\nabla}(M\otimes_{\mathcal{R}^{\bd}}\mathcal{E}).
\]

A similar assertion holds for a (log-)$(\varphi,\nabla)$-module over $K[\![t]\!]_0$.
\end{cor}

Recall that Dwork's conjecture $\mathbf{LGF}_{\mathrm{Dw}}$ (\cite[Conjecture 2.7]{CT2}) asserts that for $M$ a $(\varphi,\nabla)$-module over $K[\![t]\!]_0$, the log-growth Newton polygon of $M$ lies on or above that of $M_{\mathcal{E}}$ with the same endpoints.

\begin{cor}\label{cor:LGFDW}
Dwork's conjecture $\mathbf{LGF}_{\mathrm{Dw}}$ is true.
\end{cor}

\begin{rem}[{cf. \cite[Remark 2.8]{CT2}}]\label{rem:Dwo4}
We discuss differences between the conjecture $\mathbf{LGF}_{\mathrm{Dw}}$ and \cite[Conjecture 1.1.1]{And}, which are ``different forms'' of Dwork's conjecture below. Let $M$ be a $\nabla$-module of rank $n$ over $K[\![t]\!]_0$ solvable in $K\{t\}$. Dwork originally defines the log-growth Newton polygon of $M$ as the lower convex polygon, where the multiplicity of the slope $\lambda$ is equal to $\dim_K\Sol_{\lambda+\varepsilon}(M)-\dim_K\Sol_{\lambda-\varepsilon}(M)$ with sufficiently small $\varepsilon>0$ (\cite[(4) in p. 368]{Dw}). He also defines the log-growth Newton polygon of $M_{\mathcal{E}}$ as that of $\tau^*M$. Then Dwork conjectures (\cite[Conjecture 2]{DwIII}) that ``the log-growth Newton polygon of $M$ is lies on or above that of $M_{\mathcal{E}}$'': precisely speaking, the conjecture makes sense only after fixing the left or right endpoints of the log-growth Newton polygons of $M$ and $M_{\mathcal{E}}$, however, the author cannot find an information how to define the endpoints in Dwork's articles. Andr\'e defines his log-growth Newton polygons $\mathcal{NP}(M)$ and $\mathcal{NP}(M_{\mathcal{E}})$ by fixing the right endpoints at $(n,0)$ (\cite[3.3]{And}), then proves Dwork's conjecture without the coincidence of the left endpoints (\cite[Theorem 4.1.1]{And}). In \cite{Ohkb}, the author constructs $M$ of rank two such that the left endpoints of $\mathcal{NP}(M)$ and $\mathcal{NP}(M_{\mathcal{E}})$ do not coincide.

Note that if $M$ admits a Frobenius structure, then our log-growth Newton polygon of $M$ (resp. $M_{\mathcal{E}}$) coincides with Andr\'e's log-growth Newton polygon of $M$ (resp. $M_{\mathcal{E}}$) after moving $-h$ (resp. $-h_{\mathcal{E}}$) in the $y$-direction, where $h$ (resp. $h_{\mathcal{E}}$) denotes the $y$-coordinate of the left endpoint of $\NP(M)$ (resp. $\NP(M_{\mathcal{E}})$). Since we have $h=h_{\mathcal{E}}$ by Corollary \ref{cor:LGFDW}, the left endpoints of $\mathcal{NP}(M)$ and $\mathcal{NP}(M_{\mathcal{E}})$ also coincides, hence, the conjecture $\mathbf{LGF}_{\mathrm{Dw}}$ can be seen as a refinement of Andr\'e's theorem under the existence of Frobenius structures. However it is not clear that one can deduce (even ``lies above'' part of ) the conjecture $\mathbf{LGF}_{\mathrm{Dw}}$ from Andr\'e's theorem, and vice versa. Thus the author thinks that Andr\'e's theorem and Corollary \ref{cor:LGFDW}, both are regarded as Dwork's conjectures, are of different natures.
\end{rem}


\section{Converse theorem}\label{sec:conv}

It is natural to ask the necessity of the PBQ hypothesis in parts (ii) of Theorems \ref{conj:CT}, \ref{conj:CTgen}, or \ref{conj:bd}. In this section, we prove that the PBQ hypothesis is necessary in all the cases.

\begin{lem}[{Converse theorem over $\mathcal{E}$}]\label{lem:conv gen}
Let $M$ be a $(\varphi,\nabla)$-module over $\mathcal{E}$. Then, $M$ is PBQ if and only if $M^0=(S_{-\lambda_{\max}(M)}(M\spcheck))^{\perp}$. In particular, if
\[
M^{\lambda}=(S_{\lambda-\lambda_{\max}(M)}(M\spcheck))^{\perp}\ \forall\lambda,
\]
then $M$ is PBQ.
\end{lem}
\begin{proof}
Assume $M^0=(S_{-\lambda_{\max}(M)}(M\spcheck))^{\perp}$. By Lemma \ref{lem:slope5} (II)-(iv), $M^0=\cup_{\mu<\lambda_{\max}(M)}S_{\mu}(M)$. Hence $M/M^0$ is pure of slope $\lambda_{\max}(M)$ as a $\varphi$-module over $\mathcal{E}$, in particular, $M$ is PBQ.

Assume that $M$ is PBQ. Then $M/M^0$ is pure of slope $\lambda_{\max}(M)$ by Corollary \ref{cor:new2}. Hence the image of $\cup_{\mu<\lambda_{\max}(M)}S_{\mu}(M)$ under the projection $M\to M/M^0$ is zero, i.e., $\cup_{\mu<\lambda_{\max}(M)}S_{\mu}(M)\subset M^0$. By Proposition \ref{prop:CT2}, we conclude $\cup_{\mu<\lambda_{\max}(M)}S_{\mu}(M)=M^0$.
\end{proof}

\begin{rem}
By the equivalence above, we can rephrase Theorem \ref{conj:CTgen} (ii) as the following form: if the equality $M^{\lambda}=(S_{\lambda-\lambda_{\max}(M)}(M))^{\perp}$ holds for $\lambda=0$, then it holds for an arbitrary $\lambda$.
\end{rem}

\begin{lem}[{Converse theorem over $\mathcal{R}^{\bd}$}]\label{lem:conv bd}
Let $M$ be a $(\varphi,\nabla)$-module over $\mathcal{R}^{\bd}$ solvable in $\mathcal{R}_{\log}$. If
\[
\Sol_{\lambda}(M)=S_{\lambda-\lambda_{\max}(M_{\mathcal{E}})}(\Sol(M))\ \forall\lambda,
\]
then $M$ is PBQ.
\end{lem}

Note that an analogue of Converse theorem for a (log-)$(\varphi,\nabla)$-module over $K[\![t]\!]_0$ follows from Lemma \ref{lem:conv bd} thanks to Lemma \ref{lem:invariance log-growth}.

\begin{proof}
We may assume $M\neq 0$. Suppose that $M$ is not PBQ. Put $M'=P_1(M),M''=M/P_1(M)$. Recall that $\lambda_{\max}(M''_{\mathcal{E}})<\lambda_{\max}(M_{\mathcal{E}})=\lambda_{\max}(M'_{\mathcal{E}})$, and $M'$ is PBQ. Hence we have, by Theorem \ref{conj:bd} (ii) and Proposition \ref{prop:main} respectively,
\[
\Sol_{\lambda}(M')=S_{\lambda-\lambda_{\max}(M_{\mathcal{E}})}(\Sol(M'))\ \forall\lambda,
\]
\begin{equation}\label{eq:conv2}
S_{\lambda-\lambda_{\max}(M_{\mathcal{E}})}(\Sol(M''))\subset S_{\lambda-\lambda_{\max}(M''_{\mathcal{E}})}(\Sol(M''))\subset \Sol_{\lambda}(M'')\ \forall\lambda.
\end{equation}
Therefore we obtain a canonical commutative diagram with exact rows by Lemmas \ref{lem:exact} and \ref{lem:slope5} (II)-(ii),
\[\xymatrix{
0\ar[r]&S_{\lambda-\lambda_{\max}(M_{\mathcal{E}})}(\Sol(M''))\ar[r]\ar[d]&S_{\lambda-\lambda_{\max}(M_{\mathcal{E}})}(\Sol(M))\ar[r]\ar@{=}[d]&S_{\lambda-\lambda_{\max}(M_{\mathcal{E}})}(\Sol(M'))\ar[r]\ar@{=}[d]&0\\
0\ar[r]&\Sol_{\lambda}(M'')\ar[r]&\Sol_{\lambda}(M)\ar[r]&\Sol_{\lambda}(M').&
}\]
By the snake lemma, the left vertical arrow is an equality. Hence, by (\ref{eq:conv2}),
\[
S_{\lambda-\lambda_{\max}(M_{\mathcal{E}})}(\Sol(M''))=S_{\lambda-\lambda_{\max}(M''_{\mathcal{E}})}(\Sol(M''))\ \forall\lambda,
\]
which contradicts to the uniqueness of Frobenius slope filtration.
\end{proof}


\section{Tensor compatibility}\label{sec:ten}

As pointed out by Andr\'e (\cite[1.3]{And}), the log-growth filtrations for $\nabla$-modules (even assuming Frobenius structures) are not strongly compatible with the tensor product; see Remark \ref{rem:non-compatible tensor} for details. In this section, we prove a weak tensor compatibility of the log-growth filtration (Proposition \ref{prop:weak tensor}). As a corollary, we obtain the additivity of $b^{\nabla}$ (Corollary \ref{cor:tensor bound}).

We start by noting that in the category of $(\varphi,\nabla)$-modules over $\mathcal{E}$, the subcategory of PBQ objects is not closed under the tensor product: an explicit example is given by Example 6.2.5, which is generalized as follows.

\begin{lem}
Let $M$ be a non-zero $(\varphi,\nabla)$-module over $\mathcal{E}$. If $M$ is not pure as a $\varphi$-module over $\mathcal{E}$, then $M\otimes_{\mathcal{E}}M\spcheck$ is not PBQ.
\end{lem}
\begin{proof}
We have $\lambda_{\max}(M\otimes_{\mathcal{E}}M\spcheck)=\lambda_{\max}(M)+\lambda_{\max}(M\spcheck)=\lambda_{\max}(M)-\lambda_{\min}(M)>0$. We regard $\mathcal{E}$ as a quotient of $M\otimes_{\mathcal{E}}M\spcheck$ via the canonical pairing $M\otimes_{\mathcal{E}}M\spcheck\to\mathcal{E}$. Since $\lambda_{\max}(\mathcal{E})=0$, $M\otimes_{\mathcal{E}}M\spcheck$ is not PBQ by Proposition \ref{prop:key}.
\end{proof}

We will refine the proof of Theorem \ref{conj:bd} (ii) in \S \ref{sec:pf} to prove Proposition \ref{prop:weak tensor}.

\begin{dfn}
Let $M$ be a $(\varphi,\nabla)$-module over $\mathcal{R}^{\bd}$ solvable in $\mathcal{R}_{\log}$. Recall that a solution $f\in \Sol(M)$ is of log-growth $\lambda$ for $\lambda\in\mathbb{R}$ if $f\in \Sol_{\lambda}(M)$, i.e., $f(M)\subset \Fil_{\lambda}\mathcal{R}_{\log}$. We say that $f$ is {\it exactly of log-growth} $\lambda$ if $f$ is of log-growth $\lambda$, and not of log-growth $\mu$ for any $\mu<\lambda$.

Note that if $f$ is exactly of log-growth $\lambda$, then $\lambda$ is a slope of the filtration $\Sol_{\bullet}(M)$, and hence, a slope of $V(M)^{\bullet}$.
\end{dfn}

\begin{construction}\label{const:ten}
We fix a $\varphi$-equivariant embedding $\mathcal{R}_{\log}\hookrightarrow\tilde{\mathcal{R}}_{\log}$ in Definition \ref{dfn:log embedding}. Let $M$ be a $\varphi$-module over $\mathcal{R}^{\bd}$, and $f:M\to\mathcal{R}_{\log}$ an $\mathcal{R}^{\bd}$-linear map. Let $\tilde{M}$ denote the $\varphi$-module $M\otimes_{\mathcal{R}^{\bd}}\tilde{\mathcal{R}}^{\bd}$ over $\tilde{\mathcal{R}}^{\bd}$, and $\tilde{f}:\tilde{M}\to\tilde{\mathcal{R}}_{\log}$ denote the composition of
\[
M\otimes_{\mathcal{R}^{\bd}}\tilde{\mathcal{R}}^{\bd}\xrightarrow[]{f\otimes \id}\mathcal{R}_{\log}\otimes_{\mathcal{R}^{\bd}}\tilde{\mathcal{R}}^{\bd}\hookrightarrow\tilde{\mathcal{R}}_{\log},
\]
where the second map is the multiplication map given in Lemma \ref{lem:log inj}. Note that the following.
\begin{enumerate}
\item[$\bullet$] If $f$ is injective (resp. $\varphi$-equivariant), then so is $\tilde{f}$.
\item[$\bullet$] If $f$ is of log-growth $\lambda$, then
\[
\tilde{f}(\tilde{M})\subset \Fil_{\lambda}\tilde{\mathcal{R}}_{\log}
\]
by a similar argument as in the proof of Proposition \ref{prop:pf}.
\end{enumerate}
\end{construction}

The following proposition can be seen as a generalization of Proposition \ref{prop:pf}.

\begin{prop}\label{prop:refine}
Assume that $k$ is algebraically closed. Let $M$ be a $(\varphi,\nabla)$-module over $\mathcal{R}^{\bd}$ solvable in $\mathcal{R}_{\log}$, and $\lambda$ a slope of $\Sol_{\bullet}(M)$. Then there exist $f\in \Sol_{\lambda}(M)$ and $\tilde{m}\in\tilde{M}$ such that $\tilde{f}(\tilde{m})$ is a $d$-eigenvector of slope $\lambda$ (for some $d\ge 1$). Moreover, any such $f$ is exactly of log-growth $\lambda$.
\end{prop}
\begin{proof}
We first prove the last assertion. Suppose the contrary, i.e., there exists $\mu<\lambda$ such that $f$ is of log-growth $\mu$. By Construction \ref{const:ten},  we have $\tilde{f}(\tilde{M})\subset \Fil_{\mu}\tilde{\mathcal{R}}_{\log}$. By assumption and Lemma \ref{lem:calc log-growth}, $\tilde{f}(\tilde{m})$ is exactly of log-growth $\lambda$, which is a contradiction.

We will prove the first assertion. Since $k$ is algebraically closed, there exists a solution $f\in \Sol(M)$ such that:
\begin{enumerate}
\item[(a)] $f$ is exactly of log-growth $\lambda$;
\item[(b)] $f$ is a Frobenius $d$-eigenvector.
\end{enumerate}
In fact, there exists a (non-canonical) splitting of $\varphi$-modules $\Sol_{\lambda}(M)\cong (\cup_{\mu<\lambda}\Sol_{\mu}(M))\oplus (\Sol_{\lambda}(M)/(\cup_{\mu<\lambda}\Sol_{\mu}(M)))$. Then we choose a $d$-eigenvector in the second summand as $f$. We will prove that $f$ satisfies the desired condition by repeating the proof of Proposition \ref{prop:pf}.

$\bullet$ The case $\varphi(f)=f$

Since $f:M\to\mathcal{R}_{\log}$ is $\varphi,\nabla$-equivariant by assumption, $\ker{f}$ is a $(\varphi,\nabla)$-submodule of $M$. Denote by $M''$ the quotient $M/\ker{f}$, which is a $(\varphi,\nabla)$-module over $\mathcal{R}^{\bd}$ solvable in $\mathcal{R}_{\log}$, and by $f'':M''\hookrightarrow\mathcal{R}_{\log}$ the induced map by $f$. Then $f''\in\Sol(M)$ is exactly of log-growth $\lambda$ by (a) and $f''(M'')=f(M)$. We also have $\varphi(f'')=f''$ by definition. Therefore, after replacing $(M,f)$ by $(M'',f'')$, we may also assume that $f$ is injective. In this case, $\tilde{f}$ is $\varphi$-equivariant and injective.

We choose a Frobenius $d$-eigenvector $\tilde{m}\in\tilde{M}$ of slope $\lambda_{\max}(M_{\mathcal{E}})$ (Proposition \ref{prop:maximum eigen}). Then $\tilde{f}(\tilde{m})$ is also a Frobenius $d$-eigenvector of slope $\lambda_{\max}(M_{\mathcal{E}})$. By Lemma \ref{lem:calc log-growth}, $\tilde{f}(\tilde{m})$ is exactly of log-growth $\lambda_{\max}(M_{\mathcal{E}})$. It suffices to prove $\lambda_{\max}(M_{\mathcal{E}})=\lambda$. Since $f\in S_0(\Sol(M))\subset \Sol_{\lambda_{\max}(M_{\mathcal{E}})}(M)$ by Proposition \ref{prop:main}, $f$ is of log-growth $\lambda_{\max}(M_{\mathcal{E}})$. Hence $\lambda\le\lambda_{\max}(M_{\mathcal{E}})$ by (a). We also have $\tilde{f}(\tilde{M})\subset\Fil_{\lambda}\tilde{\mathcal{R}}_{\log}$ by (a), in particular, $\tilde{f}(\tilde{m})$ is of log-growth $\lambda$. Hence $\lambda_{\max}(M_{\mathcal{E}})\le\lambda$. Thus $\lambda_{\max}(M_{\mathcal{E}})=\lambda$.

$\bullet$ The general case

Write $\varphi^d(f)=cf$ for some $c\in K^{\times}$. Put $N=([d]_*M)(c)$, which is a $(\varphi^d,\nabla)$-module over $\mathcal{R}^{\bd}$. Since $T:M\to N;m\mapsto m\otimes e_c$ is an isomorphism of $\nabla$-modules over $\mathcal{R}^{\bd}$, $N$ is also solvable in $\mathcal{R}_{\log}$. Moreover, $T$ induces an isomorphism $T^*:\Sol(N)\cong \Sol(M);f\mapsto f\circ T$, which also induces isomorphisms of filtrations $\Sol_{\bullet}(N)\cong \Sol_{\bullet}(M)$. Put $f':=(T^*)^{-1}f\in \Sol(N)$. Then $\varphi^d(f')=f'$, and $f'$ is exactly of log-growth $\lambda$. By applying the previous case to $f'$, we obtain $\tilde{n}\in\tilde{N}$ such that $\tilde{f}'(\tilde{n})$ is exactly of log-growth $\lambda$. Since $T$ naturally extends to an $\tilde{\mathcal{R}}^{\bd}$-linear isomorphism $\tilde{T}:\tilde{M}\to\tilde{N}$ satisfying $\tilde{f}=\tilde{f}'\circ\tilde{T}$, the element $\tilde{m}=\tilde{T}^{-1}(\tilde{n})$ satisfies the desired condition.
\end{proof}

\begin{lemdfn}\label{lem:product solution}
Let $M,N$ be $(\varphi,\nabla)$-modules over $\mathcal{R}^{\bd}$ solvable in $\mathcal{R}_{\log}$. Then there exists a canonical isomorphism of $\varphi$-modules over $K$
\[
\Psi:\Sol(M)\otimes_K\Sol(N)\to \Sol(M\otimes_{\mathcal{R}^{\bd}}N);g\otimes h\mapsto (m\otimes n\mapsto g(m)h(n)).
\]
Moreover,
\[
\Psi(\Sol_{\lambda}(M)\otimes_K\Sol_{\mu}(N))\subset \Sol_{\lambda+\mu}(M\otimes_{\mathcal{R}^{\bd}}N)\ \forall\lambda,\mu\in\mathbb{R}.
\]
\end{lemdfn}
\begin{proof}
The map $\Psi$ is identified with the canonical map $\psi:V(M\spcheck)\otimes_K V(N\spcheck)\to V(M\spcheck\otimes_{\mathcal{R}^{\bd}}N\spcheck)$. Since $\psi$ is injective, $\psi$ is an isomorphism by a dimension count. The second assertion is a consequence of Lemma \ref{lem:property log} (i).
\end{proof}

The log-growth filtration satisfies a weak tensor compatibility as follows.

\begin{prop}\label{prop:weak tensor}
Let $M,N$ be $(\varphi,\nabla)$-modules over $\mathcal{R}^{\bd}$ solvable in $\mathcal{R}_{\log}$. Then the sum of a slope of $V(M)^{\bullet}$ and a slope of $V(N)^{\bullet}$ is a slope of $V(M\otimes_{\mathcal{R}^{\bd}}N)^{\bullet}$ (without multiplicity).

A similar assertion holds for a (log-)$(\varphi,\nabla)$-module over $K[\![t]\!]_0$ or $\mathcal{E}$.
\end{prop}
\begin{proof}
By Remark \ref{rem:principle}, we have only to prove the assertion in the case of $\mathcal{R}^{\bd}$. By duality, we may replace $V(\cdot)^{\bullet}$ by $\Sol_{\bullet}(\cdot)$. By Lemma \ref{lem:coefficient Robba}, we may assume that $k$ is algebraically closed. Let $\lambda,\mu$ be slopes of $\Sol_{\bullet}(M),\Sol_{\bullet}(N)$ respectively. By Proposition \ref{prop:refine}, we choose $g\in \Sol_{\lambda}(M)$ and $\tilde{m}\in\tilde{M}$ (resp. $h\in \Sol_{\mu}(N)$ and $\tilde{n}\in\tilde{N}$) such that $\tilde{g}(\tilde{m})$ (resp. $\tilde{h}(\tilde{n})$) is a $d$-eigenvector of slope $\lambda$ (resp. $\mu$). Put $f=\Psi(g\otimes h)$. Then $f\in Sol_{\lambda+\mu}(M\otimes_{\mathcal{R}^{\bd}}N)$ by Proposition \ref{prop:refine}, and $\tilde{f}(\tilde{m}\otimes\tilde{n})=\tilde{g}(\tilde{m})\cdot\tilde{h}(\tilde{n})$ is a $d$-eigenvector of slope $\lambda+\mu$. By Proposition \ref{prop:refine} again, $f$ is exactly of log-growth $\lambda+\mu$. In particular, $\lambda+\mu$ is a slope of $\Sol_{\bullet}(M\otimes_{\mathcal{R}^{\bd}}N)$.
\end{proof}

\begin{rem}\label{rem:non-compatible tensor}
An analogous assertion with multiplicity fails: in Example 6.1.5, the slope multisets of $V(M_{\mu})^{\bullet},V(M_{\mu}\spcheck)^{\bullet}$ are both $\{0,\mu\}$, while that of $V(M_{\mu}\otimes_{K[\![t]\!]_0}M_{\mu}\spcheck)^{\bullet}$ is $\{0,0,\mu,2\mu\}$.
\end{rem}

As a consequence, $b^{\nabla}$ is additive under tensor products.

\begin{cor}\label{cor:tensor bound}
Let $M,N$ be $(\varphi,\nabla)$-modules over $\mathcal{R}^{\bd}$ solvable in $\mathcal{R}_{\log}$. Then
\[
b^{\nabla}(M\otimes_{\mathcal{R}^{\bd}}N)=b^{\nabla}(M)+b^{\nabla}(N).
\]

A similar assertion holds for $M,N$ (log-)$(\varphi,\nabla)$-modules over $K[\![t]\!]_0$ or $\mathcal{E}$.
\end{cor}
\begin{proof}
By Remark \ref{rem:principle}, we have only to prove the assertion in the case of $\mathcal{R}^{\bd}$. By duality, we may use $\Sol_{\bullet}(\cdot)$ instead of $V(\cdot)^{\bullet}$. By Lemma \ref{lem:product solution} and Corollary \ref{cor:right continuous},
\[
\Sol(M\otimes_{\mathcal{R}^{\bd}}N)=\Psi(\Sol(M)\otimes_K\Sol(N))=\Psi(\Sol_{b^{\nabla}(M)}(M)\otimes_K \Sol_{b^{\nabla}(N)}(N))\subset \Sol_{b^{\nabla}(M)+b^{\nabla}(N)}(M\otimes_{\mathcal{R}^{\bd}}N).
\]
Hence $b^{\nabla}(M\otimes_{\mathcal{R}^{\bd}}N)\le b^{\nabla}(M)+b^{\nabla}(N)$. By Proposition \ref{prop:weak tensor}, $b^{\nabla}(M)+b^{\nabla}(N)$ is a slope of $\Sol_{\bullet}(M\otimes_{\mathcal{R}^{\bd}}N)$, in particular, $b^{\nabla}(M)+b^{\nabla}(N)\le b^{\nabla}(M\otimes_{\mathcal{R}^{\bd}}N)$.
\end{proof}

The following supplementary results are not used in the rest of this paper. The reader may skip here for the first time of reading.

As an application of Proposition \ref{prop:weak tensor}, we can give a sufficient condition that the tensor product of two $(\varphi,\nabla)$-modules is PBQ.

\begin{prop}\label{prop:product PBQ}
Let $M,N$ be $(\varphi,\nabla)$-modules over $\mathcal{R}^{\bd}$ solvable in $\mathcal{R}_{\log}$. Assume that
\begin{enumerate}
\item[(a)] $M,N$ are PBQ,
\item[(b)] each Frobenius slope of $V(M\otimes_{\mathcal{R}^{\bd}}N)$ has multiplicity one.
\end{enumerate}
Then $M\otimes_{\mathcal{R}^{\bd}}N$ is PBQ.

A similar assertion holds for $M,N$ (log-)$(\varphi,\nabla)$-modules over $K[\![t]\!]_0$ or $\mathcal{E}$.
\end{prop}

\begin{lem}\label{lem:ten8}
Let $V$ be a finite dimensional vector space over a field. Let $V_{1,\bullet},V_{2,\bullet}$ be increasing filtrations on $V$ by subspaces, which is separated, exhaustive, and right continuous. Assume $V_{1,\bullet}\subset V_{2,\bullet}$. If the slope multiset of $V_{1,\bullet},V_{2,\bullet}$ coincide, then $V_{1,\bullet}=V_{2,\bullet}$.
\end{lem}
\begin{proof}
Obvious.
\end{proof}

\begin{proof}[{Proof of Proposition \ref{prop:product PBQ}}]
By Remark \ref{rem:principle}, we have only to prove the assertion in the case of $\mathcal{R}^{\bd}$. By duality, each Frobenius slope of $\Sol(M\otimes_{\mathcal{R}^{\bd}}N)$ has multiplicity one. Let $\mu_1(M)\le\dots\le\mu_m(M)$ (resp. $\mu_1(N)\le\dots\le\mu_n(N)$) denote the slope multiset of $S_{\bullet}(\Sol(M))$ (resp. $S_{\bullet}(\Sol(N))$) with increasing order. By Theorem \ref{conj:bd} (ii),
\[
\lambda_i(M)=\mu_i(M)+\lambda_{\max}(M_{\mathcal{E}})\ \forall i,
\]
\[
\lambda_j(N)=\mu_j(N)+\lambda_{\max}(N_{\mathcal{E}})\ \forall j.
\]
Hence $\lambda_i(M)+\lambda_j(N)$ for all $i,j$ are distinct by assumption. Therefore the multiset
\[
\Lambda=\{\lambda_i(M)+\lambda_j(N);1\le i\le m,1\le j\le n\}
\]
coincides with the slope multiset of $\Sol_{\bullet}(M\otimes_{\mathcal{R}^{\bd}}N)$ by Proposition \ref{prop:weak tensor}. Since
\[
\Lambda=\{\mu_i(M)+\mu_j(N)+\lambda_{\max}(M_{\mathcal{E}})+\lambda_{\max}(N_{\mathcal{E}});1\le i\le m,1\le j\le n\},
\]
$\Lambda$ also coincides with the slope multiset of $S_{\bullet-(\lambda_{\max}(M_{\mathcal{E}})+\lambda_{\max}(N_{\mathcal{E}}))}(\Sol(M\otimes_{\mathcal{R}^{\bd}}N))$. By Proposition \ref{prop:main} and Lemma \ref{lem:ten8}, $\Sol_{\bullet}(M\otimes_{\mathcal{R}^{\bd}}N)=S_{\bullet-(\lambda_{\max}(M_{\mathcal{E}})+\lambda_{\max}(N_{\mathcal{E}}))}(\Sol(M\otimes_{\mathcal{R}^{\bd}}N))$. By Converse theorem over $\mathcal{R}^{\bd}$ (Lemma \ref{lem:conv bd}), $M\otimes_{\mathcal{R}^{\bd}}N$ is PBQ.
\end{proof}

If the tensor product of two $(\varphi,\nabla)$-modules is PBQ as in Proposition \ref{prop:product PBQ}, then the log-growth filtration on the tensor product is easily computed as in the case of Frobenius slope filtration (Lemma \ref{lem:slope5} (II)-(iii)).

\begin{lem}
Let $M,N$ be $(\varphi,\nabla)$-modules over $\mathcal{R}^{\bd}$ solvable in $\mathcal{R}_{\log}$. Assume that $M\otimes_{\mathcal{R}^{\bd}}N$ is PBQ. Then $M,N$ are PBQ. Moreover, there exist canonical isomorphisms
\[
\cap_{\lambda+\mu=\delta}(V(M)^{\lambda}\otimes_{K}V(N)+V(M)\otimes_{K}V(N)^{\mu})\cong V(M\otimes_{\mathcal{R}^{\bd}}N)^{\delta},
\]
\[
\sum_{\lambda+\mu=\delta}\Sol_{\lambda}(M)\otimes_K\Sol_{\mu}(N)\cong \Sol_{\delta}(M\otimes_{\mathcal{R}^{\bd}}N)
\]
for all $\delta\in\mathbb{R}$.

A similar assertion holds for $M,N$ (log-)$(\varphi,\nabla)$-modules over $K[\![t]\!]_0$ or $\mathcal{E}$.
\end{lem}
\begin{proof}
To prove the first assertion, we may consider only in the case of $\mathcal{E}$. By Lemma \ref{lem:product solution}, there exists an injection $\Sol_0(M)\otimes_{\mathcal{E}}\Sol_0(N)\hookrightarrow \Sol_0(M\otimes_{\mathcal{E}}N)$ of $\varphi$-modules over $\mathcal{E}$. Since $\Sol_0(M\otimes_{\mathcal{E}}N)$ is pure by assumption, so are $\Sol_0(M)$ and $\Sol_0(N)$, which implies the assertion.

We prove the second assertion: by Remark \ref{rem:principle}, it suffices to prove the second isomorphism in the case of $\mathcal{R}^{\bd}$. We obtain the assertion by
\begin{align*}
\Sol_{\delta}(M\otimes_{\mathcal{R}^{\bd}}N)=&S_{\delta-(\lambda_{\max}(M)+\lambda_{\max}(N))}(\Sol(M\otimes_{\mathcal{R}^{\bd}}N))\\
\cong&\sum_{\lambda+\mu=\delta}S_{\lambda-\lambda_{\max}(M)}(\Sol(M))\otimes_{K}S_{\mu-\lambda_{\max}(N)}(\Sol(N))\\
=&\sum_{\lambda+\mu=\delta}\Sol_{\lambda}(M)\otimes_{K}\Sol_{\mu}(N),
\end{align*}
where the equalities follow from Theorem \ref{conj:bd} (ii).
\end{proof}


\section{Generating theorem}\label{sec:gen}

As we have seen, particularly in \S \ref{subsec:Dwo}, the PBQ filtration is a useful tool to handle an arbitrary $(\varphi,\nabla)$-modules. In this section, we prove Generating theorem, which is another tool to describe $M$ using PBQ $(\varphi,\nabla)$-modules in a very na\"ive way.

\begin{thm}[Generating theorem]\label{thm:generating}
Assume that $k$ is perfect. Let $M$ be a $(\varphi,\nabla)$-module over $\mathcal{E}$. Then there exist PBQ $(\varphi,\nabla)$-submodules $M_1,\dots,M_r$ of $M$ such that $\sum_{i=1}^rM_i=M$.
\end{thm}
\begin{proof}
We consider the following condition on $M$.

($C_M$): There exist $(\varphi,\nabla)$-submodules $M_1,\dots,M_r$ of $M$ such that $M_1,\dots,M_r\neq M$ and $M_1+\dots+M_r=M$.

\noindent Note that if $(C_Q)$ is true for a quotient $Q$ of $M$, then $(C_M)$ is true. In particular, if $M$ is not PBQ, then $(C_M)$ is true since $(C_{M/M^0})$ is true by Splitting theorem (Theorem \ref{thm:split}).

We prove the assertion by induction on $\dim_{\mathcal{E}}M$. When $M$ is PBQ, we set $n=1$ and $M_1=M$. When $M$ is not PBQ, we obtain the assertion by applying the induction hypothesis to the $M_i$'s in $(C_M)$.
\end{proof}

Note that in the above proof, we use Splitting theorem in an essential way, as is the case of the PBQ filtration (see the proof of \cite[Proposition 5.4]{CT2}).

\begin{ex}
Let $M=M_{\mu,\delta}$ be as in Example 6.1.4. Then $M$ is not PBQ (Example 6.2.4). By the pushout diagram in Example 6.1.4, $M$ contains copies of $M_{\mu},M_{\delta}$, which are PBQ by Example 6.2.3 and generate $M$.
\end{ex}


\section{The slopes of log-growth Newton polygon}\label{sec:bound}

In this section, we study the slopes of the log-growth Newton polygons of $(\varphi,\nabla)$-modules $M$ over $\mathcal{E}$. Particularly, we prove the dual invariance of the maximum slope $b^{\nabla}$ (Theorem \ref{thm:inv}) by using Generating theorem (Theorem \ref{thm:generating}), and an analogue of Drinfeld-Kedlaya theorem for indecomposable convergent $F$-isocrystals (Theorem \ref{thm:Gri}).

We first study the ``generic bound'' $b^{\nabla}$ for a $(\varphi,\nabla)$-module over $\mathcal{E}$. The following two lemmas are consequences of Lemmas \ref{lem:lgE2}, \ref{lem:lgE3}, and Corollary \ref{cor:right continuous}.

\begin{lem}\label{lem:bound2}
Let $f:M\to N$ be a morphism of $(\varphi,\nabla)$-modules over $\mathcal{E}$.
\begin{enumerate}
\item If $f$ is injective, then $b^{\nabla}(M)\le b^{\nabla}(N)$.
\item If $f$ is surjective, then $b^{\nabla}(M)\ge b^{\nabla}(N)$.
\end{enumerate}
\end{lem}

\begin{lem}\label{lem:sum}
For $(\varphi,\nabla)$-modules $M_1,\dots,M_n$ over $\mathcal{E}$, we have $b^{\nabla}(\oplus_{i=1}^nM_i)=\max_{1\le i\le n}b^{\nabla}(M_i)$.
\end{lem}

\begin{lem}\label{lem:sum2}
Let $M$ be a $(\varphi,\nabla)$-module over $\mathcal{E}$. If $M_1,\dots,M_n$ are $(\varphi,\nabla)$-submodules of $M$ such that $\sum_{i=1}^nM_i=M$, then
\[
b^{\nabla}(M)=\max_{1\le i\le n}b^{\nabla}(M_i),
\]
\[
b^{\nabla}(M\spcheck)=\max_{1\le i\le n}b^{\nabla}(M_i\spcheck).
\]
\end{lem}
\begin{proof}
By Lemma \ref{lem:bound2} (i), $b^{\nabla}(M_i)\le b^{\nabla}(M)$. By applying Lemmas \ref{lem:bound2} (ii) and \ref{lem:sum} to the canonical surjection $\oplus_{i=1}^nM_i\to M$, we have $b^{\nabla}(M)\le b^{\nabla}(\oplus_{i=1}^nM_i)=\max_{1\le i \le n}b^{\nabla}(M_i)$, which implies the first equality.

Similarly, regarding $M_i\spcheck$ as a quotient of $M\spcheck$, $b^{\nabla}(M_i\spcheck)\le b^{\nabla}(M\spcheck)$ by Lemma \ref{lem:bound2} (ii). By applying Lemmas \ref{lem:bound2} (i) and \ref{lem:sum} to the canonical injection $M\spcheck\hookrightarrow\oplus_{i=1}^nM_i\spcheck$, we have $b^{\nabla}(M\spcheck)\le b^{\nabla}(\oplus_{i=1}^nM_i\spcheck)=\max_{1\le i\le n}b^{\nabla}(M_i\spcheck)$, which implies the second equality.
\end{proof}

\begin{lem}\label{lem:dual}
Let $M$ be a $(\varphi,\nabla)$-module over $\mathcal{E}$.
\begin{enumerate}
\item If $M$ is PBQ, then $b^{\nabla}(M)=\lambda_{\max}(M)-\lambda_{\min}(M)$.
\item If $M$ and $M\spcheck$ are PBQ, then $b^{\nabla}(M)=b^{\nabla}(M\spcheck)$.
\end{enumerate}
\end{lem}
\begin{proof}
\begin{enumerate}
\item By Theorem \ref{conj:CTgen} (ii) and Corollary \ref{cor:right continuous},
\begin{align*}
b^{\nabla}(M)=&\min\{\lambda;M^{\lambda}=0\}=\min\{\lambda;S_{\lambda-\lambda_{\max}(M)}(M\spcheck)=M\spcheck\}\\
=&\lambda_{\max}(M)+\lambda_{\max}(M\spcheck)=\lambda_{\max}(M)-\lambda_{\min}(M).
\end{align*}
\item It follows from (i) and $\lambda_{\max}(M)-\lambda_{\min}(M)=\lambda_{\max}(M\spcheck)-\lambda_{\min}(M\spcheck)$.
\end{enumerate}
\end{proof}

The following theorem is highly non-trivial, and we make an essential use of a Frobenius structure in the proof.

\begin{thm}[{the dual invariance of $b^{\nabla}$}]\label{thm:inv}
For any $(\varphi,\nabla)$-module $M$ over $\mathcal{E}$,
\[
b^{\nabla}(M)=b^{\nabla}(M\spcheck).
\]
\end{thm}
\begin{proof}
Note that the assertion follows from Lemma \ref{lem:dual} (ii) when both $M$ and $M\spcheck$ are PBQ (including the case $\dim_{\mathcal{E}}M=1$). We prove the assertion by induction on $\dim_{\mathcal{E}}M$. Assume that the assertion holds for all $M$ of dimension $<m$. We may assume that either $M$ or $M\spcheck$ is not PBQ. By symmetry, we may assume that $M$ is not PBQ. By Generating theorem (Theorem \ref{thm:generating}), there exist PBQ $(\varphi,\nabla)$-submodules $M_1,\dots,M_n$ of $M$ such that $\sum_{i=1}^nM_i=M$. Note that $M_i\neq M$ for all $i$ since $M$ is not PBQ. Since $b^{\nabla}(M_i)=b^{\nabla}(M_i\spcheck)$ for all $i$ by the induction hypothesis, we obtain the assertion by Lemma \ref{lem:sum2}.
\end{proof}

\begin{rem}
As for an analogue over $K[\![t]\!]_0$, only a na\"ive inequality is known: for an arbitrary $\nabla$-module over $K[\![t]\!]_0$ of rank $n$, we have $b^{\nabla}(M\spcheck)\le (n-1)b^{\nabla}(M)$ (\cite[Proposition 1.4]{CT}).
\end{rem}

We estimate $b^{\nabla}$ under a certain extension.

\begin{lem}\label{lem:Dw}
Let
\[
0\to M'\to M\to M''\to 0
\]
be an exact sequence of $(\varphi,\nabla)$-modules over $\mathcal{E}$. If either $M'$ or $M''$ is bounded, then
\[
b^{\nabla}(M)\le b^{\nabla}(M')+b^{\nabla}(M'')+1.
\]
\end{lem}
\begin{proof}
By taking dual and using Theorem~\ref{thm:inv}, we may assume that $M''$ is bounded. Let $\{e_1,\dots,e_{r+s}\}$ be a basis of $\tau^*M$ such that $\{e_1,\dots,e_r\}$ is a basis of $\tau^*M'$ and $\{e_{r+1},\dots,e_{r+s}\}$ modulo $\tau^*M'$ is a horizontal basis of $\tau^*M''$. We have only to prove that for any $f\in \Sol(M)$, we have $f(e_i)\in\mathcal{E}[\![X-t]\!]_{b^{\nabla}(M')+1}$ for all $i$. Since $f|_{\tau^*M'}\in \Sol(M')$ is of log-growth $b^{\nabla}(M')$, we have $f(e_1),\dots,f(e_r)\in \mathcal{E}[\![X-t]\!]_{b^{\nabla}(M')}$. Write
\[
D(e_{r+1},\dots,e_{r+s})=(e_1,\dots,e_r)A,\ A\in \M_{rs}(\mathcal{E}[\![X-t]\!]_0).
\]
Then,
\begin{align*}
D(f(e_{r+1}),\dots,f(e_{r+s}))=&f(D(e_{r+1},\dots,e_{r+s}))=f((e_1,\dots,e_r)A)\\
=&(f(e_1),\dots,f(e_r))A\in (\mathcal{E}[\![X-t]\!]_{b^{\nabla}(M')})^r
\end{align*}
by Lemma \ref{lem:property Robba} (i). By Lemma \ref{lem:property Robba} (iv),
\[
f(e_{r+1}),\dots,f(e_{r+s})\in\mathcal{E}[\![X-t]\!]_{b^{\nabla}(M')+1},
\]
which implies the assertion.
\end{proof}

\begin{rem}
\begin{enumerate}
\item The above proof in the case that $M''$ is bounded is based on the proof of \cite[Theorem 2]{Dw}.
\item Without the boundedness assumptions on $M'$ or $M''$, we have a similar inequality with an extra factor
\[
b^{\nabla}(M)\le b^{\nabla}(M')+b^{\nabla}(M'')+\min\{b^{\nabla}(M'),b^{\nabla}(M'')\}+1.
\]
In fact, we may assume $b^{\nabla}(M')\le b^{\nabla}(M'')$ by taking dual and using Theorem~\ref{thm:inv}. We consider the short exact sequence $0\to M'\otimes_{\mathcal{E}} (M'')\spcheck\to Q\to\mathcal{E}\to 0$ associated to the original short exact sequence under the canonical isomorphism $\Ext^1(M'',M')\cong \Ext^1(\mathcal{E},M'\otimes_{\mathcal{E}} (M'')\spcheck)$ of Yoneda extension groups in the category of $(\varphi,\nabla)$-modules over $\mathcal{E}$ (\cite[Lemma 5.3.3]{pde}). Then we have
\[
b^{\nabla}(Q)\le b^{\nabla}(M'\otimes_{\mathcal{E}} (M'')\spcheck)+1=b^{\nabla}(M')+b^{\nabla}((M'')\spcheck)+1=b^{\nabla}(M')+b^{\nabla}(M'')+1,
\]
where the inequality follows from Lemma \ref{lem:Dw}, and the first and second equalities follow from Corollary \ref{cor:tensor bound}, Theorem \ref{thm:inv} respectively. Since $M$ is isomorphic to a quotient of $Q\otimes_{\mathcal{E}} M''$, we have
\[
b^{\nabla}(M)\le b^{\nabla}(Q\otimes_{\mathcal{E}} M'')=b^{\nabla}(Q)+b^{\nabla}(M'')\le 2b^{\nabla}(M')+b^{\nabla}(M'')+1,
\]
where the first inequality follows from Lemma \ref{lem:bound2} (ii), and the equality follows from Corollary \ref{cor:tensor bound}.
\item We have the following analogous assertion with a similar proof as above: if $0\to M'\to M\to M''\to 0$ is an exact sequence of $(\varphi,\nabla)$-modules over $\mathcal{R}^{\bd}$ solvable in $\mathcal{R}_{\log}$ such that $M''$ is a trivial $\nabla$-module over $\mathcal{R}^{\bd}$, then $b^{\nabla}(M)\le b^{\nabla}(M')+1$.
\end{enumerate}
\end{rem}

The statement of the following theorem is inspired by that of Drinfeld-Kedlaya theorem on Frobenius Newton polygons for indecomposable convergent $F$-isocrystals (\cite[Theorem 1.1.5]{DK}). However we do not know that (some special case of) Drinfeld-Kedlaya theorem implies our theorem, or vice versa.

\begin{thm}\label{thm:Gri}
Let $M$ be a $(\varphi,\nabla)$-module over $\mathcal{E}$ of rank $n$. Then the slope multiset $\{\lambda_i(M);1\le i\le n\}$ of the log-growth Newton polygon of $M$ satisfies
\[
\lambda_{i+1}(M)-\lambda_i(M)\le 1\text{ for }i=1,\dots,n-1.
\]
\end{thm}
\begin{proof}
Let $\lambda_i$ denote $\lambda_i(M)$ for simplicity. We may assume $\lambda_{i}<\lambda_{i+1}$. Consider an exact sequence of $(\varphi,\nabla)$-modules over $\mathcal{E}$
\begin{equation}\label{eq:Gri1}
0\to (M^{\lambda_i}/M^{\lambda_{i+1}})/(M^{\lambda_i}/M^{\lambda_{i+1}})^0\to (M/M^{\lambda_{i+1}})/(M^{\lambda_i}/M^{\lambda_{i+1}})^0\to M/M^{\lambda_{i}}\to 0.
\end{equation}
Write the middle term as $M/M'$. Then we have
\[
b^{\nabla}(M/M')\le b^{\nabla}(M/M^{\lambda_i})+1=\lambda_i+1,
\] 
where the inequality is obtained by applying Lemma \ref{lem:Dw} to (\ref{eq:Gri1}), and the equality is obtained by Corollary \ref{cor:right continuous} and $(M/M^{\lambda_i})^{\lambda}=(M^{\lambda}+M^{\lambda_i})/M^{\lambda_i}$ (Lemma \ref{lem:lgE2} (ii)). 

Since $M^{\lambda_i}/M^{\lambda_{i+1}}\neq 0$, we also have $(M^{\lambda_i}/M^{\lambda_{i+1}})/(M^{\lambda_i}/M^{\lambda_{i+1}})^0\neq 0$ by Lemma \ref{lem:lgE1} (ii). Hence $M'\subsetneq M^{\lambda_i}$ by the definition of $M'$. Since $(M/M')^{b^{\nabla}(M/M')}=(M^{b^{\nabla}(M/M')}+M')/M'=0$ by Lemma \ref{lem:lgE2} (ii) and Corollary \ref{cor:right continuous}, we have $M^{b^{\nabla}(M/M')}\subset M'$. Hence $M^{b^{\nabla}(M/M')}\subsetneq M^{\lambda_i}$. By Lemma \ref{lem:right continuous}, we have $\lambda_{i+1}=\min\{\lambda;M^{\lambda}\subsetneq M^{\lambda_i}\}\le b^{\nabla}(M/M')$ as desired.
\end{proof}

\begin{rem}
Dwork proves the following result (see \cite[p. 368]{Dw}): let $M$ be a $\nabla$-module over $\mathcal{E}$ of rank $n$ such that $\tau^*M$ is solvable in $\mathcal{E}\{X-t\}$. Let $P(n)$ denote the lower convex polygon with left endpoint fixed at $(0,0)$, whose slope multiset is $\{0,1,\dots,n-1\}$ with multiplicity. Then the log-growth Newton polygon of $M$ (with left endpoint fixed at $(0,0)$) is bounded from above by $P(n)$. Theorem \ref{thm:Gri} is a refinement of Dwork's result under the existence of Frobenius structures.
\end{rem}

With notation as in the remark above, one may ask when $\NP(M)$ coincides with $P(n)$. An answer is given as follows.

\begin{cor}
Let $M$ be a $(\varphi,\nabla)$-module over $\mathcal{E}$ of rank $n\ge 1$. If $b^{\nabla}(M)=n-1$, then
\[
\lambda_i(M)=i-1\text{ for }i=1,\dots,n.
\]
\end{cor}


\section{Appendix: List of notation}
The following is a list of notation in order defined.

Part I
\begin{itemize}
\item[\S \ref{subsec:Robba1}] $\mathcal{R}^r,\mathcal{R},|\cdot|_r,\mathcal{R}^{\Int},\mathcal{R}^{\bd},\mathcal{E},\varphi$
\item[\S \ref{subsec:Robba2}] $k((t^{\mathbb{Q}})),\tilde{\mathcal{R}}^r,\tilde{\mathcal{R}},\tilde{\mathcal{R}}^{\bd},\tilde{\mathcal{R}}^{\Int},\tilde{\mathcal{E}},\psi$
\item[\S \ref{subsec:Robba3}] $\mathcal{R}_{\log},\tilde{\mathcal{R}}_{\log}$
\item[\S \ref{subsec:lg1}] $\Fil_{\bullet}\tilde{\mathcal{R}}$
\item[\S \ref{subsec:lg2}] $\Fil_{\bullet}\mathcal{R},K[\![t]\!]_{\bullet}$
\item[\S \ref{subsec:lg3}] $\Fil_{\bullet}\mathcal{R}_{\log},\Fil_{\bullet}\tilde{\mathcal{R}}_{\log},K\{t\},K\{t\}_{\log},\Fil_{\bullet}K\{t\}_{\log}$
\item[\S \ref{subsec:phinabla1}] $e_c,R(c),M(c)$
\item[\S \ref{subsec:phinabla3}] $\nabla_{\log},\tau$
\item[\S \ref{subsec:uni}] $\mathbf{V}(\cdot)$
\item[\S \ref{sec:slope}] $|\cdot|_{\SP},S_{\bullet}(\cdot)$
\item[\S \ref{subsec:logfil1}] $V(\cdot),\Sol(\cdot),\Sol_{\bullet}(\cdot),V(\cdot)^{\bullet}$
\item[\S \ref{subsec:logfil3}] $(\cdot)^{\bullet}$
\item[\S \ref{subsec:statement1}] $\lambda_{\max},\mathbf{LGF}_{K[\![t]\!]_0}$
\item[\S \ref{subsec:statement2}] $\mathbf{LGF}_{\mathcal{E}}$
\item[\S \ref{sec:criterion}] $\lambda_{\max}(\cdot),\lambda_{\min}(\cdot)$
\end{itemize}

Part II
\begin{itemize}
\item[\S \ref{sec:CTgen}] $\ev$
\item[\S \ref{sec:max}] $P_{\bullet}(\cdot)$
\item[\S \ref{subsec:Dwo}] $\NP(\cdot),\lambda_i(\cdot),b^{\nabla}(\cdot),m(\cdot),\mathbf{LGF}_{\mathrm{Dw}},\mathcal{NP}(\cdot)$
\item[\S \ref{sec:ten}] $\Psi$
\end{itemize}

For $0\le \lambda_1\le \lambda_2$, we have the following diagram of spaces of functions: all the morphisms other than those denoted by $\psi$ are the natural inclusions.
\[\xymatrix{
&K[\![t]\!]_0\ar[r]\ar[d]&K[\![t]\!]_{\lambda_1}\ar[r]\ar[d]&K[\![t]\!]_{\lambda_2}\ar[r]\ar[d]&K\{t\}\ar[r]\ar[d]&K[\![t]\!]\\
\mathcal{E}\ar[d]^{\psi}&\mathcal{R}^{\bd}\ar[l]\ar[r]\ar[d]^{\psi}&\Fil_{\lambda_1}\mathcal{R}_{(\log)}\ar[r]\ar[d]^{\psi}&\Fil_{\lambda_2}\mathcal{R}_{(\log)}\ar[r]\ar[d]^{\psi}&\mathcal{R}_{(\log)}\ar[d]^{\psi}&\\
\tilde{\mathcal{E}}&\tilde{\mathcal{R}}^{\bd}\ar[l]\ar[r]&\Fil_{\lambda_1}\tilde{\mathcal{R}}_{(\log)}\ar[r]&\Fil_{\lambda_2}\tilde{\mathcal{R}}_{(\log)}\ar[r]&\tilde{\mathcal{R}}_{(\log)}&
}\]


\section*{Acknowledgement}
The author thanks Nobuo Tsuzuki for discussion and many comments on the results of this paper. This work is supported by JSPS KAKENHI Grant-in-Aid for Young Scientists (B) JP17K14161.

\end{document}